\documentclass[article,10pt]{elsarticle}
\biboptions{sort&compress}
\usepackage{amsmath, amsthm, amssymb, enumerate}
\usepackage[T1]{fontenc}
\usepackage{amssymb}
\usepackage{amsmath}
\usepackage{tikz}
\usepackage{graphicx}
\usepackage{tabularx}
\usepackage{placeins}

\usepackage{lineno}
\usepackage{bbold}
\usepackage{bm}
\usepackage{derivative}
\usepackage{empheq}
\usepackage{siunitx}
\usepackage{wasysym}
\usepackage{hyperref}
\usepackage{pgfplots}
\pgfplotsset{compat=1.10}
\usepgfplotslibrary{statistics}
\usepackage{ifthen}
\usepackage{caption}
\usepackage{float}
\usepackage{fancyhdr}
\usepackage{caption}
\usepackage{hyperref}
\usepackage{geometry}
\usepackage{color}
\definecolor{revisionfivecolor}{rgb}{0.0,0.0,0.0}
\newcommand{\revfive}[1]{\textcolor{revisionfivecolor}{#1}}
\newenvironment{revisionfive}{\color{revisionfivecolor}}{}

\usepackage{datetime}
\usepackage{geometry}
\usepackage{eurosym}%
\usepackage[normalem]{ulem}
\usepackage{pdfpages}
\usepackage{comment}
\excludecomment{confidential}

\usepackage{dsfont}
\usepackage{subcaption} 
\newtheorem{theorem}{Theorem}[section]

\newtheorem{proposition}[theorem]{Proposition}

\newtheorem{definition}{Definition}[section]
\theoremstyle{definition}
\newtheorem{remark}{Remark}

\usetikzlibrary{matrix,calc}

\def\tilde{\widetilde}

\numberwithin{equation}{section}

\newcounter{eps}

\newcounter{ceps}

\begin{document}
\begin{frontmatter}
\title{A Temperature-Coupled Cahn–Hilliard–Stokes–Heat Model for Thermally Driven Phase Separation}

\author[1]{Maria Deliyianni}
\author[2]{Boris Muha}
\author[2,3,4]{Andrej Novak}

\affiliation[1]{organization={Department of Mathematics,
University of California, Riverside},
             addressline={Riverside, California},
             country={USA}}

 \affiliation[2]{organization={University of Zagreb Faculty of Science, Department of Mathematics},
             addressline={Bijenicka cesta 30},
             city={Zagreb},
             postcode={10000},
             country={Croatia}}

\affiliation[3]{organization={Faculty of Mathematics, University of Vienna},
             addressline={Oskar-Morgenstern-Platz 1},
             city={Vienna},
             postcode={1090},
             country={Austria}}
\affiliation[4]{organization={Dubrava University Hospital},
             addressline={Avenija Gojka Šuška 6},
             city={Zagreb},
             postcode={10000},
             country={Croatia}}

\renewcommand{\thefootnote}{\arabic{footnote}}
\vspace{-1cm}

\begin{abstract}
We study a diffuse–interface model for thermally driven phase separation in viscous incompressible mixtures. The system couples a convective Cahn–Hilliard equation for the order parameter with a Stokes subsystem for the velocity–pressure field and a heat equation for the temperature.
Temperature enters the bulk free energy through a Landau-type coefficient, while the phase field feeds back on the flow through concentration-dependent density and viscosity, \revfive{yielding a phenomenological temperature-coupled Cahn--Hilliard--Stokes--heat system. 
The model serves as a proxy for temperature-triggered 
condensation-like phase separation and does not include humidity, latent heat, 
vapor pressure, or capillary forcing; these effects are absorbed into the choice 
of the threshold temperature $\Theta_S$.}
We motivate the chemical potential by a temperature-dependent Landau free energy and derive a priori estimates for the regularized subproblems. On the analytical side, we prove local-in-time existence of weak solutions for a regularized auxiliary model of the coupled system. \revfive{This regularized formulation serves as an auxiliary analytical model for establishing local existence.}
On the numerical side, we use a first-order sequential finite-element discretization of a simplified quasi-static formulation. The heat equation is advanced by an implicit diffusion step, the variable-coefficient Stokes problem by a Taylor--Hood discretization, and the Cahn--Hilliard bulk derivative is evaluated at the previous time level, so every algebraic subproblem is linear.
\revfive{An isothermal diffusive test confirms mass conservation to roundoff and exhibits monotone discrete-energy decay for the tested parameter set. Time-step and mesh-refinement studies show first-order temporal and approximately second-order spatial behavior. The remaining computations are qualitative, parameter-specific illustrations, and no global discrete energy law is claimed for the non-isothermal sequential scheme.}
\end{abstract}




\end{frontmatter}


\section{Introduction} %
\label{sec:Intro}

Water condensation inside battery housings is a practically relevant issue in electric-vehicle systems. Moisture entering through pressure-compensation or ventilation elements may condense on colder internal surfaces, contributing to corrosion, insulation degradation, and accelerated ageing \cite{Kim2019}. Motivated by this application, we study an idealized diffuse-interface model for temperature-triggered phase transition in battery-inspired geometries, and provide a mathematically tractable phase-field framework for qualitative condensation-like behaviour driven by temperature gradients. Since experimental testing over wide ranges of geometry and operating conditions is costly, numerical simulation offers a useful complementary tool.

\revfive{
We emphasize however that the present model is phenomenological: it does not 
include humidity, latent heat release, a vapor-pressure relation, or 
phase-change mass sources. Rather, these effects are represented 
implicitly through the threshold temperature $\Theta_S$ in the 
Landau potential, which should be interpreted as a phenomenological 
proxy for a temperature-triggered condensation-like transition.}

\subsection*{The Cahn--Hilliard Equation}
The Cahn--Hilliard equation (CHE) is a classical diffuse-interface model for
phase separation and coarsening in binary mixtures~\cite{Cah58, Fan17, Novi08, Zhao20}.
It describes the evolution of a conserved order parameter
$c(\mathbf{x},t)$, which may be interpreted as a concentration difference, volume fraction, or other rescaled composition variable depending on the application.

The evolution is governed by
\begin{equation} \label{che}
    \frac{\partial c}{\partial t} + (\mathbf{u}\cdot\nabla)c
    = \nabla\cdot\bigl(M\nabla\mu\bigr),
\end{equation}
where $M>0$ is the mobility and $\mathbf{u}$ is the prescribed velocity
field. The chemical potential \(\mu\) is the variational derivative of the
free-energy functional
\[
F[c] = \int_\Omega \left(f(c) + \frac{\epsilon^2}{2}|\nabla c|^2\right)\,dx,
\]
namely
\[
\mu = \frac{\delta F}{\delta c} = f'(c) - \epsilon^2 \Delta c.
\]
Here $f(c)$ is the bulk free-energy density, while the gradient term penalizes
spatial inhomogeneity and gives rise to a diffuse interface of finite thickness.

A key aspect of the CHE is the concept of a diffuse interface: the interface between two phases is not sharp but has a finite thickness over which $c$ changes gradually. This eliminates the need for explicit interface tracking, making the CHE advantageous for simulating microstructural evolution in complex geometries~\cite{Cah58}. In regions where $c$ transitions from one equilibrium composition $c_\alpha$ to another $c_\beta$, the interface is represented smoothly by a narrow transition layer. In contrast, classical sharp-interface models (e.g., Stefan-type formulations or Navier--Stokes systems with an explicit surface-tension boundary condition) treat the interface as a lower-dimensional surface whose motion must be tracked explicitly in time, which becomes cumbersome in the presence of topological changes such as pinch-off, coalescence, or nucleation.

\subsection*{The Role of the Bulk Free-Energy Density}
An essential ingredient in the Cahn--Hilliard equation is the bulk free-energy density $f(c)$. Its derivative $f'(c)$ forms the homogeneous part of the
chemical potential, while the full chemical potential is
$\mu = f'(c)-\epsilon^2\Delta c$. Cahn and Hilliard~\cite{Cah58} originally suggested a singular logarithmic form for this potential:
\begin{equation}\label{non-smooth}
    f_{\text{log}}(c) = \frac{\theta}{2} \left( (1 + c) \ln(1 + c) + (1 - c) \ln(1 - c) \right) + \frac{\theta_c}{2} (1 - c^2),
\end{equation}
where \( 0 < \theta < \theta_c \) are constants (see~\cite{Cope92, Cher15}). 

This potential can be challenging to treat numerically due to its singularities near $c = \pm 1$ and is therefore commonly approximated by a polynomial of degree four, known as the double-well potential:
\begin{equation}
    f(c) = \frac{\alpha}{4}(c^2 - 1)^2,
    \quad\quad \text{or equivalently} \quad\quad
    f(c) = \frac{\alpha}{4} c^4 - \frac{\alpha}{2} c^2 + \text{const},
\end{equation}
where $\alpha$ is a positive constant. The double-well potential has been used extensively in models of spinodal
decomposition and two-phase flow~\cite{applications,diff}, and has also appeared in Cahn--Hilliard-based image inpainting models~\cite{Bert07,Bosc14,Brk20,Garc18,Nov22,Nov24}. As illustrated in Figure~\ref{fig1}, it provides a good approximation of the logarithmic potential away from the singular endpoints.

\begin{figure}[!ht]
  \centering
    \includegraphics[width=0.5\textwidth]{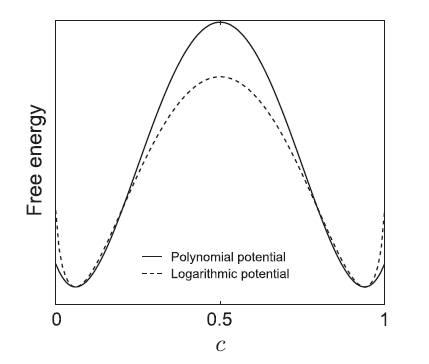}
  \caption{Comparison of the logarithmic potential and the polynomial double-well potential.}
  \label{fig1}
\end{figure}

Essentially, $f(c)$ drives the values of the conserved quantity $c$ toward one of its wells, i.e., toward energetically stable states. The double-well potential implies that $c = \pm 1$ (or $c = 0$ and $c = 1$, depending on scaling) are energetically favorable, thereby promoting phase separation in the system. This potential structure is key for modeling systems where the field $c$ oscillates between two distinct phases, such as in binary alloys, mixtures, or other phase-separating media.

\subsection*{Alternative Potentials}
In recent years, variants of the Cahn--Hilliard equation based on non-smooth potentials have been studied, particularly in the context of image inpainting problems~\cite{Bosc14, Brk20, Garc18}. One prominent example is the double-obstacle potential, defined as
\begin{equation}\label{non-smooth-1}
    f_{\text{do}}(c) = \psi(c) + \frac{1}{2}(1 - c^2),
     \quad \text{where} \quad \psi(c) = I_{[-1,1]}(c) = 
    \begin{cases} 
        0, & \text{if } c \in [-1,1],\\   
        +\infty, & \text{otherwise},
    \end{cases}
\end{equation}
and $I_A$ is the indicator function of the set $A$. In~\cite{Bosc14}, the authors developed an efficient finite-element solver for image inpainting problems and demonstrated that the non-smooth potential can produce superior reconstruction results compared to the smooth double-well potential.

From a modeling perspective, there has been considerable effort to identify physically meaningful forms of the free-energy density $f(c)$, motivated in part by the observation that solutions $c$ of the Cahn--Hilliard equation with the double-well potential may not remain bounded within $[-1,1]$. In many applications such unbounded solutions violate physical constraints (e.g., concentration or volume fraction bounds). The singularity of $f_{\text{log}}'(c)$, in contrast to the polynomial case, enforces $c \in (-1, 1)$ at the continuous level, at the price of numerical stiffness.

From a computational standpoint, the violation of bounds can be mitigated by thresholding, where values of $c > 1$ are reset to $c = 1$, and values of $c < -1$ are reset to $c = -1$. This approach, employed for instance by Cherfils et al.~\cite{Cher15d}, incorporates thresholding in the final steps of the numerical scheme to enforce physical bounds. At the analytical level, a well-posedness result for a more general form of \eqref{che} with degenerate diffusion mobility, where the diffusion coefficient depends on $c$, was established in \cite{Dai16}. These frameworks provide a rigorous background but typically focus on isothermal settings and do not explicitly resolve temperature-driven phase transitions such as condensation.

\subsection*{Mathematical Modeling of Phase Transitions}
Phase-field models describe phase transitions using an order parameter that varies continuously across interfaces. They are based on energy functionals (e.g., Landau-type free energies) with minima corresponding to stable phases (e.g., liquid and vapor). A classical approach postulates a Helmholtz free energy density with a double-well form for the two phases, plus a gradient energy term to penalize interfaces. The Cahn--Hilliard equation arises from this framework as the $H^{-1}$ gradient flow of the free energy, governing the evolution of the phase field by free-energy dissipation \cite{Granasy2000}.

Fabrizio et al.~\cite{Fabrizio2016} formulated a phase-field model for liquid--vapor transitions that enforces thermodynamic consistency. Their model introduces a temperature-dependent Helmholtz free energy and derives the phase-field evolution from the second law of thermodynamics, ensuring energy conservation and proper entropy production. 

Jamet et al.~\cite{Jamet2001} introduced a diffuse-interface formulation for liquid--vapor flows with phase change using a van der Waals-type gradient theory (similar in spirit to the Cahn--Hilliard framework) to model interfacial thermodynamics. In their approach, deviations from the equilibrium chemical potential drive evaporation/condensation, while the hydrodynamics are governed by a stress tensor that includes capillary effects. More recently, Wang et al.~\cite{Wang2021} developed a phase-field method for boiling heat transfer by coupling the Cahn--Hilliard equation with the Navier--Stokes and energy equations. Their method introduces a source term in the energy equation to account for latent heat and a velocity divergence correction to handle volumetric expansion due to phase change. While their focus was on boiling, the same framework can be adapted to simulate droplet formation under subcooled (condensing) conditions. The inclusion of a thermal solver ensures that temperature-dependent phase transitions are directly captured: when the local temperature falls below the dew point, the free energy landscape shifts to favor the liquid phase, and condensation ensues.

Recent research has evolved toward more sophisticated mathematical models governed by mass, momentum, energy balance, and diffusion equations, capturing the complex interactions within the condensation of water vapor from humid air. These models address the dynamics of condensing vapor interacting with non-condensable gases under varying flow conditions, such as those in turbulent or laminar regimes within vertical ducts (see, e.g., \cite{Kri07, El12, Das21}).

Another extension is the incorporation of additional fields into the free energy to model multi-phase systems involving gas, liquid, and solid components (e.g., electrodes or separators). von Wolff et al.~\cite{vonWolff2022} extended the Cahn--Hilliard--Navier--Stokes model to include a solid phase to study salt precipitation in porous media. Although their focus was on salt deposition, the framework informs multi-phase battery models where interactions among liquid water, gas, and solids are critical.

Phase-field models have also been implemented within the lattice Boltzmann framework, which is attractive for complex boundary conditions and parallelization. Zheng et al. \cite{Zheng} demonstrated how a phase-field approach can be implemented in LBM to simulate droplet formation during condensation. By coupling a phase-field variable with a thermal lattice Boltzmann solver, their approach successfully captured dropwise condensation and boiling in complex geometries relevant to battery thermal management.  Because condensation involves thin liquid layers and sharp thermal gradients, adaptive mesh refinement or multiscale methods are often employed. Jamet et al.~\cite{Jamet2001} used a fine grid near the interface to resolve the phase-change region, while coarser grids handled larger-scale flow features, ensuring numerical stability and accuracy in regions with rapid phase transitions.

These developments demonstrate that diffuse-interface, phase-field formulations are well suited to capture temperature-dependent liquid--vapor transitions in complex geometries. However, most existing studies either focus on idealized thermodynamic settings or on generic boiling/condensation configurations, rather than on battery-relevant housings with realistic boundary conditions and coupling to variable-density Stokes flow. This gap motivates the temperature-coupled Cahn--Hilliard--Stokes--Heat model and its analysis presented in this work.

\subsection*{Mathematical Challenges and Analytical Approach}

From an analytical perspective, the coupled Cahn--Hilliard--Stokes--heat system considered here differs in several essential ways from existing diffuse-interface frameworks, most notably the isothermal Cahn--Hilliard--Navier--Stokes analysis of \cite{abels2009diffuse}. While the overall structure is similar, the introduction of temperature dependence, mixed boundary conditions, and variable material coefficients creates new difficulties that require more delicate treatment. \revfive{{\bf To overcome these difficulties, the existence theory developed in Section \ref{sec:analysis} is carried out for a suitably regularized auxiliary formulation of the coupled system.} The regularization is introduced solely for the analytical construction of a local weak solution and will be described precisely when the analytical problem is formulated.}

A first challenge arises from the temperature-dependent free energy. The homogeneous part of the chemical potential depends explicitly on the evolving temperature field, so the phase-field equation is no longer autonomous even at fixed velocity. Classical energy arguments cannot be closed at the level of the Cahn--Hilliard subsystem alone; additional control of the temperature field is required to handle the coupling. Moreover, the temperature enters the bulk potential with only limited regularity, which necessitates a suitable regularization at the analytical level.

A second difficulty stems from the absence of global $L^\infty$ bounds for the order parameter. Unlike models based on singular potentials, the polynomial free energy employed here does not prevent overshoots of the concentration field. This precludes arguments relying on uniform boundedness of $c$ and forces the analysis to be carried out in weaker topologies. \textcolor{revisionfivecolor}{
The analysis is therefore carried out in a weaker functional setting. Rather than relying on global $L^\infty$ control of the concentration, we work with Sobolev regularity compatible with the continuity properties required by the fixed-point argument and combine energy estimates, compactness, and interpolation to recover sufficient control of the coupled system.
}

A further complication is introduced by the mixed velocity and thermal boundary conditions. Boundary terms generated by transport cannot be discarded
without compatibility between the two boundary decompositions; this motivates the assumption $\Gamma_{\Theta_N}\subseteq\Gamma_D$ in the analytical problem.
The chemical potential is not normalized independently: its mean is determined
by the constitutive relation and is estimated directly. For the Stokes subproblem, the pressure normalization depends on whether the traction boundary is empty.

Finally, the system exhibits strong coupling between distinct physical mechanisms: phase separation, creeping flow with variable density and viscosity, and heat transport. We address this by analyzing each component operator separately—first the heat equation, then the Cahn--Hilliard system, and finally the Stokes problem—and proving that each is well-defined and continuous with respect to its inputs. These operators are then composed into a single map, and existence of solutions to the full coupled system follows from a standard Schauder fixed-point argument.

\subsection*{Current contribution, overview}

In this work we develop and analyse a temperature–coupled Cahn--Hilliard–Stokes–Heat model tailored to condensation of humid air in confined geometries. The present work combines a thermodynamically motivated phase-field formulation, a local existence theory for a regularized auxiliary model, and numerical simulations based on a finite-element implementation. Its main contributions are summarized as follows.

\begin{revisionfive}
We also note that the analytical and numerical parts of the paper concern two closely related models. The existence theory is carried out for a regularized auxiliary formulation introduced to construct a local weak solution of the coupled problem, while the numerical experiments consider the corresponding unregularized quasi-static model. The regularized formulation serves solely as an analytical tool, whereas the numerical simulations are intended to illustrate the qualitative behavior of the underlying physical model. The precise relationship between these two formulations is described in Sections \ref{sec:analysis} and \ref{sec:numerics}.
\end{revisionfive}

\smallskip
\noindent
\emph{(i) Temperature–dependent Landau potential for condensation.}
Starting from the classical double–well free energy, we introduce a Landau–type bulk potential
\[
    f(c,\Theta) = 2c^4 - (1-\beta(\Theta))\,c^2,
\]
where $\beta(\Theta)=\Theta/\Theta_S$ encodes the local ratio between the temperature $\Theta$ and the saturation temperature $\Theta_S$. \revfive{The parameter $\Theta_S$ is a prescribed transition temperature and $\beta(\Theta)=\Theta/\Theta_S$ is a phenomenological control parameter.
For $\Theta>\Theta_S$, the bulk potential is single-well and favours a homogeneous composition state; for $\Theta<\Theta_S$, it is double-well and
favours phase-separated composition states. The model does not contain a humidity equation or latent heat, and therefore $\Theta_S$ is not a
quantitatively predicted dew point.}

\smallskip
\noindent
\emph{(ii) Coupled CH--Stokes--Heat system with variable density and viscosity.}
On the hydrodynamic side, we embed this potential into a diffuse-interface formulation for a two-phase gas-rich/liquid-rich mixture, leading to a
Cahn--Hilliard--Stokes system with:
\begin{itemize}
    \item affine density and viscosity laws \(\rho(c)\) and \(\eta(c)\) with fixed reference phase parameters,
    \item gravitational forcing through a buoyancy term $G\,\rho(c)\,\hat{\bm{g}}$,
    \item an advected temperature field satisfying a convection–diffusion equation and entering the bulk energy only through $\beta(\Theta)$.
\end{itemize}
The subsystems are coupled through advection, concentration-dependent density and viscosity, gravity, and the temperature dependence of the bulk potential.
The momentum equation does not contain a Korteweg/capillary force, and the heat equation contains no latent-heat source. 

\smallskip
\noindent
\emph{(iii) Rigorous existence theory under mixed boundary conditions.}
\textcolor{revisionfivecolor}{On the analytical side, we establish local-in-time existence of weak solutions for a regularized auxiliary formulation of the coupled system in two and three space dimensions.} The proof analyzes each component operator individually:
\begin{itemize}
    \item the heat operator $\mathcal{H}:\mathbf{u}\mapsto \Theta$ is shown to be well-defined and continuous, with a maximum principle ensuring nonnegativity of the temperature;
    \item the Cahn--Hilliard operator $\mathcal{C}:(\mathbf{u},\Theta)\mapsto c$ is proven to exist uniquely and depend continuously on both velocity and temperature;
    \item the Stokes operator $\mathcal{S}:c\mapsto \mathbf{u}$ is analyzed for variable-coefficient flow with mixed boundary conditions, yielding existence, uniqueness, and continuity.
\end{itemize}
These operators are composed into a single map $\mathcal{T}(\mathbf{u}) := \mathcal{S}(\mathcal{C}(\mathbf{u}, \mathcal{H}(\mathbf{u})))$ that exploits the triangular coupling structure. The analysis proceeds by establishing well-posedness and continuity of each component operator separately and combining these results through a Schauder fixed-point argument.

\smallskip
\noindent
\revfive{\emph{(iv) Linearized finite-element discretization and focused numerical
verification.}
For the numerical studies we use a first-order sequential scheme with
continuous $P_1$ elements for the scalar fields and Taylor--Hood $P_2$--$P_1$ elements for the Stokes problem. The bulk derivative is evaluated at the previous time level, so the Cahn--Hilliard update is linear.  For the stand-alone isothermal diffusive step we report mass conservation, discrete energy evolution, time-step refinement, mesh refinement, and a limited one-factor $Pe$/$Ch$ check.}

We qualitatively compare the hydrodynamic discretization and interface transport against a lid--driven cavity benchmark for two-phase Stokes flow and then explore four classes of scenarios relevant to battery-inspired geometries: isothermal phase separation above and below the transition temperature, gravity-driven phase separation under mixed no-slip/traction-free boundary conditions, and heated-boundary/cooled-boundary transition scenarios. The simulations illustrate how the proposed temperature–sensitive potential and coupled model reproduce qualitative features such as domain formation, coarsening, settling under gravity and wall–induced phase change, providing a mathematically controlled framework for future studies of condensation in electrochemical energy systems.

\paragraph*{Structure and main results}
In Section \ref{sec:model} we specify the temperature-coupled 
Cahn--Hilliard--Stokes--heat system and the modeling assumptions 
relevant for battery housings. Section \ref{sec:analysis} contains
the functional setting, the notion of weak solution, the core a priori
estimates and compactness arguments, and proves local-in-time existence
of weak solutions for the regularized auxiliary formulation of the coupled system.
In Section \ref{sec:numerics} we present numerical simulations for the corresponding unregularized quasi-static model, illustrating condensation patterns under battery-relevant operating conditions.

\section{Temperature coupled condensation model}
\label{sec:model}
\subsection{Temperature–dependent potential for the Cahn--Hilliard equation}
\label{potencijal}

To incorporate temperature effects into the Cahn--Hilliard equation, we modify the classical double–well potential so that the free–energy density depends explicitly on temperature. We consider a binary mixture of air and water, with the order parameter $c$ representing a (rescaled) local composition. At the dimensional level we assume that the minima of the double–well potential correspond to the pure phases $c=0$ (air) and $c=1$ (water). With this choice the reference concentration is $c_R = 0.5$, and the classical quartic double-well potential can be written in the form
\[
f_{\mathrm{dw}}(c) \;=\; -a\,(c-0.5)^2 + 2a\,(c-0.5)^4,
\]
for a suitable parameter $a>0$.
For convenience, we shift the order parameter so that the pure phases are symmetric around zero:
\[
c \;\mapsto\; c - 0.5,
\]
and, by reuse of notation, we denote the shifted variable again by $c$. This maps $[0,1]$ to $[-0.5,0.5]$ and yields
\begin{equation}
    f(c) \;=\; -a\,c^2 + 2a\,c^4.
\end{equation}

To introduce temperature dependence, we let the parameter $a$ change sign at the saturation temperature $\Theta_S$:
\begin{equation}
    a(\Theta) = \alpha\,(\Theta_S - \Theta),
\end{equation}
with $\alpha>0$ a scaling constant. For nondimensionalization we set $\alpha = 1/\Theta_S$ and introduce the dimensionless temperature
\begin{equation}\label{betaDefinition}
    \beta = \frac{\Theta}{\Theta_S},
\end{equation}
so that
\begin{equation}
    a(\Theta) = 1 - \beta.
\end{equation}
Motivated by this shifted double-well and in order to keep the quartic
coefficient positive for all temperatures, we adopt the Landau ansatz
\begin{equation}
    f(c,\Theta)=2c^4-(1-\beta)c^2.
    \label{eq:modified_potential}
\end{equation}
For $\beta<1$, the minima of $f(\cdot,\Theta)$ are located at
\[
c = \pm \frac{\sqrt{1-\beta}}{2},
\]
so the preferred phase values depend on temperature. The values
$\pm 0.5$ correspond to the reference case $\beta=0$.
When $\Theta<\Theta_S$ (i.e.\ $\beta<1$), the coefficient of $c^2$ is \emph{negative} and $f(\cdot,\Theta)$ has a symmetric double–well shape with two minima away from $c=0$, promoting phase separation. For $\Theta>\Theta_S$ (i.e.\ $\beta>1$), the coefficient of $c^2$ is positive and the potential has a unique minimum at $c=0$, favouring a homogeneous phase.

\begin{figure}[hbt!]
\centering
\begin{tikzpicture}[scale=3]

    \draw[->] (-1.0,0) -- (1.0,0) node[right] {$c$};
    \draw[->] (0,-0.2) -- (0,1.4) node[above] {$f(c, \Theta)$};
    
    \draw[thick, smooth, domain=-0.9:0.9, variable=\x, blue] 
        plot ({\x}, {2*\x*\x*\x*\x - 0.8*\x*\x});
    \node[blue] at (1.15,0.5) {\small $\Theta < \Theta_S$};
    
    \draw[thick, smooth, domain=-0.9:0.9, variable=\x, red] 
        plot ({\x}, {2*\x*\x*\x*\x + 0.3*\x*\x});
    \node[red] at (1.15,1.3) {\small $\Theta > \Theta_S$};
    
    \node[below] at (0,-0.05) {\small $c = 0$};
\end{tikzpicture}
\caption{Transition of the potential $f(c,\Theta)$ from a double–well form ($\Theta<\Theta_S$, phase separation) to a single–well form ($\Theta>\Theta_S$, homogeneous phase).}
\label{fig:potential_transition}
\end{figure}
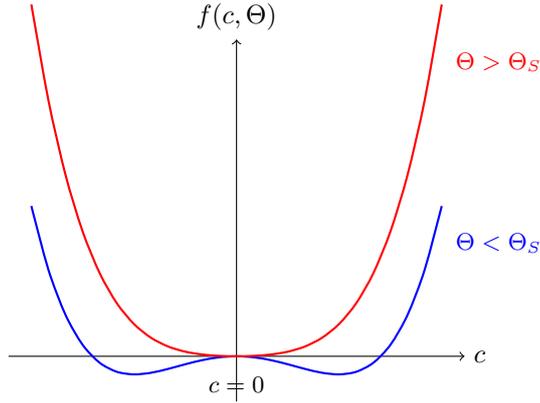

\subsection{System of equations}
\label{izvod}

Diffuse–interface models based on the Cahn--Hilliard and Navier--Stokes (or Stokes) equations are widely used to describe two–phase flows. Following the volume–averaged formulation of Ding et al.~\cite{diff}, we take the \emph{volume fraction} of one phase as the order parameter. Let
\begin{equation}
    \phi = \frac{V_1}{V_1+V_2} \in [0,1]
\end{equation}
denote the local volume fraction of phase~1 in a representative volume element. The partial mass densities in that element are
\[
\tilde{\rho}_1 = \rho_1 \phi, \qquad \tilde{\rho}_2 = \rho_2 (1-\phi),
\]
and the total density is
\begin{equation}
    \rho = \rho_1 \phi + \rho_2 (1-\phi).
\end{equation}

\noindent The mass balance for phase~1 reads
\begin{equation}
    \frac{\partial \tilde{\rho}_1}{\partial t} + \nabla \cdot \bm{m}_1 = 0,
    \label{cont_1}
\end{equation}
where $\bm{m}_1$ is the mass flux of phase~1. We decompose $\bm{m}_1$ into advective and diffusive parts:
\begin{equation}
    \bm{m}_1 = \tilde{\rho}_1 \bm{u} - \rho_1 \bm{j}_1,
\end{equation}
with $\bm{u}$ the volume–averaged velocity and $\bm{j}_1$ the (volume) diffusive flux of phase~1. Inserting this into \eqref{cont_1} gives
\begin{equation}
    \frac{\partial \tilde{\rho}_1}{\partial t} + \nabla \cdot (\tilde{\rho}_1 \bm{u}) - \nabla \cdot (\rho_1 \bm{j}_1) =0.
\end{equation}
Using $\tilde{\rho}_1 = \rho_1 \phi$ and assuming $\rho_1$ constant,
\begin{equation}
    \frac{\partial \phi}{\partial t} + \nabla \cdot (\bm{u}\phi) - \nabla \cdot \bm{j}_1 =0.
    \label{phi-1}
\end{equation}
Analogously, for phase~2 we obtain
\begin{equation}
    \frac{\partial (1-\phi)}{\partial t} + \nabla \cdot [\bm{u}(1-\phi)] - \nabla \cdot \bm{j}_2 =0.
    \label{phi-2}
\end{equation}

Adding \eqref{phi-1} and \eqref{phi-2} yields
\begin{equation}
    \nabla \cdot \bm{u} = \nabla \cdot (\bm{j}_1 + \bm{j}_2).
\end{equation}
Imposing \emph{volume conservation} of the diffusive fluxes,
\begin{equation}
    \bm{j}_1 + \bm{j}_2 = \bm{0},
\end{equation}
we obtain an incompressible velocity field,
\begin{equation}
    \nabla \cdot \bm{u} = 0.
    \label{final_incomp}
\end{equation}

Using \eqref{phi-1}, the incompressibility relation
\eqref{final_incomp}, and $\mathbf j_2=-\mathbf j_1$, we obtain
\begin{equation}
    \frac{\partial \phi}{\partial t}
    +(\mathbf u\cdot\nabla)\phi
    =\nabla\cdot\mathbf j,
    \qquad
    \mathbf j:=\mathbf j_1
    =\frac{\mathbf j_1-\mathbf j_2}{2}.
    \label{CH-phi}
\end{equation}
We close the constitutive relation by setting
$\mathbf j=M\nabla\mu$, with constant mobility $M>0$.

For consistency with the temperature–dependent potential in Section~\ref{potencijal}, we introduce the centred order parameter
\[
c = \phi - \tfrac12 \in [-0.5,0.5],
\]
and, by abuse of notation, write the Cahn--Hilliard equation in terms of $c$. The mixture density then takes the affine form
\begin{equation} 
    \rho(c) = \bigl(0.5 - c\bigr)\,\rho_2 + \bigl(0.5 + c\bigr)\,\rho_1,
    \label{DefRhoc}
\end{equation}
so that $c=-0.5$ corresponds to the pure phase with density $\rho_2$, and
$c=0.5$ to the pure phase with density $\rho_1$. Similarly, we define the viscosity by
\begin{equation}
    \eta(c) = \bigl(0.5 - c\bigr)\,\eta_2 + \bigl(0.5 + c\bigr)\,\eta_1,
    \label{DefEtac}
\end{equation}
so that $c=-0.5$ corresponds to the pure phase with viscosity $\eta_2$, and
$c=0.5$ to the pure phase with viscosity $\eta_1$.

In the model studied in this work, the reference phase parameters $\rho_1,\rho_2,\eta_1,\eta_2$ are fixed positive constants. Thus $\rho(c)$ and $\eta(c)$ depend on the order parameter \(c\) only. For the
battery-condensation interpretation, phase~1 may be viewed as a lighter gas-rich state and phase~2 as a heavier liquid-rich state, typically with $\rho_1<\rho_2$. Temperature effects are incorporated exclusively through the coefficient $\beta(\Theta)$ in the bulk free energy, and not through
additional temperature dependence in $\rho$  or $\eta$.

In the Cahn--Hilliard equation, the mobility $M$ controls the rate of diffusive relaxation. Since $M$ cannot be derived from first principles and anisotropic transport is not considered here, we take $M$ to be a positive constant, following e.g.\ \cite{Khat,diff}. More elaborate, concentration–dependent mobilities (such as $M(\phi)=\phi(1-\phi)$) would require a significantly more involved numerical treatment and are not pursued in this work.

\subsection{Nondimensionalization of the governing equations}

We use the characteristic scales
\[
L,\qquad U,\qquad \tau=\frac{L}{U},\qquad P=\frac{\eta_2 U}{L},
\]
together with the reference density \(\rho_2\), reference viscosity \(\eta_2\), transition temperature \(\Theta_S\), concentration scale \(c_0\), and chemical-potential scale \(\mu_0\). The nondimensional variables are defined by
\begin{equation}
\mathbf{x}=L\mathbf{x}^*,\qquad
t=\tau t^*,\qquad
\mathbf{u}=U\mathbf{u}^*,\qquad
c=c_0 c^*,\qquad
\mu=\mu_0 \mu^*,\qquad
p=P p^*,\qquad
\Theta=\Theta_S \Theta^*.
\end{equation}
Accordingly,
\[
\nabla=L^{-1}\nabla^*,\qquad \Delta=L^{-2}\Delta^*.
\]

After scaling temperature by $\Theta_S$, the coefficient in the Landau potential becomes the dimensionless temperature itself: $\beta(\Theta^*)=\Theta^*.$ After dropping asterisks, we therefore write
$\beta(\Theta)=\Theta.$ Thus, throughout the analytical developments in Section \ref{sec:analysis}, we write the Landau coefficient simply as $\Theta$.

For convenience, the variables and dimensionless groups are summarized in Tables~\ref{tab:nondim_variables} and~\ref{tab:nondim_parameters}.

\begin{table}[h!]
\centering
\caption{Nondimensional variables and reference scales}
\label{tab:nondim_variables}
\begin{tabular}{lll}
\hline
\textbf{Quantity} & \textbf{Scale} & \textbf{Dimensionless form} \\
\hline
Position \(\mathbf{x}\) & \(L\) & \(\mathbf{x}^*=\mathbf{x}/L\) \\
Time \(t\) & \(\tau=L/U\) & \(t^*=t/\tau\) \\
Velocity \(\mathbf{u}\) & \(U\) & \(\mathbf{u}^*=\mathbf{u}/U\) \\
Order parameter \(c\) & \(c_0\) & \(c^*=c/c_0\) \\
Chemical potential \(\mu\) & \(\mu_0\) & \(\mu^*=\mu/\mu_0\) \\
Pressure \(p\) & \(P=\eta_2 U/L\) & \(p^*=p/P\) \\
Density \(\rho\) & \(\rho_2\) & \(\rho^*=\rho/\rho_2\) \\
Viscosity \(\eta\) & \(\eta_2\) & \(\eta^*=\eta/\eta_2\) \\
Temperature \(\Theta\) & \(\Theta_S\) & \(\Theta^*=\Theta/\Theta_S\) \\
\hline
\end{tabular}
\end{table}

\begin{table}[h!]
\centering
\caption{Dimensionless groups}
\label{tab:nondim_parameters}
\begin{tabular}{lll}
\hline
\textbf{Group} & \textbf{Definition} & \textbf{Interpretation} \\
\hline
\(Pe\) & \(Pe=\dfrac{ULc_0}{M\mu_0}\) & mass advection / diffusion \\
\(Pe_{\Theta}\) & \(Pe_{\Theta}=\dfrac{UL}{\kappa}\) & heat advection / diffusion \\
\(Ch\) & \(Ch=\dfrac{\epsilon}{L}\) & interface thickness / domain length \\
\(\lambda_\rho\) & \(\lambda_\rho=\dfrac{\rho_1}{\rho_2}\) & density ratio \\
\(\lambda_\eta\) & \(\lambda_\eta=\dfrac{\eta_1}{\eta_2}\) & viscosity ratio \\
\(Re\) & \(Re=\dfrac{\rho_2 UL}{\eta_2}\) & inertial / viscous forces \\
\(G\) & \(G=\dfrac{\rho_2 gL^2}{\eta_2 U}\) & gravitational / viscous forces \\
\hline
\end{tabular}
\end{table}

\subsubsection{Nondimensionalization of the Cahn--Hilliard equation}

The dimensional Cahn--Hilliard system is
\begin{align}
\frac{\partial c}{\partial t} + \nabla\cdot(\mathbf{u}c) &= \nabla\cdot(M\nabla\mu), \label{CH_dimensional}\\
\mu &= \frac{\partial f(c,\Theta)}{\partial c} - \epsilon^2 \Delta c. \label{mu_dimensional}
\end{align}
Using \(\nabla\cdot\mathbf{u}=0\), substituting the above scalings, choosing \(c_0=1\), and selecting \(\mu_0\) so that the prefactor of \(\partial_c f\) is one, we obtain
\begin{align}
\frac{\partial c^*}{\partial t^*} + \mathbf{u}^*\cdot\nabla^* c^* &= \frac{1}{Pe}\Delta^* \mu^*, \\
\mu^* &= \frac{\partial f(c^*,\Theta^*)}{\partial c^*} - Ch^2 \Delta^* c^*.
\end{align}
For the Landau potential
\[
f(c,\Theta)=2c^4-(1-\Theta)c^2,
\]
this gives
\[
\mu^* = 8c^{*3} - 2(1-\Theta^*)c^* - Ch^2\Delta^* c^*.
\]
Dropping asterisks yields
\begin{align}
\frac{\partial c}{\partial t} + \mathbf{u}\cdot\nabla c &= \frac{1}{Pe}\Delta\mu, \label{CH_nondim_final}\\
\mu &= 8c^3 - 2(1-\Theta)c - Ch^2\Delta c. \label{mu_nondim_final}
\end{align}

After nondimensionalization, the affine density and viscosity laws become
\begin{align}
\rho(c) &= (0.5-c) + (0.5+c)\lambda_\rho, \\
\eta(c) &= (0.5-c) + (0.5+c)\lambda_\eta.
\end{align}

\subsubsection{Nondimensionalization of the Stokes equations}

The dimensional momentum balance reads
\begin{align}
\rho(c)\bigl(\partial_t\mathbf{u}+\mathbf{u}\cdot\nabla\mathbf{u}\bigr)
&= -\nabla p + \nabla\cdot\bigl(2\eta(c)\mathcal{E}(\mathbf{u})\bigr) + \rho(c)\mathbf{g}, \label{eq:NS_dimensional}\\
\nabla\cdot\mathbf{u} &= 0, \label{eq:continuity_dimensional}
\end{align}
where
\[
\mathcal{E}(\mathbf{u})=\tfrac12\bigl(\nabla\mathbf{u}+(\nabla\mathbf{u})^\top\bigr).
\]
Substituting the scales gives
\[
Re\,\rho^*(c^*)\Bigl(\frac{\partial \mathbf{u}^*}{\partial t^*}
+\mathbf{u}^*\cdot\nabla^*\mathbf{u}^*\Bigr)
= -\nabla^* p^*
+ \nabla^*\cdot\bigl(2\eta^*(c^*)\mathcal{E}(\mathbf{u}^*)\bigr)
+ G\,\rho^*(c^*)\,\hat{\mathbf g}.
\]

In the creeping-flow regime \(Re\ll 1\), inertia is neglected and the dimensionless Stokes system becomes
\begin{align}
-\nabla p + \nabla\cdot\bigl(2\eta(c)\mathcal{E}(\mathbf{u})\bigr)
+ G\,\rho(c)\,\hat{\mathbf g} &= \mathbf{0}, \label{eq:stokes_nondimensional}\\
\nabla\cdot\mathbf{u} &= 0. \label{eq:continuity_nondimensional}
\end{align}
\begin{revisionfive}
 For notational convenience in the analysis of Section \ref{sec:analysis}, we introduce the 
constant body-force vector
\begin{equation}\label{theonewithG}
\mathbf{G} := G\hat{\mathbf{g}} 
\end{equation}
\end{revisionfive}
\subsubsection{Nondimensionalization of the heat equation}
The dimensional heat equation is
\begin{equation}
\frac{\partial \Theta}{\partial t} + \mathbf{u}\cdot\nabla\Theta
= \nabla\cdot(\kappa\nabla\Theta). \label{eq:heat_dimensional_full}
\end{equation}
Using \(\Theta=\Theta_S\Theta^*\), we have
\[
\frac{\partial \Theta}{\partial t}
= \frac{\Theta_S}{\tau}\frac{\partial \Theta^*}{\partial t^*},\qquad
\mathbf{u}\cdot\nabla\Theta
= \frac{U\Theta_S}{L}\mathbf{u}^*\cdot\nabla^*\Theta^*,\qquad
\nabla\cdot(\kappa\nabla\Theta)
= \frac{\kappa\Theta_S}{L^2}\Delta^*\Theta^*.
\]
Therefore,
\[
\frac{\partial \Theta^*}{\partial t^*}
+ \mathbf{u}^*\cdot\nabla^*\Theta^*
= \frac{1}{Pe_\Theta}\Delta^*\Theta^*,
\qquad Pe_\Theta=\frac{UL}{\kappa}.
\]
Dropping asterisks gives
\begin{equation}
\frac{\partial \Theta}{\partial t} + \mathbf{u}\cdot\nabla\Theta
= \frac{1}{Pe_\Theta}\Delta\Theta. \label{eq:heat_nondim_final}
\end{equation}

\subsection{The coupled temperature-dependent phase-field system}
\label{fullsystem}

Collecting the nondimensional equations above, the regularized system used in Section~\ref{sec:analysis} is
\begin{align}
\frac{\partial c}{\partial t} + \mathbf{u}\cdot\nabla c &= \frac{1}{Pe}\Delta\mu, \label{CHND}\\
\mu &= 8c^3 - 2\bigl(1-\hat\Theta\bigr)c - Ch^2\Delta c, \label{CHmND}\\
-\alpha \partial_t\mathbf{u} - \nabla p + \nabla\cdot\bigl(2\eta(c)\mathcal{E}(\mathbf{u})\bigr)
+ \,\rho(c)\,{\mathbf G} &= \mathbf{0}, \label{NDStokes}\\
\nabla\cdot\mathbf{u} &= 0, \label{NDStokes2}\\
\frac{\partial \Theta}{\partial t} + \mathbf{u}\cdot\nabla\Theta &= \frac{1}{Pe_\Theta}\Delta\Theta. \label{HeatND}
\end{align}

The initial conditions are
\begin{equation}\label{theonewithinitialconditions}
\mathbf{u}(\mathbf{x},0)=\mathbf{u}_0(\mathbf{x}),\qquad
c(\mathbf{x},0)=c_0(\mathbf{x}),\qquad
\Theta(\mathbf{x},0)=\Theta_0(\mathbf{x}),
\qquad \mathbf{x}\in\Omega.
\end{equation}
The boundary conditions are
\begin{align}
\nabla c\cdot\mathbf{n} = \nabla\mu\cdot\mathbf{n} &= 0
&& \text{on }\partial\Omega, \\
\mathbf{u} &= \mathbf{0}
&& \text{on }\Gamma_D,\quad t>0, \label{uBCD}\\
\bigl(-p\mathbf{I}+2\eta(c)\mathcal{E}(\mathbf{u})\bigr)\mathbf{n} &= \mathbf{0}
&& \text{on }\Gamma_N:=\partial\Omega\setminus\Gamma_D,\quad t>0, \label{uBCTF}\\
\Theta &= \Theta_D
&& \text{on }\Gamma_\Theta,\quad t>0, \label{HeatDBC}\\
\nabla\Theta\cdot\mathbf{n} &= 0
&& \text{on }\Gamma_{\Theta_N}:=\partial\Omega\setminus\Gamma_\Theta,\quad t>0. \label{HeatNBC}
\end{align}
We assume that \(\Gamma_D\subset\partial\Omega\) has positive measure.

Here \(\rho(c)\) and \(\eta(c)\) are the dimensionless mixture laws introduced in \eqref{DefRhoc} and \eqref{DefEtac}, respectively and $\mathbf{G}$ is given by \eqref{theonewithG}.

\begin{revisionfive}
The auxiliary regularizations consisting of the parameter \(\alpha>0\) and the spatially mollified temperature \(\hat{\Theta}\) are introduced solely for the analytical construction of a local weak solution. Accordingly, {\bf Section \ref{sec:analysis} studies this regularized auxiliary formulation, whereas Section \ref{sec:numerics} considers the corresponding unregularized quasi-static model obtained by setting \(\alpha=0\) and \(\hat{\Theta}=\Theta\).} The numerical simulations are intended to illustrate qualitative features of the quasi-static model rather than to provide a numerical validation of the analytical results.
\end{revisionfive}

\begin{revisionfive}
\begin{remark}[Boundary and sign assumptions for the analytical problem]
\label{geometric}
Throughout Section \ref{sec:analysis}, the decompositions
\[
\partial\Omega=\overline{\Gamma_D}\cup\overline{\Gamma_N},
\qquad
\partial\Omega=\overline{\Gamma_\Theta}\cup\overline{\Gamma_{\Theta_N}},
\]
are understood up to sets of surface measure zero, with the members of each pair relatively open and disjoint. We assume
\[
|\Gamma_D|>0,
\qquad
\Gamma_{\Theta_N}\subseteq\Gamma_D.
\]
Consequently every velocity satisfying the no-slip condition on \(\Gamma_D\) also satisfies
\(\mathbf u\cdot\mathbf n=0\) on \(\Gamma_{\Theta_N}\). This condition is used both in the heat-energy identity and in the nonnegativity argument. We also assume that \(\Theta_D\ge0\) is constant in space and time and that \(\Theta_0\ge0\) almost everywhere in \(\Omega\). These additional boundary assumptions concern only the analytical problem; numerical configurations outside this class are not covered by Theorem \ref{thm:main-existence}.
\end{remark}
\end{revisionfive}


\section{Mathematical Analysis}\label{sec:analysis}

We establish local-in-time existence of weak solutions to the coupled Cahn--Hilliard--Stokes--heat system introduced in Section \ref{fullsystem}. The heat operator $\mathcal{H}$ depends only on the velocity $\mathbf{u}$, the Cahn--Hilliard operator $\mathcal{C}$ depends on both $(\mathbf{u}, \Theta)$, and the Stokes operator $\mathcal{S}$ depends on the concentration $c$, yielding the composite solution operator
\[
\mathcal{T}(\mathbf{u}) := \mathcal{S}\big(\mathcal{C}(\mathbf{u}, \mathcal{H}(\mathbf{u}))\big).
\]
Existence follows by applying the Schauder Fixed-Point Theorem to $\mathcal{T}$.

\begin{revisionfive}
For a fixed \(s\in(1/2,1)\), define
\begin{align}
V_D:= &\left\{\mathbf v\in H^1(\Omega)^d:
\nabla\!\cdot\!\mathbf v=0\text{ in }\mathcal D'(\Omega),\ 
\operatorname{Tr}\mathbf v=0\text{ on }\Gamma_D\right\},
\quad
L^2_\sigma(\Omega):=\overline{V_D}^{\,L^2(\Omega)^d},\label{theonewithVD}\\[0.3em]
Y_s:= & \left\{\mathbf v\in H^s(\Omega)^d:
\nabla\!\cdot\!\mathbf v=0\text{ in }\mathcal D'(\Omega),\ 
\operatorname{Tr}\mathbf v=0\text{ on }\Gamma_D\right\},
\quad
E_s(T):=L^2(0,T;Y_s) \label{theonewithEs}.
\end{align}
Here, $\operatorname{Tr}$ denotes the trace operator on $\partial\Omega$. Since $s>\frac12$, the trace condition defining $Y_s$ is well defined. In addition, the space \(Y_s\) is a closed subspace of \(H^s(\Omega)^d\), and \(V_D\hookrightarrow Y_s\) continuously.

\begin{theorem}[Local existence of weak solutions]
\label{thm:main-existence}
Let \(\Omega\subset\mathbb R^d\), \(d=2,3\), be a bounded \(C^{1,1}\) domain. Assume the boundary decompositions and the containment \(\Gamma_{\Theta_N}\subseteq\Gamma_D\) stated in Remark \ref{geometric}. Let
\[
c_0\in H^2(\Omega),\quad \partial_{\mathbf n}c_0=0\text{ on }\partial\Omega,
\quad
\mathbf u_0\in L^2_\sigma(\Omega),
\quad
\Theta_0\in H^1(\Omega),
\]
with \(\Theta_0|_{\Gamma_\Theta}=\Theta_D\), \(\Theta_0\ge0\) almost everywhere, and constant \(\Theta_D\ge0\). Assume further that
\[
\eta\in W^{1,\infty}(\mathbb R),
\quad
0<\eta_{\min}\le \eta(r)\le\eta_{\max}<\infty
\quad(r\in\mathbb R),
\]
and
\[
\rho\in C_b(\mathbb R),
\quad
0<\rho(r)\le\rho_{\max}
\quad(r\in\mathbb R).
\]
Fix \(\alpha>0\) and \(\delta>0\). Then there exists \(T>0\), depending on the fixed dimensionless parameters, \(\alpha\), \(\delta\), the coefficient bounds, the domain, and the norms of the initial data, such that the regularized system \eqref{CHND}--\eqref{HeatND} admits a weak solution \((c,\mu,\mathbf u,\Theta)\) on \((0,T)\) satisfying
\begin{align}
X_c(T)&:=L^\infty(0,T;H^1(\Omega))\cap L^2(0,T;H^2(\Omega))
\cap H^1(0,T;H^{-1}(\Omega)), \label{Xcdef}\\[0.3em]
X_u(T)&:=L^\infty(0,T;L^2_\sigma(\Omega))\cap L^2(0,T;V_D)
\cap H^1(0,T;V_D'), \label{Xudef}\\[0.3em]
X_\Theta(T)&:=L^\infty(0,T;L^2(\Omega))\cap L^2(0,T;H^1(\Omega))
\cap W^{1,1}(0,T;H^{-1}(\Omega)), \label{Xthetadef}
\end{align}
with
\[
c\in X_c(T),\quad \mu\in L^2(0,T;H^1(\Omega)),\quad
\mathbf u\in X_u(T),\quad \Theta\in X_\Theta(T).
\]
The initial conditions are attained in the standard weakly continuous representatives; in particular \(\mathbf u(0)=\mathbf u_0\) in \(L^2_\sigma(\Omega)\). Moreover, \(\Theta\ge0\) almost everywhere on \(\Omega\times(0,T)\). A pressure associated with the divergence-free weak formulation is recovered as described in Section \ref{sec:Stokes}.
\end{theorem}
\end{revisionfive}

\noindent {\bf Proof outline}
\begin{revisionfive}
We establish existence and continuity of the heat operator $\mathcal{H}$ (Section \ref{sec:Heat}), the Cahn--Hilliard operator $\mathcal{C}$ (Section \ref{sec:Cahn-Hilliard}), and the Stokes operator $\mathcal{S}$ (Section \ref{sec:Stokes}). The composite map $\mathcal{T}(\mathbf{u}) := \mathcal{S}(\mathcal{C}(\mathbf{u}, \mathcal{H}(\mathbf{u})))$ is then shown to be continuous with relatively compact image in the admissible Schauder space $E_s(T)$. Existence follows by applying Schauder's theorem in Section \ref{sec:fixedpoint}.
\end{revisionfive}

\begin{revisionfive}
\begin{remark}[Regularization and numerical implementation]\label{remarkfortemp}
The analytical results in this section are established for the doubly regularized system \eqref{CHND}--\eqref{HeatND}. First, the term \(-\alpha\partial_t\mathbf u\), with \(\alpha>0\), is retained in the Stokes equation. Second, the temperature in the chemical potential is replaced by
\[
\hat\Theta=(\rho_\delta*\widetilde\Theta)|_\Omega,
\]
where \(\rho_\delta\) is a standard spatial mollifier and \(\widetilde\Theta\) is the extension of \(\Theta\) by zero to \(\mathbb R^d\). These regularizations are used only to construct a local weak solution of the auxiliary analytical problem. In Section \ref{sec:numerics}, we instead set \(\alpha=0\) and use the unmollified temperature. Accordingly, the numerical experiments are carried out for the corresponding quasi-static model, and no analytical connection between the two formulations is established in the present work.
\end{remark}
\end{revisionfive}

\begin{remark}[Notation conventions]
Throughout this section, we adopt the shorthand notation
\[
L^2(H^2)=L^2(0,T;H^2(\Omega)),\qquad
L^\infty(H^1)=L^\infty(0,T;H^1(\Omega)),
\]
and similar conventions for other function spaces. Likewise, we write
\(\|\cdot\|\) for the \(L^2(\Omega)\) norm. In both cases, the full notation is
occasionally retained for consistency of presentation.
\end{remark}

\subsection{The Heat Operator} \label{sec:Heat}

We begin by analyzing the heat equation for a prescribed divergence-free velocity field $\mathbf{u}$. We establish existence and regularity of weak solutions and prove continuity of the solution operator
\[
\mathcal{H} : \mathbf{u} \mapsto \Theta,
\]
where $\Theta$ denotes the temperature field solving the heat subproblem \eqref{HeatND} subject to the mixed boundary conditions \eqref{HeatDBC}--\eqref{HeatNBC}.

\begin{revisionfive}
\begin{proposition}[Existence and regularity of weak solutions to the heat subproblem]
\label{heatwell}
Let \(R>0\), \(s\in(1/2,1)\), and let \(\mathbf u\in E_s(T)\) satisfy
\(\|\mathbf u\|_{E_s(T)}\le R\), where $E_s(T)$ is defined in \eqref{theonewithEs}. Assume the boundary relation
\(\Gamma_{\Theta_N}\subseteq\Gamma_D\) from Remark \ref{geometric}. Let
\(\Theta_0\in H^1(\Omega)\) satisfy
\(\Theta_0|_{\Gamma_\Theta}=\Theta_D\), where \(\Theta_D\) is constant.
Then, for \(T>0\) sufficiently small, the heat problem \eqref{HeatND}, \eqref{HeatDBC}--\eqref{HeatNBC} has a unique weak solution satisfying
\[
\Theta\in L^\infty(0,T;L^2(\Omega))\cap L^2(0,T;H^1(\Omega))
\cap W^{1,1}(0,T;H^{-1}(\Omega)).
\]
For almost every \(t\in(0,T)\),
\[
\frac12\frac{d}{dt}\|\Theta-\Theta_D\|_{L^2(\Omega)}^2
+\frac1{Pe_\Theta}\|\nabla(\Theta-\Theta_D)\|_{L^2(\Omega)}^2=0.
\]
In particular, by restricting to \(T\le1\), there is a constant
\(M_\Theta=M_\Theta(R,\Theta_0,\Theta_D,Pe_\Theta,\Omega)\), independent of the choice of \(\mathbf u\) in the radius-\(R\) ball of \(E_s(T)\), such that
\[
\|\Theta\|_{X_\Theta(T)}\le M_\Theta,
\]
where $X_\Theta(T)$ is defined by \eqref{Xthetadef}.
\end{proposition}

\begin{proof}
We first introduce the variable
$\widetilde\Theta=\Theta-\Theta_D$,
which satisfies homogeneous Dirichlet boundary conditions on \(\Gamma_\Theta\). Since \(\Theta_D\) is constant, \(\widetilde\Theta\) also satisfies the same equation as $\Theta$. We then apply a standard Galerkin approximation based on the mixed Laplacian to construct approximate solutions \(\widetilde\Theta_n\). Testing the resulting equation with \(\widetilde\Theta_n\) and using the incompressibility condition \(\nabla\!\cdot\mathbf u=0\), we obtain
\[
(\mathbf u\cdot\nabla\widetilde\Theta_n,\widetilde\Theta_n)
=\frac12\int_{\partial\Omega}(\mathbf u\cdot\mathbf n)|\widetilde\Theta_n|^2\,dS=0,
\]
which holds since \(\widetilde\Theta_n=0\) on \(\Gamma_\Theta\), whereas
\(\mathbf u=0\) on \(\Gamma_{\Theta_N}\subseteq\Gamma_D\). Hence
\[
\|\widetilde\Theta_n\|_{L^\infty(L^2)}^2
+\frac2{Pe_\Theta}\|\nabla\widetilde\Theta_n\|_{L^2(L^2)}^2
\le \|\Theta_0-\Theta_D\|^2.
\]

To apply the Aubin--Lions compactness theorem, it remains to obtain a bound for the time derivative. We therefore estimate each term in equation \eqref{HeatND}. For any \(\varphi\in H^1(\Omega)\) vanishing on \(\Gamma_\Theta\),
\[
|\langle\mathbf u\cdot\nabla\widetilde\Theta_n,\varphi\rangle|
\le C\|\mathbf u\|_{H^s}\|\nabla\widetilde\Theta_n\|\|\varphi\|_{H^1}.
\]
Hence, $\mathbf u\cdot\nabla\widetilde\Theta_n$ is bounded in $L^1(H^{-1})$. Moreover, the energy estimate above implies that $\widetilde\Theta_n$ is bounded in $L^2(H^1)$, and therefore $\Delta\widetilde\Theta_n$ is bounded in $L^2(H^{-1})$. It then follows from equation \eqref{HeatND} that $\partial_t\widetilde\Theta_n$ is bounded in $L^1(H^{-1})$.

The Aubin--Lions compactness theorem therefore yields strong convergence (up to a subsequence) in $L^2(L^2)$. This strong convergence allows us to pass to the limit in the Galerkin formulation, yielding the existence of a weak solution together with the stated energy identity. Finally, uniqueness follows by applying the same energy estimate to the difference of two solutions.
\end{proof}
\end{revisionfive}

The next result establishes positivity of the temperature field, which will be used to ensure nonnegativity of the mollified temperature in the Cahn--Hilliard analysis.

\begin{revisionfive}
\begin{proposition}[Nonnegativity of the temperature]\label{pro:maxprinciple}
Let \(\Theta\) be the solution from Proposition \ref{heatwell}. Under the assumptions of Remark \ref{geometric}, if \(\Theta_0\ge0\) almost everywhere and \(\Theta_D\ge0\), then
\[
\Theta\ge0\quad\text{almost everywhere in }\Omega\times(0,T).
\]
\end{proposition}

\begin{proof}
Define the negative part of the temperature by
\[
\Theta^-:=\max\{-\Theta,0\}.
\]
Testing equation \eqref{HeatND} by $-\Theta^-$ gives
\[
\frac12\frac{d}{dt}\|\Theta^-\|^2
+\frac1{Pe_\Theta}\|\nabla\Theta^-\|^2
-\int_\Omega(\mathbf u\cdot\nabla\Theta)\Theta^-\,dx=0.
\]
Since $\nabla\Theta=-\nabla\Theta^-$ on the set $\{\Theta<0\}$ and $\nabla\!\cdot\mathbf u=0$, integration by parts yields
\[
\int_\Omega(\mathbf u\cdot\nabla\Theta)\Theta^-\,dx
=-\frac12\int_{\partial\Omega}(\mathbf u\cdot\mathbf n)|\Theta^-|^2\,dS.
\]
The boundary integral vanishes because $\Theta^-=0$ on $\Gamma_\Theta$ and $\mathbf u=0$ on $\Gamma_{\Theta_N}\subseteq\Gamma_D$. Since $\Theta_0\ge0$, we have $\Theta^-(0)=0$, and therefore
\[
\Theta^-\equiv0.
\]
Hence $\Theta\ge0$ almost everywhere in $\Omega\times(0,T)$.
\end{proof}
\end{revisionfive}

We conclude this subsection by establishing the continuity of the heat solution operator with respect to the prescribed velocity field. This result will be one of the key ingredients in the fixed-point argument.

\begin{revisionfive}
\begin{proposition}[Continuity of the heat solution operator]
\label{prop:H-continuity}
Let \(\mathbf u_n,\mathbf u\in E_s(T)\) satisfy
\[
\mathbf u_n\to\mathbf u
\quad\text{strongly in }E_s(T),
\]
and let $\Theta_n=\mathcal H(\mathbf u_n)$ and
 $\Theta=\mathcal H(\mathbf u)$
denote the corresponding weak solutions given by Proposition \ref{heatwell}. Then
\[
\Theta_n\to\Theta
\quad\text{strongly in }L^2(0,T;L^2(\Omega)).
\]
\end{proposition}

\begin{proof}
By Proposition \ref{heatwell}, the sequence $\{\Theta_n\}$ is uniformly bounded in
$L^\infty(L^2)\cap L^2(H^1)$, while
$\{\partial_t\Theta_n\}$ is uniformly bounded in
$L^1(H^{-1})$. Therefore, the Aubin--Lions compactness theorem yields the existence of a subsequence, still denoted by $\{\Theta_n\}$, and a function $\widetilde\Theta$ such that
\[
\Theta_n\to\widetilde\Theta
\quad\text{strongly in }L^2(L^2).
\]

It remains to identify the limit. Passing to the weak formulation, we write the difference of the transport terms as
\[
\mathbf u_n\cdot\nabla\Theta_n-
\mathbf u\cdot\nabla\widetilde\Theta
=
(\mathbf u_n-\mathbf u)\cdot\nabla\Theta_n
+
\mathbf u\cdot\nabla(\Theta_n-\widetilde\Theta).
\]
The first term converges to zero by the strong convergence of $\mathbf u_n$ in $E_s(T)$ together with the uniform $L^2(H^1)$ bound for $\Theta_n$. For the second term, we integrate the spatial derivative onto the test function and use the strong convergence of $\Theta_n$ in $L^2(L^2)$ together with the admissible boundary conditions.

Consequently, $\widetilde\Theta$ satisfies the weak formulation of the heat equation corresponding to the velocity field $\mathbf u$. By uniqueness of weak solutions, we conclude that $\widetilde\Theta=\Theta$. Finally, the standard subsequence argument implies that the entire sequence converges, completing the proof.
\end{proof}
\end{revisionfive}

\subsection{The Cahn--Hilliard Operator}\label{sec:Cahn-Hilliard}

We now analyze the Cahn--Hilliard system for prescribed velocity and temperature fields $(\mathbf{u}, \Theta)$. The main analytical difficulty is handling the temperature-dependent coefficient $\Theta$ in the free energy, which requires spatial regularity beyond what the heat equation provides. To address this, we introduce a mollified temperature field.

For a fixed mollification parameter $\delta>0$, we extend $\Theta$ by zero outside $\Omega$, denoted by $\widetilde{\Theta}$, and define
\begin{equation}\label{thetahatdefinition}
\hat\Theta := (\rho_\delta * \widetilde{\Theta})\big|_{\Omega},
\end{equation}
where $\rho_\delta$ is a standard smooth mollifier and the convolution is taken with respect to the spatial variables only.
Then $\hat\Theta \in C^\infty(\Omega)$ for almost every $t \in (0,T)$. Moreover, 
since $\rho_\delta \in C^\infty_c(\mathbb{R}^d)$, Young's convolution 
inequality gives
\begin{equation}\label{theonewiththeta}
\|\hat\Theta\|_{W^{2,\infty}(\Omega)} 
\le C_\delta \|\Theta\|,
\end{equation}
where $C_\delta$ depends only on the mollification parameter. \begin{revisionfive}
This estimate will play an important role in the compactness arguments that follow.

We establish existence and uniqueness for prescribed inputs by Galerkin approximation, energy estimates, and compactness. We then prove sequential continuity of the solution operator
\[
\mathcal C:(\mathbf u,\Theta)\mapsto c
\]
in the topology required by the fixed-point argument.
\end{revisionfive}

\begin{revisionfive}
\begin{definition}[Weak solution]\label{def:weakdef}
Let \(T>0\), let \(\mathbf u\in E_s(T)\), and let
\(\Theta\in X_\Theta(T)\), where the spaces $E_s(T)$ and $X_{\Theta}(T)$ are given by \eqref{theonewithEs} and \eqref{Xthetadef}, respectively, with \(\Theta\ge0\) almost everywhere. A pair
\((c,\mu)\) is a weak solution of the prescribed-input Cahn--Hilliard problem on \((0,T)\) if
\[
c\in L^\infty(0,T;H^1(\Omega))\cap L^2(0,T;H^2(\Omega)),
\quad
\partial_t c\in L^2(0,T;H^{-1}(\Omega)),
\quad
\mu\in L^2(0,T;H^1(\Omega)),
\]
and, for all admissible test functions \(\varphi,\psi\in L^2(0,T;H^1(\Omega))\),
\begin{align}
\int_0^T\left[
\langle\partial_t c,\varphi\rangle_{H^{-1},H^1}
+(\mathbf u\cdot\nabla c,\varphi)
+\frac1{Pe}(\nabla\mu,\nabla\varphi)
\right]dt=0,&\label{weakofc}\\[0.3em]
\int_0^T(\mu,\psi)\,dt
=\int_0^T\left[
(8c^3-2(1-\hat\Theta)c,\psi)
+Ch^2(\nabla c,\nabla\psi)
\right]dt.&\label{weakofmi}
\end{align}
The initial condition is \(c(0)=c_0\).
\end{definition}
\end{revisionfive}

\begin{revisionfive}
\begin{proposition}[Existence and regularity of weak solutions]
\label{prop:CH_exist}
Let \(\mathbf u\in E_s(T)\) and \(\Theta\in X_\Theta(T)\) satisfy
\(\Theta\ge0\) almost everywhere and
\[
\|\mathbf u\|_{E_s(T)}+\|\Theta\|_{X_\Theta(T)}\le R.
\]
For every fixed \(\delta>0\), there exists \(T_{CH}>0\), depending only on the fixed parameters, \(R\), \(\delta\), and \(c_0\), such that the prescribed-input Cahn--Hilliard problem has a weak solution in the sense of Definition \ref{def:weakdef} on \((0,T_{CH})\).
\end{proposition}
\end{revisionfive}

\begin{proof}[Proof of Proposition \ref{prop:CH_exist}]

\noindent\textbf{Step 1: Galerkin approximation.}
We construct a Galerkin approximation by projecting the system onto a finite-dimensional space. Let
$\{\phi_k\}_{k=1}^\infty$ be the eigenfunctions of the Neumann Laplacian satisfying
$\partial_n\phi_k=0$ on $\partial\Omega$ \cite{Evans}, and define
\[
V_n=\operatorname{span}\{\phi_1,\dots,\phi_n\}.
\]
Since the Neumann eigenfunctions belong to $H^2(\Omega)$, we have
\[
V_n\subset H^2(\Omega)\subset H^1(\Omega).
\]

We seek approximate solutions of the form
\[
c_n(t,x)=\sum_{j=1}^n g_j(t)\phi_j(x),\quad
\mu_n(t,x)=\sum_{j=1}^n h_j(t)\phi_j(x),
\]
which satisfy, for every $\phi,\psi\in V_n$,
\begin{align}
\langle \partial_t c_n,\phi\rangle_{H^{-1},H^1}
+(\mathbf u\!\cdot\!\nabla c_n,\phi)
+\frac1{Pe}(\nabla\mu_n,\nabla\phi)=0,
&\label{weak}\\[0.3em]
(\mu_n,\psi)
=
(8c_n^3-2(1-\hat\Theta)c_n,\psi)
+
Ch^2(\nabla c_n,\nabla\psi).
&\label{definitionofmu_n}
\end{align}
Here, $\hat{\Theta}$ denotes the regularized temperature introduced in \eqref{thetahatdefinition}.

The initial condition is prescribed as
\[
c_n(0) = P_n c_0,
\]
where $P_n$ denotes the $L^2$-projection onto $V_n$.
Since the coefficients are measurable in time and locally Lipschitz in the Galerkin variables, standard ODE theory guarantees a unique local solution on some interval \((0,t^*(n))\).

We now derive uniform estimates independent of $n$, starting with a bound for the time derivative.

\medskip
\begin{revisionfive}
\noindent\textbf{Estimate for \(\partial_t c_n\) in \(H^{-1}\).}
Since the Galerkin formulation is valid only for test functions in \(V_n\), we first project arbitrary test functions onto the Galerkin space. Let \(P_n\) denote the spectral projection onto \(V_n\). Since the Neumann eigenfunctions diagonalize the Neumann Laplacian, \(P_n\) is uniformly bounded on \(H^1(\Omega)\). Moreover, \(\partial_t c_n\in V_n\), and therefore, for every \(\phi\in H^1(\Omega)\),
\[
\langle\partial_t c_n,\phi\rangle
=
\langle\partial_t c_n,P_n\phi\rangle.
\]
Using \eqref{weak} with \(P_n\phi\), H\"older's inequality, and
\[
H^s(\Omega)\hookrightarrow L^{6/(3-2s)}(\Omega),
\quad
H^1(\Omega)\hookrightarrow L^{3/s}(\Omega),
\]
we obtain
\[
|\langle\mathbf u\cdot\nabla c_n,P_n\phi\rangle|
\le C\|\mathbf u\|_{H^s}\|\nabla c_n\|\|\phi\|_{H^1},
\quad
|(\nabla\mu_n,\nabla P_n\phi)|
\le C\|\nabla\mu_n\|\|\phi\|_{H^1}.
\]
Consequently,
\begin{equation}\label{FirstBigH2}
\|\partial_t c_n\|_{H^{-1}}
\le C\left(\|\mathbf u\|_{H^s}\|\nabla c_n\|
+\|\nabla\mu_n\|\right).
\end{equation}
\end{revisionfive}
\revfive{The projected estimate above is used throughout the remainder of the proof.}

\medskip
We next test the weak formulation with $\mu_n$ to control $\|\nabla\mu_n\|$ and $\|\nabla c_n\|$.

\medskip
\noindent\textbf{Multiplying by $\mu_n$.}
Testing \eqref{weak} with $\mu_n \in V_n$ and using \eqref{definitionofmu_n}, we treat each term separately.

\medskip
\noindent\textit{Time-dependent term.}
We obtain
\[
\textcolor{revisionfivecolor}{
\langle \partial_t c_n, \mu_n \rangle_{H^{-1},H^1}}
=
\frac12\frac{d}{dt}\big(4\|c_n^2\|^2+Ch^2\|\nabla c_n\|^2\big)
\textcolor{revisionfivecolor}{
-2\langle \partial_t c_n,(1-\hat{\Theta})c_n\rangle_{H^{-1},H^1}.}
\]
\textcolor{revisionfivecolor}{Using the duality pairing between $H^{-1}(\Omega)$ and $H^1(\Omega)$, we obtain}
\[
\textcolor{revisionfivecolor}{
\begin{aligned}
\big|\langle \partial_t c_n,(1-\hat{\Theta})c_n\rangle_{H^{-1},H^1}\big|
\le\;&
\|\partial_t c_n\|_{H^{-1}}
\|(1-\hat{\Theta})c_n\|_{H^1}.
\end{aligned}}
\]

\begin{revisionfive}
For the \(H^1\)-norm, the product rule gives
\[
\nabla\big((1-\hat\Theta)c_n\big)
=(1-\hat\Theta)\nabla c_n-c_n\nabla\hat\Theta.
\]
Therefore, using \(\hat\Theta\in W^{1,\infty}(\Omega)\),
\end{revisionfive}
\[
\textcolor{revisionfivecolor}{
\begin{aligned}
\|(1-\hat{\Theta})c_n\|_{H^1}
\le\;
C\bigl(1+\|\hat{\Theta}\|_{W^{1,\infty}}\bigr)
\|c_n\|_{H^1}.
\end{aligned}}
\]

\textcolor{revisionfivecolor}{Combining the above estimate with Young's inequality, we conclude}
\[
\textcolor{revisionfivecolor}{
\begin{aligned}
\big|\langle \partial_t c_n,(1-\hat{\Theta})c_n\rangle_{H^{-1},H^1}\big|
\le\;
\varepsilon\|\partial_t c_n\|_{H^{-1}}^2+
C_{\varepsilon}\bigl(1+\|\hat{\Theta}\|_{W^{1,\infty}}^2\bigr)
\|c_n\|_{H^1}^2.
\end{aligned}}
\]

\medskip
\noindent\textit{Convective term.}
\revfive{Using the same H\"older--Sobolev estimate as in \eqref{FirstBigH2}, we obtain}
\begin{align*}
|(\mathbf{u}\!\cdot\!\nabla c_n,\mu_n)|
&\le
C_\varepsilon\|\mathbf u\|_{H^s}^2\|\nabla c_n\|^2
+\varepsilon\|\mu_n\|_{H^1}^2 =
C_\varepsilon\|\mathbf u\|_{H^s}^2\|\nabla c_n\|^2
+\varepsilon\|\mu_n\|^2
+\varepsilon\|\nabla\mu_n\|^2.
\end{align*}
\textcolor{revisionfivecolor}{Taking $\psi=\mu_n$ in \eqref{definitionofmu_n}, applying Young's inequality, and using the Sobolev embedding $H^1(\Omega)\hookrightarrow L^6(\Omega)$, we deduce}
\[
\textcolor{revisionfivecolor}{
\|\mu_n\|^2
\le
\varepsilon\|\nabla\mu_n\|^2
+
C_\varepsilon\left(
\|c_n\|_{H^1}^2
+
\|c_n\|_{H^1}^6
\right).
}
\]
\textcolor{revisionfivecolor}{Substituting this estimate into the previous inequality, we conclude that}
\[
|(\mathbf{u}\!\cdot\!\nabla c_n,\mu_n)|
\le
C(\|\mathbf{u}\|_{H^s},\|\hat{\Theta}\|_{W^{2,\infty}})
(\|c_n\|_{H^1}^2+\|c_n\|_{H^1}^6)
+\varepsilon\|\nabla\mu_n\|^2.
\]

\medskip
\noindent\textit{Dissipative term.}
The remaining term satisfies
\[
\frac{1}{Pe}(\nabla\mu_n,\nabla\mu_n)=\frac{1}{Pe}\|\nabla\mu_n\|^2.
\]

\medskip
\noindent\textit{Combined estimate.}
Combining the above bounds and using \eqref{FirstBigH2}, the term $\|\nabla\mu_n\|^2$ can be absorbed, yielding
\begin{align}\label{secondbig}
&\frac12\frac{d}{dt}\big(4\|c_n^2\|^2+Ch^2\|\nabla c_n\|^2\big)
+ C\|\nabla\mu_n\|^2 \nonumber \\[0.3em]
&~~~~\le  C(\|\mathbf{u}\|_{H^s},\|\hat{\Theta}\|_{W^{2,\infty}})
(\|c_n\|_{H^1}^2+\|c_n\|_{H^1}^6).
\end{align}

\medskip
We now test the weak formulation with $c_n$ to obtain a complementary estimate, which, together with previous estimates, provides, among other things, control of the full $H^1(\Omega)$-norm of $c_n$.

\noindent\textbf{Multiplying by $c_n$.}
Testing \eqref{weak} with $c_n \in V_n$ and integrating by parts yields
\begin{align}\label{cnmult1}
   \frac{1}{2}\frac{d}{dt}\|c_n\|^2
   + \frac{Ch^2}{Pe}\|\Delta c_n\|^2
   &= -(\mathbf{u}\!\cdot\!\nabla c_n, c_n)
   - \frac{1}{Pe}(\nabla(8c_n^3 - 2(1-\hat{\Theta})c_n), \nabla c_n).
\end{align}
Expanding the nonlinear gradient gives
\[
\nabla(8c_n^3 - 2(1-\hat{\Theta})c_n)
= 24c_n^2\nabla c_n - 2(1-\hat{\Theta})\nabla c_n
+ 2(\nabla \hat{\Theta})\,c_n.
\]
Substitution yields
\begin{align}\label{preestimate}
   \frac{1}{2}\frac{d}{dt}\|c_n\|^2
   + \frac{Ch^2}{Pe}\|\Delta c_n\|^2
   + \frac{24}{Pe}\|c_n\nabla c_n\|^2
   + \frac{2}{Pe}(\hat{\Theta}\nabla c_n,\nabla c_n)
   = \frac{2}{Pe}\|\nabla c_n\|^2 + I_1 + I_2,
\end{align}
where
\[
I_1 := -(\mathbf{u}\!\cdot\!\nabla c_n,\,c_n),
\quad
I_2 := -\frac{2}{Pe}(\nabla\hat{\Theta}\,c_n,\,\nabla c_n).
\]

The term $(\hat{\Theta}\nabla c_n,\nabla c_n)$ on the left-hand side is nonnegative. This follows from the nonnegativity of the prescribed temperature $\Theta$ (established in Proposition \ref{pro:maxprinciple}), a property preserved by extension by zero and convolution with a nonnegative mollifier, and hence $\hat{\Theta}\ge 0$.

\medskip
\noindent
\textit{Estimates for $I_1$ and $I_2$.}
Applying Hölder's and Young's inequalities to the terms $I_1$ and $I_2$ defined above, we obtain
\begin{align*}
|I_1| &\le \|\mathbf{u}\|_{H^s}\,\|c_n\nabla c_n\|
      \le \varepsilon\|c_n\nabla c_n\|^2
      + C_{\varepsilon}\|\mathbf{u}\|_{H^s}^2,\\[0.3em]
|I_2| &\le \frac{2}{Pe}\|c_n\nabla c_n\|\,\|\nabla \hat{\Theta}\|
      \le \varepsilon\|c_n\nabla c_n\|^2
      + C_{\varepsilon}\|\nabla \hat{\Theta}\|^2.
\end{align*}

Combining these bounds with \eqref{preestimate}, dropping the nonnegative term $(\hat{\Theta}\nabla c_n,\nabla c_n)$ from the left-hand side, and choosing $\varepsilon_1,\varepsilon_2$ sufficiently small, we obtain
\begin{align}\label{ThirdBig}
   \frac{1}{2}\frac{d}{dt}\|c_n\|^2
   + \frac{Ch^2}{Pe}\|\Delta c_n\|^2
   + C_3\|c_n\nabla c_n\|^2
   \le C(\|\mathbf{u}\|_{H^s},\|\nabla \hat{\Theta}\|)
   + \frac{2}{Pe}\|\nabla c_n\|^2,
\end{align}
where $C_3>0$ absorbs the small parameters $\varepsilon_1$ and $\varepsilon_2$.

Together with the previous estimates, \eqref{ThirdBig} provides the differential inequality from which uniform $n$-independent bounds for $c_n$ and $\mu_n$ are obtained.

\medskip
\noindent\textbf{Step 2: Uniform bounds.}
Using \eqref{secondbig} and \eqref{ThirdBig}, together with standard absorption arguments and Young's inequality, we obtain
\begin{align}\label{finalboundexistence}
&\frac12\frac{d}{dt}E_n(t)
+C_1\big(\|\nabla\mu_n\|^2+\|\Delta c_n\|^2+\|c_n\nabla c_n\|^2\big)
\nonumber\\[0.3em]
&\qquad\le
C(\|\mathbf u\|_{H^s},\|\hat\Theta\|_{W^{2,\infty}})
\bigl(1+\|c_n\|_{H^1}^2+\|c_n\|_{H^1}^6\bigr),
\end{align}
where
\[
E_n(t)
:=
4\|c_n^2(t)\|^2
+
Ch^2\|\nabla c_n(t)\|^2
+
\|c_n(t)\|^2.
\]

Since $c_n(0)=P_nc_0$, the projection property together with the embedding
$H^2(\Omega)\hookrightarrow L^\infty(\Omega)$ implies that $E_n(0)$ is bounded independently of $n$. Moreover,
\begin{revisionfive}
\[
\|\mathbf u\|_{E_s(T)}+
\|\Theta\|_{X_\Theta(T)}
\le R.
\]
\end{revisionfive}
\begin{revisionfive}
Furthermore, the spatial mollification satisfies
\[
\|\hat\Theta(t)\|_{W^{2,\infty}(\Omega)}
\le C_\delta\|\Theta(t)\|
\quad\text{for a.e. }t.
\]
Hence, the $L^\infty(L^2)$ bound for $\Theta$ implies that
$\hat\Theta$ is uniformly bounded in
$L^\infty(W^{2,\infty})$. Therefore, the dependence on
$\hat\Theta$ can be absorbed into the constant in
\eqref{finalboundexistence}, yielding
\[
\frac{d}{dt}E_n(t)+c_0\mathcal D_n(t)
\le C_{R,\delta}\bigl(1+\|\mathbf u(t)\|_{H^s}^2\bigr)
\bigl(1+E_n(t)^3\bigr),
\]
where
\[
\mathcal D_n
=
\|\nabla\mu_n\|^2
+
\|\Delta c_n\|^2
+
\|c_n\nabla c_n\|^2.
\]
Since $\mathbf u\in E_s(T)$, with $E_s(T)$ defined in \eqref{theonewithEs}, we have
$1+\|\mathbf u\|_{H^s}^2\in L^1(0,T)$. Therefore, a nonlinear version of
Gr\"onwall's inequality \cite{Dragomir2003} yields a time $T_{CH}>0$, depending only on $R$,
$\delta$, and the initial energy, such that $E_n$ remains uniformly bounded
on $(0,T_{CH})$. Integrating the same inequality gives uniform bounds for \(\mathcal D_n\) in \(L^1(0,T_{CH})\).

\end{revisionfive} \textcolor{revisionfivecolor}{Following this, integrating \eqref{finalboundexistence} over $(0,T_{CH})$ immediately gives
\[
\nabla\mu_n,\;
\Delta c_n,\;
c_n\nabla c_n
\in
L^2(0,T_{CH};L^2(\Omega))
\]
uniformly in $n$. Furthermore, the estimate obtained from
\eqref{definitionofmu_n} implies
\[
\|\mu_n\|^2
\le
\varepsilon\|\nabla\mu_n\|^2
+
C_\varepsilon
\left(
\|c_n\|_{H^1}^2
+
\|c_n\|_{H^1}^6
\right),
\]
and therefore $\mu_n
\in
L^2(0,T_{CH};H^1(\Omega))$ uniformly in $n$. Lastly, since
\[
\Delta c_n\in L^2(0,T_{CH};L^2(\Omega))
\quad\text{and}\quad
c_n\in L^\infty(0,T_{CH};H^1(\Omega)),
\]
\revfive{the Neumann elliptic estimate for functions satisfying the homogeneous Neumann condition gives \cite{Grisvard}
\[
\|c_n\|_{H^2}
\le
C\bigl(\|\Delta c_n\|+\|c_n\|_{H^1}\bigr)
\]}
so that
\[
c_n
\in
L^2(0,T_{CH};H^2(\Omega))
\]
uniformly in $n$.
Consequently,
\[
c_n
\in
L^\infty(0,T_{CH};H^1(\Omega))
\cap
L^2(0,T_{CH};H^2(\Omega)),
\]
\[
\mu_n
\in
L^2(0,T_{CH};H^1(\Omega)),
\quad
\partial_t c_n
\in
L^2(0,T_{CH};H^{-1}(\Omega)),
\]
with bounds independent of $n$. These estimates provide the compactness required to pass to the limit in the Galerkin approximation.}

\medskip \textcolor{revisionfivecolor}{
\noindent {\bf Step 3: Compactness and passage to the limit.}
Using the uniform bounds established above, the Aubin--Lions lemma yields the
existence of a subsequence (not relabeled) and a limit function $c$ such that
\begin{equation}\label{theonewithstrongconvergence}
c_n \to c \quad\text{strongly in }L^2(H^r), \quad \forall\, r\in[0,2).
\end{equation}
By weak compactness, we further obtain
\[
c_n \rightharpoonup c \text{ in }L^2(H^2), \quad
\partial_t c_n \rightharpoonup \partial_t c \text{ in }L^2(H^{-1}), \quad
\nabla\mu_n \rightharpoonup \nabla\mu \text{ in }L^2(L^2),
\]
for some limit function $\mu$.}

\medskip
\noindent \textcolor{revisionfivecolor}{
{\it Passage to the limit.}
It remains to pass to the limit in the Galerkin formulation. We verify convergence of each term in the weak formulation of
Definition \ref{def:weakdef}. Parts (a)--(b) establish the convergence in
\eqref{weakofc}, while parts (c)--(d) establish the convergence in
\eqref{weakofmi}.
~~\\
\noindent ({\it a}) Since
$\partial_t c_n \rightharpoonup \partial_t c
\quad\text{in }L^2(H^{-1})$ and $\nabla\mu_n \rightharpoonup \nabla\mu
\quad\text{in }L^2(L^2)$, the time-derivative and diffusive terms converge directly by the definition of weak convergence.
~~\\
\revfive{\noindent ({\it b}) For the convective term, let $\varphi\in L^\infty(H^1)$. Using the same H\"older--Sobolev estimate as in \eqref{FirstBigH2}, we obtain}
\[
|(\mathbf u\!\cdot\!\nabla(c_n-c),\varphi)|
\le
C
\|\mathbf u\|_{H^s}
\|\nabla(c_n-c)\|
\|\varphi\|_{H^1}.
\]
~~
Integrating over $(0,T)$, applying Hölder's inequality in time, and using \eqref{theonewithstrongconvergence}, we have
\[
\int_0^T
|(\mathbf u\!\cdot\!\nabla(c_n-c),\varphi)|
\,dt
\le
C
\|\mathbf u\|_{L^2(H^s)}
\|\nabla(c_n-c)\|_{L^2(L^2)}
\|\varphi\|_{L^\infty(H^1)}
\to0.
\]
Hence,
\[
(\mathbf u\!\cdot\!\nabla c_n,\varphi)
\to
(\mathbf u\!\cdot\!\nabla c,\varphi).
\]}

\noindent ({\it c}) For the cubic nonlinearity in $\mu$, we write
\[
c_n^3 - c^3 = (c_n - c)(c_n^2 + c_nc + c^2).
\]
Testing against $\psi\in L^\infty(H^1)$ and using the Sobolev embedding
$H^{3/2+\varepsilon}(\Omega)\hookrightarrow L^\infty(\Omega)$ gives
\[
|(c_n^3 - c^3,\psi)|
\le C\|c_n - c\|_{H^{3/2+\varepsilon}}
\|c_n^2 + c_nc + c^2\|\,\|\psi\|.
\]
Integrating in time and applying Hölder’s inequality yields
\[
\int_0^T |(c_n^3 - c^3,\psi)|\,dt
\le C\|c_n - c\|_{L^2(H^{3/2+\varepsilon})}
\|c_n^2 + c_nc + c^2\|_{L^2(L^2)}
\|\psi\|_{L^\infty(H^1)}.
\]
The first factor converges to zero by
\eqref{theonewithstrongconvergence}, while the remaining
terms are uniformly bounded. Consequently,
\[
(8c_n^3,\psi)\to(8c^3,\psi).
\]
\noindent ({\it d}) It remains to pass to the limit in the weak formulation defining the chemical potential. The convergence of the cubic term follows from part (c),
while the linear term converges by the strong convergence
$c_n\to c$ in $L^2(L^2)$. Finally, since
$\nabla c_n\rightharpoonup\nabla c
\quad\text{in }L^2(L^2),
$ the diffusion term converges directly by weak convergence.

Combining (a)--(d), we conclude that the pair $(c,\mu)$ satisfies the weak
formulation of the Cahn--Hilliard system in the sense of
Definition \ref{def:weakdef}.

\end{proof}

We now establish continuous dependence of the solution with respect to the input data $(\mathbf{u}, \Theta)$, which is essential for the fixed-point argument.

\begin{revisionfive}
\begin{proposition}[Uniqueness and sequential continuity of the Cahn--Hilliard solution map]
\label{prop:CH_continuity_full}
Let \(R>0\) and \(T\le T_{CH}(R)\). For every prescribed pair
\((\mathbf u,\Theta)\in E_s(T)\times X_\Theta(T)\) with
\(\Theta\ge0\) and norm bounded by \(R\), the weak solution of the prescribed-input Cahn--Hilliard problem is unique. Moreover, suppose
\[
\mathbf u_n\to\mathbf u\quad\text{strongly in }E_s(T),
\quad
\Theta_n\to\Theta\quad\text{strongly in }L^2(0,T;L^2(\Omega)),
\]
and let $c_n=\mathcal C(\mathbf u_n,\Theta_n)$ and $c=\mathcal C(\mathbf u,\Theta)$ denote the corresponding weak solutions given by Proposition \ref{prop:CH_exist}. If the pairs \((\mathbf u_n,\Theta_n)\) are uniformly bounded in
\(E_s(T)\times X_\Theta(T)\), with \(\Theta_n\ge0\), then, for every $r\in [0,2)$,
\[
c_n\to c
\quad\text{strongly in }L^2(0,T;H^r(\Omega)).
\]
\end{proposition}

\begin{proof}

\noindent\textbf{Step 1: Uniqueness.} Let $(c_1,\mu_1)$ and $(c_2,\mu_2)$ be two weak solutions corresponding to the same prescribed pair $(\mathbf u,\Theta)$ and the same initial datum. Define
\[
e:=c_1-c_2,
\qquad
q:=\mu_1-\mu_2.
\]
Then $(e,q)$ satisfies
\begin{equation}\label{theonewithe}
\partial_t e+\mathbf u\cdot\nabla e
=\frac1{Pe}\Delta q,
\end{equation}
together with
\[
q
=
8(c_1^3-c_2^3)
-
2(1-\hat\Theta)e
-
Ch^2\Delta e.
\]

We begin by establishing two auxiliary estimates that will be used repeatedly throughout the proof. For the cubic nonlinearity, we observe that
\[
c_1^3-c_2^3
=
(c_1-c_2)(c_1^2+c_1c_2+c_2^2).
\]
Using H\"older's inequality together with the embedding
$H^1(\Omega)\hookrightarrow L^6(\Omega)$, the uniform
$L^\infty(H^1)$ bounds from Proposition
\ref{prop:CH_exist} imply
\begin{equation}
    \label{theonewithcubes}
    \|c_1^3-c_2^3\|
\le
C
\left(
\|c_1\|_{H^1}^2
+
\|c_2\|_{H^1}^2
\right)
\|e\|_{H^1}.
\end{equation}

Next, we derive an interpolation estimate for $e$. By the interpolation inequality together with the Neumann elliptic estimate \cite{Grisvard},
\[
\|e\|_{H^1}
\le
C
\|e\|^{1/2}
\|e\|_{H^2}^{1/2}
\quad \text{ and } \quad
\|e\|_{H^2}
\le
C\bigl(\|\Delta e\|+\|e\|_{H^1}\bigr).
\]
Therefore,
\[
\|e\|_{H^1}^2
\le
C\|\Delta e\|\|e\|
+
C\|e\|_{H^1}\|e\|.
\]
Applying Young's inequality to both terms and absorbing the resulting
$\|e\|_{H^1}^2$ contribution into the left-hand side yields
\begin{equation}\label{H1interpolationdifference}
\|e\|_{H^1}^2
\le
\varepsilon\|\Delta e\|^2
+
C_\varepsilon\|e\|^2.
\end{equation}

We now test \eqref{theonewithe} by $e$.
Using the homogeneous Neumann boundary condition and the expression for $q$, we obtain
\begin{align*}
\frac12\frac{d}{dt}\|e\|^2
+
\frac{Ch^2}{Pe}\|\Delta e\|^2
={}&
-(\mathbf u\cdot\nabla e,e)
+\frac8{Pe}(c_1^3-c_2^3,\Delta e)
-
\frac2{Pe}((1-\hat\Theta)e,\Delta e).
\end{align*}

It remains to estimate each term on the right-hand side.

\medskip

We begin with the cubic term. Using
\eqref{theonewithcubes},
\eqref{H1interpolationdifference},
and Young's inequality, we obtain
\[
\begin{aligned}
|(c_1^3-c_2^3,\Delta e)|
\le
\varepsilon\|\Delta e\|^2
+
C_\varepsilon\|e\|_{H^1}^2 \le
\varepsilon\|\Delta e\|^2
+
C_\varepsilon\|e\|^2.
\end{aligned}
\]

For the linear temperature term, the regularity
$\hat\Theta\in W^{2,\infty}(\Omega)$
together with Young's inequality yields
\[
|((1-\hat\Theta)e,\Delta e)|
\le
\varepsilon\|\Delta e\|^2
+
C_\varepsilon\|e\|^2.
\]

Lastly, to estimate the transport term, we proceed as follows. 
As in the derivation of \eqref{FirstBigH2}, H\"older's inequality 
together with the Sobolev embeddings yields
\[
|(\mathbf u\cdot\nabla e,e)|
\le
C
\|\mathbf u\|_{H^s}
\|\nabla e\|
\|e\|_{L^{3/s}}.
\]
A direct application of the interpolation estimate
\eqref{H1interpolationdifference} to bound
$\|e\|_{L^{3/s}}$ by $\|e\|_{H^1}$ would produce the factor
$\|\mathbf u\|_{H^s}$ in front of the absorbable term
$\|\Delta e\|^2$. Since
$\mathbf u\in L^2(H^s)$, this coefficient cannot be controlled 
uniformly in time. We therefore estimate the $L^{3/s}$ norm 
directly by Gagliardo--Nirenberg interpolation 
\[
\|e\|_{L^{3/s}}
\le
C
\|e\|_{H^2}^{\frac{3-2s}{4}}
\|e\|^{\frac{1+2s}{4}}.
\]
Combining the Neumann elliptic estimate with 
\eqref{H1interpolationdifference} gives 
$\|e\|_{H^2}\le C(\|\Delta e\|+\|e\|)$, and therefore
\[
\|e\|_{L^{3/s}}
\le
C
\left(
\|\Delta e\|^{\frac{3-2s}{4}}
\|e\|^{\frac{1+2s}{4}}
+
\|e\|
\right).
\]
Similarly, from the interpolation inequality 
$\|e\|_{H^1}\le C\|e\|^{1/2}\|e\|_{H^2}^{1/2}$ 
and the same bound for $\|e\|_{H^2}$,
\[
\|\nabla e\|
\le
C\|e\|_{H^1}
\le
C
\left(
\|\Delta e\|^{1/2}
\|e\|^{1/2}
+
\|e\|
\right).
\]
Substituting these estimates and expanding the product yields
\[
|(\mathbf u\cdot\nabla e,e)|
\le
C\|\mathbf u\|_{H^s}
\left(
\|\Delta e\|^{\frac{5-2s}{4}}\|e\|^{\frac{3+2s}{4}}
+
\|\Delta e\|^{1/2}\|e\|^{3/2}
+
\|\Delta e\|^{\frac{3-2s}{4}}\|e\|^{\frac{5+2s}{4}}
+
\|e\|^2
\right).
\]
Applying Young's inequality to each term so that every power 
of $\|\Delta e\|$ is raised to 2 to be absorbed into the 
left-hand side, we obtain
\[
|(\mathbf u\cdot\nabla e,e)|
\le
\varepsilon\|\Delta e\|^2
+
C_\varepsilon
\bigl(1+\|\mathbf u\|_{H^s}^{q_1}
+\|\mathbf u\|_{H^s}^{q_2}
+\|\mathbf u\|_{H^s}^{q_3}
+\|\mathbf u\|_{H^s}\bigr)
\|e\|^2,
\]
where \[q_1=\frac{8}{3+2s}, \quad q_2=\frac{4}{3}, \quad 
q_3=\frac{8}{5+2s}.\] 
Since $q_1,q_2,q_3<2$ with $1/2<s<1$ and 
$\mathbf u\in L^2(H^s)$, H\"older's inequality  implies that 
$\|\mathbf u\|_{H^s}^{q_i}\in L^1(0,T)$ for $i=1,2,3$, 
and likewise $\|\mathbf u\|_{H^s}\in L^1(0,T)$.
Combining the previous estimates and choosing 
$\varepsilon>0$ sufficiently small yields
\[
\frac{d}{dt}\|e\|^2
+
c\|\Delta e\|^2
\le
C
\bigl(
1+\|\mathbf u\|_{H^s}^{q_1}
+\|\mathbf u\|_{H^s}^{q_2}
+\|\mathbf u\|_{H^s}^{q_3}
+\|\mathbf u\|_{H^s}
\bigr)
\|e\|^2.
\]
Since the coefficient on the right-hand side belongs to
$L^1(0,T)$, Gr\"onwall's lemma together with the initial 
condition $e(0)=0$ implies that $e\equiv0$. This completes 
the proof of uniqueness.

\medskip

\noindent\textbf{Step 2: Compactness.} Let
$c_n=\mathcal C(\mathbf u_n,\Theta_n)$. By Proposition \ref{prop:CH_exist}, the sequence $\{c_n\}$ is uniformly bounded in
\[
L^\infty(H^1)
\cap
L^2(H^2)\quad \text {and} \quad
\{\partial_t c_n\}
\subset
L^2(H^{-1}),
\quad
\{\mu_n\}
\subset
L^2(H^1).
\]
Therefore, the Aubin--Lions compactness theorem yields the existence of a subsequence, which we still denote by $\{c_n\}$, and a function $\widetilde c$ such that
\[
c_n
\to
\widetilde c
\quad\text{strongly in }
L^2(H^r),
\quad
r<2.
\]
By weak compactness, the remaining quantities also admit weakly convergent subsequences in their respective natural spaces.

\medskip
 \noindent\textbf{Step 3: Identification of the limit.}
We now verify that $\widetilde c$ satisfies the weak formulation
\eqref{weakofc}--\eqref{weakofmi} corresponding to the prescribed pair
$(\mathbf u,\Theta)$.

For the convective term, we write
\[
\mathbf u_n\cdot\nabla c_n-
\mathbf u\cdot\nabla\widetilde c
=(\mathbf u_n-\mathbf u)\cdot\nabla c_n
+\mathbf u\cdot\nabla(c_n-\widetilde c).
\]
The first term converges to zero by the strong convergence in
\(E_s(T)\) together with the uniform
\(L^\infty(H^1)\)-bound for \(c_n\), while the second converges by the strong convergence of
\(c_n\) in \(L^2(H^1)\).

Similarly,
\[
\|c_n^3-\widetilde c^{\,3}\|_{L^2(L^2)}
\le
C\bigl(\|c_n\|_{L^\infty (H^1)}^2
+\|\widetilde c\|_{L^\infty (H^1)}^2\bigr)
\|c_n-\widetilde c\|_{L^2(H^1)}
\longrightarrow0.
\]

Moreover, by \eqref{theonewiththeta},
\[
\|\hat\Theta_n-\hat\Theta\|_{L^2(W^{2,\infty})}
\le
C_\delta
\|\Theta_n-\Theta\|_{L^2(L^2)}
\longrightarrow0,
\]
and therefore
\[
(1-\hat\Theta_n)c_n
\to
(1-\hat\Theta)\widetilde c
\quad\text{strongly in }L^2(L^2).
\]

Lastly, since \(c_n(0)=c_0\) for every \(n\), the Lions--Magenes continuity theorem implies that
$c_n$ and $\widetilde c$ are in $C([0,T];L^2(\Omega))$,
and therefore $\widetilde c(0)=c_0$.

Passing to the limit in \eqref{weakofc}--\eqref{weakofmi} shows that
$\widetilde c$ is a weak solution of the prescribed-input problem corresponding to
$(\mathbf u,\Theta)$, with associated chemical potential given by the weak limit
$\widetilde\mu$ of $\{\mu_n\}$.
Since every subsequence admits a further subsequence converging to the same limit
$c=\mathcal C(\mathbf u,\Theta)$, the convergence holds for the full sequence. Hence,
\[
c_n\to c
\quad\text{strongly in }L^2(H^r),
\quad r<2,
\]
which completes the proof.
\end{proof}
\end{revisionfive}

\subsection{The Stokes Operator} \label{sec:Stokes}
\begin{revisionfive}
We now analyze the Stokes subsystem for a prescribed concentration $c$ arising from the Cahn--Hilliard analysis. More precisely, we establish the existence and regularity of weak solutions and study the continuity of the associated solution operator
\[
\mathcal S:c\mapsto\mathbf u.
\]

Before proving these results, we briefly collect the assumptions and notation used throughout this section. Outside the physical interval \([-0.5,0.5]\), we extend the viscosity and density functions to bounded positive functions satisfying
\[
\eta\in W^{1,\infty}(\mathbb R),
\quad
0<\eta_{\min}\le\eta(r)\le\eta_{\max}<\infty,
\]
and
\[
\rho\in C_b(\mathbb R),
\quad
0<\rho(r)\le\rho_{\max},
\qquad r\in\mathbb R.
\]
These extensions coincide with \eqref{DefEtac} and \eqref{DefRhoc} on the physical interval.

The weak formulation is posed in the divergence-free space
\(V_D\) and its closure
\(L^2_\sigma(\Omega)=\overline{V_D}^{L^2}\), introduced before Theorem
\ref{thm:main-existence}. We endow \(V_D\) with the norm
\[
\|\mathbf v\|_{V_D}:=\|\nabla\mathbf v\|_{L^2(\Omega)},
\]
which is equivalent to the \(H^1(\Omega)^d\)-norm by Korn's and Poincar\'e's inequalities. Since the pressure does not appear explicitly in this formulation, it is recovered afterwards by the standard de Rham--inf-sup theory. The boundary condition on \(\Gamma_N\) is the corresponding natural traction condition. If \(\Gamma_N=\varnothing\), the pressure is normalized by imposing zero mean; otherwise, the prescribed traction determines the additive constant. The mixed velocity--pressure finite-element formulation introduced later is the corresponding saddle-point formulation of the same Stokes subproblem.

Although the existence theory is naturally developed in the energy space
\(L^2(V_D)\), the continuous embedding
\[
V_D\hookrightarrow H^1(\Omega)^d
\hookrightarrow H^s(\Omega)^d,
\qquad \frac12<s<1,
\]
shows that every weak solution also belongs to
$\mathbf u\in L^2(H^s(\Omega)^d)$.

This choice of regularity is dictated by the Cahn--Hilliard analysis in Section \ref{sec:Cahn-Hilliard}, where the \(H^s\)-regularity of the velocity field is required to control the convective term.
\end{revisionfive}

\begin{revisionfive}

\begin{proposition}[Existence and regularity of weak solutions to the Stokes subproblem]
\label{prop:Stokes_exist}

Let \(c\in X_c(T)\), where \(X_c(T)\) is defined in
\eqref{Xcdef}, let the coefficient hypotheses above hold, and let
\(\mathbf u_0\in L^2_\sigma(\Omega)\). Then the prescribed-input Stokes problem
\eqref{NDStokes}
admits a unique weak solution satisfying
\[
\mathbf u\in L^\infty(0,T;L^2_\sigma(\Omega))
\cap
L^2(0,T;V_D),
\quad
\partial_t\mathbf u\in L^2(0,T;V_D').
\]
Moreover,
\[
\|\mathbf u\|_{L^\infty (L^2)}^2
+
\|\mathbf u\|_{L^2(V_D)}^2
+
\|\partial_t\mathbf u\|_{L^2(V_D')}^2
\le
C_{\alpha,\eta_{\min},\eta_{\max},\rho_{\max},\Omega}
\left(
\|\mathbf u_0\|^2
+
T\|\mathbf G\|^2
\right).
\]
Consequently,
\(\mathbf u\in E_s(T)\), where \(E_s(T)\) is defined in
\eqref{theonewithEs}, for every \(s\in(1/2,1)\).

\end{proposition}
\end{revisionfive}

\begin{revisionfive}
\begin{proof}
We construct a Galerkin approximation by projecting the weak formulation onto a finite-dimensional subspace of \(V_D\). Let
\(\{\boldsymbol\varphi_k\}_{k\ge1}\subset V_D\)
be an orthonormal basis of \(L^2_\sigma(\Omega)\) that is total in \(V_D\), and define
\[
V_n=\operatorname{span}\{\boldsymbol\varphi_1,\ldots,\boldsymbol\varphi_n\}.
\]

We seek approximate solutions of the form
\[
\mathbf u_n(t)
=
\sum_{j=1}^n
d_j(t)\boldsymbol\varphi_j,
\]
which satisfy, for every \(\boldsymbol\varphi\in V_n\),
\begin{equation}\label{StokesWeak}
\alpha(\partial_t\mathbf u_n,\boldsymbol\varphi)
+
(2\eta(c)\mathcal E(\mathbf u_n),\mathcal E(\boldsymbol\varphi))
=
(\rho(c)\mathbf G,\boldsymbol\varphi).
\end{equation}

The initial condition is prescribed by
\[
\mathbf u_n(0)=P_n\mathbf u_0,
\]
where \(P_n\) denotes the \(L^2\)-projection onto \(V_n\). Since the coefficients are measurable in time and uniformly bounded, while the mass matrix is positive definite, standard ODE theory guarantees the existence of a unique solution on \([0,T]\).

We next derive estimates that are uniform with respect to \(n\).

Testing \eqref{StokesWeak} by \(\mathbf u_n\), using Korn's and Poincar\'e's inequalities together with Young's inequality, yields
\[
\frac\alpha2\frac{d}{dt}\|\mathbf u_n\|^2
+
c\eta_{\min}\|\mathbf u_n\|_{V_D}^2
\le
C\rho_{\max}^2\|\mathbf G\|^2.
\]
Consequently,
\[
\alpha\|\mathbf u_n\|_{L^\infty(L^2)}^2
+
\eta_{\min}\|\mathbf u_n\|_{L^2(V_D)}^2
\le
C\left(
\alpha\|\mathbf u_0\|_{L^2}^2
+
T\rho_{\max}^2\|\mathbf G\|^2
\right).
\]

Furthermore, \eqref{StokesWeak} together with the upper bound for \(\eta\) implies
\[
\|\partial_t\mathbf u_n\|_{V_D'}
\le
\frac C\alpha
\left(
\eta_{\max}\|\mathbf u_n\|_{V_D}
+
\rho_{\max}\|\mathbf G\|
\right),
\]
and therefore
\(\{\partial_t\mathbf u_n\}\) is uniformly bounded in
\(L^2(V_D')\).

By weak compactness, there exists a subsequence, still denoted by
\(\{\mathbf u_n\}\), such that
\[
\mathbf u_n
\stackrel{*}{\rightharpoonup}
\mathbf u
\quad
\text{in }
L^\infty(L^2_\sigma),\quad
\mathbf u_n
\rightharpoonup
\mathbf u
\quad
\text{in }
L^2(V_D), \quad
\partial_t\mathbf u_n
\rightharpoonup
\partial_t\mathbf u
\quad
\text{in }
L^2(V_D').
\]

Since \(c\) is prescribed and the coefficients
\(\eta(c)\) and \(\rho(c)\) are bounded, passage to the limit in the linear weak formulation is immediate.

By the Lions--Magenes continuity theorem,
$\mathbf u\in C([0,T];L^2_\sigma(\Omega)),$
and therefore
$\mathbf u(0)=\mathbf u_0.$

Finally, uniqueness follows by testing the difference of two weak solutions by their difference. The energy estimates pass to the limit by weak lower semicontinuity, completing the proof.
\end{proof}
\end{revisionfive}

Having established existence and regularity, we now prove continuity of the solution operator with respect to the prescribed concentration field.

\begin{revisionfive}
\begin{proposition}[Continuity of the Stokes solution map]
\label{prop:Stokes_cont}

Let \(r\in(3/2,2)\), and suppose
\[
c_n\to c
\quad\text{strongly in }L^2(0,T;H^r(\Omega)).
\]

Let $\mathbf u_n=\mathcal S(c_n)$ and $\mathbf u=\mathcal S(c)$. Then, for every \(s\in(1/2,1)\),
\[
\mathbf u_n\to\mathbf u
\quad\text{strongly in }L^2(0,T;H^s(\Omega)^d).
\]
\end{proposition}

\begin{proof}
Let $\mathbf u_n=\mathcal S(c_n).$ By Proposition \ref{prop:Stokes_exist}, the sequence
\(\{\mathbf u_n\}\) is uniformly bounded in
\[
L^\infty(L^2_\sigma)
\cap
L^2(V_D),\quad
\{\partial_t\mathbf u_n\}
\subset
L^2(V_D').
\]
Therefore, the Aubin--Lions compactness theorem yields the existence of a subsequence, which we still denote by
\(\{\mathbf u_n\}\), and a function
\(\widetilde{\mathbf u}\) such that
\[
\mathbf u_n
\to
\widetilde{\mathbf u}
\quad\text{strongly in }
L^2(L^2(\Omega)^d),
\]
while
\[
\mathbf u_n
\rightharpoonup
\widetilde{\mathbf u}
\quad\text{weakly in }
L^2(V_D).
\]

We now verify that \(\widetilde{\mathbf u}\) satisfies the prescribed-input Stokes problem corresponding to \(c\). Since
\(r>3/2\), the embedding
\[
H^r(\Omega)\hookrightarrow L^\infty(\Omega)
\]
holds. Consequently, the Lipschitz continuity of \(\eta\) yields
\[
\eta(c_n)
\to
\eta(c)
\quad\text{strongly in }
L^2(L^\infty).
\]

Moreover, after passing to a subsequence,
\(c_n\to c\) almost everywhere. Since \(\rho\) is bounded and continuous, dominated convergence implies
\[
\rho(c_n)
\to
\rho(c)
\quad\text{strongly in }
L^2(L^2).
\]

It therefore remains to identify the limit in the weak formulation. For every
\(\boldsymbol\varphi\in C_c^\infty(0,T;V_D)\),
\begin{align*}
&\left|
\int_0^T
((\eta(c_n)-\eta(c))
\mathcal E(\mathbf u_n),
\mathcal E(\boldsymbol\varphi))
\,dt
\right|
\\[0.3em]
&\qquad\le
\|\eta(c_n)-\eta(c)\|_{L^2(L^\infty)}
\|\mathcal E(\mathbf u_n)\|_{L^2(L^2)}
\|\mathcal E(\boldsymbol\varphi)\|_{L^\infty (L^2)}
\longrightarrow0.
\end{align*}

The remaining viscous term converges by the weak convergence of
\(\mathbf u_n\), while the forcing term converges by the strong convergence of
\(\rho(c_n)\). The time-derivative term passes to the limit by weak convergence in
\(L^2(V_D')\). Moreover, by the Lions--Magenes continuity theorem,
\[
\widetilde{\mathbf u}\in C([0,T];L^2_\sigma(\Omega))
\quad\text{and}\quad
\widetilde{\mathbf u}(0)=\mathbf u_0.
\]
Consequently,
\(\widetilde{\mathbf u}\) satisfies the prescribed-input Stokes problem corresponding to \(c\).

By the uniqueness established in Proposition \ref{prop:Stokes_exist},
\[
\widetilde{\mathbf u}
=
\mathcal S(c).
\]
Since every subsequence admits a further subsequence converging to the same limit, the full sequence converges strongly in
\[
L^2(L^2(\Omega)^d).
\]

Finally,
\[
\|\mathbf u_n-\widetilde{\mathbf u}\|_{L^2(H^s)}^2
\le
C
\|\mathbf u_n-\widetilde{\mathbf u}\|_{L^2(L^2)}^{2(1-s)}
\|\mathbf u_n-\widetilde{\mathbf u}\|_{L^2(H^1)}^{2s}.
\]
Since \(\{\mathbf u_n\}\) is uniformly bounded in
\(L^2(0,T;H^1(\Omega)^d)\), it follows that
\[
\mathbf u_n
\to
\widetilde{\mathbf u}
\quad\text{strongly in }
L^2(H^s(\Omega)^d),
\]
completing the proof.
\end{proof}
\end{revisionfive}

\subsection{The Fixed-Point Argument}\label{sec:fixedpoint}

\begin{revisionfive}

We now combine the existence and continuity results established for the heat, Cahn--Hilliard, and Stokes subproblems to prove Theorem \ref{thm:main-existence}. To this end, we apply the Schauder fixed-point theorem to the operator
\[
\mathcal T(\mathbf u)
:=
\mathcal S\bigl(\mathcal C(\mathbf u,\mathcal H(\mathbf u))\bigr),
\]
acting on the Banach space
\[
E_s(T)=L^2(0,T;Y_s) \quad \text{ with } \quad
Y_s
=
\{\mathbf v\in H^s(\Omega)^d:
\nabla\cdot\mathbf v=0,\;
\operatorname{Tr}\mathbf v=0
\text{ on }\Gamma_D\},~~
s \in (1/2,1).
\]

\begin{proof}[Proof of Theorem \ref{thm:main-existence}]
Choose \(R>0\) and define
\[
K:=
\{\mathbf u\in E_s(T):
\|\mathbf u\|_{E_s(T)}
\le R\}.
\]
Since \(Y_s\) is a closed subspace of \(H^s(\Omega)^d\), the set \(K\) is nonempty, closed, bounded, and convex. To apply the Schauder fixed-point theorem, it remains to show that
\(\mathcal T\) maps \(K\) into itself, is continuous, and has relatively compact image.

We begin by verifying that the operator \(\mathcal T\) is well defined. Let
\(T\le1\) and let \(\mathbf u\in K\). Proposition
\ref{heatwell} defines $\Theta=\mathcal H(\mathbf u),$
and provides the estimate
$\|\Theta\|_{X_\Theta(T)}
\le
M_\Theta(R)$, while Proposition \ref{pro:maxprinciple} guarantees that
\(\Theta\ge0\). The bound on \(\mathbf u\) together with this estimate satisfies the hypotheses of Proposition \ref{prop:CH_exist}. Consequently, $c=\mathcal C(\mathbf u,\Theta).$ Finally, Proposition \ref{prop:Stokes_exist} defines $\mathcal T(\mathbf u)
=
\mathcal S(c)$. We therefore choose \(T>0\) sufficiently small so that all of the preceding existence results apply simultaneously.

\medskip

\noindent\textbf{Self-map property.}
By the estimate in Proposition \ref{prop:Stokes_exist},
\[
\|\mathcal T(\mathbf u)\|_{E_s(T)}^2
\le
C_{\rm emb}
\|\mathcal T(\mathbf u)\|_{L^2(V_D)}^2
\le
C_0\alpha\|\mathbf u_0\|^2
+
C_1T\|\mathbf G\|^2,
\]
where \(C_0\) and \(C_1\) are independent of \(\mathbf u\in K\). 

We first choose \(R>0\) sufficiently large so that
$R^2>
2C_0\alpha
\|\mathbf u_0\|^2,$
and, if necessary, enlarge \(R\) to account for the initial data. We then choose \(T>0\) sufficiently small so that $C_1T\|\mathbf G\|^2
\le
\frac{R^2}{2}.$ Combining the two estimates yields
\[
\|\mathcal T(\mathbf u)\|_{E_s(T)}
\le
R,
\]
and therefore $\mathcal T(K)\subseteq K.$

\medskip
\noindent\textbf{Continuity.}
Let
\[
\mathbf u_n
\to
\mathbf u
\quad\text{strongly in }
E_s(T).
\]

By Proposition \ref{prop:H-continuity},
\[
\Theta_n
:=
\mathcal H(\mathbf u_n)
\to
\Theta
:=
\mathcal H(\mathbf u)
\quad
\text{strongly in }
L^2(L^2).
\]
Moreover, Proposition \ref{heatwell} provides a uniform bound for
\(\{\Theta_n\}\)
in
\(X_\Theta(T)\), with $X_\Theta(T)$ given by \eqref{Xthetadef}. Therefore, Proposition
\ref{prop:CH_continuity_full}
yields
\[
c_n
:=
\mathcal C(\mathbf u_n,\Theta_n)
\to
c
:=
\mathcal C(\mathbf u,\Theta)
\quad
\text{strongly in }
L^2(H^r),
\]
for every
\(r\in(3/2,2)\).
Finally, Proposition
\ref{prop:Stokes_cont}
yields
\[
\mathcal T(\mathbf u_n)
=
\mathcal S(c_n)
\to
\mathcal S(c)
=
\mathcal T(\mathbf u)
\quad
\text{strongly in }
E_s(T).
\]

\medskip
\noindent\textbf{Relative compactness.}
Let
\(\{\mathbf u_n\}\subset K\).
By Proposition
\ref{prop:Stokes_exist},
the sequence
\(\{\mathcal T(\mathbf u_n)\}\)
is uniformly bounded in
$L^2(V_D),$
while
\[
\{\partial_t\mathcal T(\mathbf u_n)\}
\subset
L^2(V_D')
\]
is uniformly bounded.
Since the embedding $V_D
\hookrightarrow
Y_s$ is compact and
$Y_s
\hookrightarrow
V_D'$ is continuous, the Aubin--Lions compactness theorem implies that
$\mathcal T(K)$ is relatively compact in
\(E_s(T)\).

Since \(\mathcal T\) is well defined, maps \(K\) into itself, is continuous, and has relatively compact image, all the hypotheses of the Schauder fixed-point theorem are satisfied. Therefore, there exists
$\mathbf u^*
\in
K$ such that $\mathcal T(\mathbf u^*)
=
\mathbf u^*.$ Setting
\[
\Theta^*
=
\mathcal H(\mathbf u^*),
\qquad
c^*
=
\mathcal C(\mathbf u^*,\Theta^*),
\]
and letting
\(\mu^*\)
denote the corresponding chemical potential, we obtain the weak solution asserted in Theorem
\ref{thm:main-existence}. The regularity and initial data follow directly from Propositions
\ref{heatwell},
\ref{prop:CH_exist},
and
\ref{prop:Stokes_exist}.
\end{proof}
\end{revisionfive}

\begin{revisionfive}
With local existence established for the regularized auxiliary system, we now turn to the numerical approximation of the corresponding unregularized quasi-static model. The parameter \(\alpha>0\) and the spatial mollification of the temperature are used only in the analytical construction above. In the numerical experiments we set \(\alpha=0\) and use the original temperature field. The simulations illustrate qualitative features of that related model and are not presented as a numerical validation of Theorem \ref{thm:main-existence}.
\end{revisionfive}

\section{Numerical experiments}
\label{sec:numerics}

This section reports computations for simplified quasi-static numerical variants of the model implemented in FreeFEM++, with $\alpha=0$ and without temperature mollification.
Experiments~1--3 use either a prescribed temperature or no thermal solve. In Experiment~4, the temperature is advanced by a diffusion-only update and the material contrasts are allowed to depend piecewise on temperature (see details below). The numerical experiments are intended as qualitative, parameter-specific illustrations and are not presented as validation of the analytical existence result.

The scalar fields $c$, $\mu$, and $\Theta$ are approximated by continuous $P_1$ finite elements, while the Stokes velocity and pressure are approximated by the Taylor--Hood pair $P_2$--$P_1$. Time integration is first order and sequential: when present, heat diffusion is updated first, followed by the quasi-static Stokes solve and the Cahn--Hilliard update. The bulk derivative in
the chemical potential is evaluated at the previous time level, so every algebraic system solved in the computations is linear. In the FreeFEM++ implementation we use the sparse direct solver \texttt{UMFPACK}. No Newton tolerance, damping parameter, or nonlinear stopping criterion is applicable. In the all-Dirichlet cavity problem, the pressure nullspace is regularized by $\varepsilon_p=10^{-10}$, whereas $\varepsilon_p=0$ in the mixed Dirichlet--traction problems.

\revfive{We use the FreeFEM++ boundary labels
\[
\Gamma_1=\{(x,0)\},\qquad
\Gamma_2=\{(L_x,y)\},\qquad
\Gamma_3=\{(x,L_y)\},\qquad
\Gamma_4=\{(0,y)\}.
\]
The random initial data in Experiments~2--4 are Bernoulli nodal data,
\begin{equation}
 c_h^0(x_i)=
 \begin{cases}
 -0.5, & \text{with probability }0.7,\\
 +0.5, & \text{with probability }0.3,
 \end{cases}
 \qquad \mathbb E[c_h^0]=-0.2.
 \label{eq:bernoulli_initial_data}
\end{equation}
The nodal levels have magnitude $0.5$ and total range $1$.  Relative to the expected mean $-0.2$, the two deviations are $-0.3$ and $+0.7$. These data are Bernoulli nodal values rather than a symmetric small-amplitude perturbation.}

\subsection{Spatial discretization}

Let $\mathcal T_h$ be a shape-regular triangulation of $\Omega$ with mesh size
$h$. We use the scalar $P_1$ space
\[
Q_h:=\bigl\{q_h\in H^1(\Omega):q_h|_K\in P_1(K),\ 
K\in\mathcal T_h\bigr\}
\]
for $c_h$, $\mu_h$, $\Theta_h$, and $p_h$.  For the velocity, let
\[
V_{h,0}:=\bigl\{\mathbf v_h\in H^1(\Omega)^d:
\mathbf v_h|_K\in P_2(K)^d,\ \mathbf v_h=\mathbf0
\text{ on }\Gamma_D\bigr\},
\]
and let $V_{h,\mathbf g_D}:=\mathbf g_{D,h}+V_{h,0}$ denote the corresponding affine trial space when nonhomogeneous Dirichlet data are prescribed. Thus the velocity trial function belongs to $V_{h,\mathbf g_D}$ and the test function
to $V_{h,0}$. The pair $P_2$--$P_1$ is the standard Taylor--Hood pair. In the all-Dirichlet cavity problem the pressure is defined only up to an additive constant and is fixed numerically by the penalty described below. In the mixed Dirichlet--traction problems the traction condition fixes the pressure level.

For the thermal problem we also introduce
\[
Q_{h,\Theta_D}
:= \bigl\{ \theta_h \in Q_h \;:\; \theta_h=\Theta_D
\text{ on } \Gamma_\Theta \bigr\},
\qquad
Q_{h,0}^{\Theta}
:= \bigl\{ \omega_h \in Q_h \;:\; \omega_h=0
\text{ on } \Gamma_\Theta \bigr\}.
\]

\subsection{Time discretization and coupling strategy}

Given the discrete fields at time $t^m$, the active subproblems are applied sequentially. When present, the heat equation is solved first, otherwise the prescribed temperature is used. The experiment-specific material coefficients are then updated and the Stokes problem is solved when flow is active. In the purely diffusive experiment we set $\mathbf u^{m+1}=\mathbf0$. Finally, the Cahn--Hilliard system is advanced. The three substeps are stated below in weak form.

\subsubsection{Heat-diffusion update used in the thermal runs}

\revfive{Given $\Theta_h^m\in Q_h$, seek
$\Theta_h^{m+1}\in Q_{h,\Theta_D}$ such that
\begin{equation}
\int_\Omega \frac{\Theta_h^{m+1}-\Theta_h^m}{\Delta t}\,\omega_h\,dx
+\frac{1}{Pe_\Theta}\int_\Omega
\nabla\Theta_h^{m+1}\cdot\nabla\omega_h\,dx=0
\label{Heat_discrete_revised}
\end{equation}
for every $\omega_h\in Q_{h,0}^{\Theta}$. The homogeneous Neumann condition
on $\Gamma_{\Theta_N}$ is natural. In Experiment~4,
$1/Pe_\Theta=0.4$, and hence $Pe_\Theta=2.5$. No advective heat term is included in the experiment.}

\subsubsection{Stokes subproblem}

Given $c^m$ and, in Experiment~4, the already computed temperature
$\Theta^{m+1}$, we define the experiment-specific coefficient fields
$\rho_h^m$ and $\eta_h^m$ by the laws stated in the corresponding experiment.
We seek $(\mathbf u^{m+1},p^{m+1})\in V_{h,\mathbf g_D}\times Q_h$ such that
\begin{align}
\int_\Omega 2\eta_h^m\mathcal E(\mathbf u^{m+1}):
\mathcal E(\mathbf v_h)\,dx
-\int_\Omega p^{m+1}\nabla\!\cdot\mathbf v_h\,dx
&=\int_\Omega G\rho_h^m\widehat{\mathbf g}\cdot\mathbf v_h\,dx,
\label{Stokes_discrete_momentum}\\[0.3em]
\int_\Omega q_h\nabla\!\cdot\mathbf u^{m+1}\,dx
+\varepsilon_p\int_\Omega p^{m+1}q_h\,dx&=0,
\label{Stokes_discrete_continuity}
\end{align}
for every $(\mathbf v_h,q_h)\in V_{h,0}\times Q_h$.  We use
$\varepsilon_p=10^{-10}$ only in the all-Dirichlet cavity problem and
$\varepsilon_p=0$ in Experiments~3--4.  The traction-free condition is encoded naturally, whereas velocity Dirichlet data are imposed strongly.  This is a linear saddle-point system at every time step.

\subsubsection{Linearized Cahn--Hilliard update used in Experiments~2--4}
\revfive{Let $I_h:C(\overline\Omega)\to Q_h$ denote the nodal $P_1$ interpolation operator and define
\[
 g_h^m:=I_h\!\left[(c_h^m)^3\right].
\]
Given $c_h^m\in Q_h$, $\mathbf u_h^{m+1}\in V_{h,\mathbf g_D}$, and the
available coefficient $\beta_h^{m+1}$ (computed from $\Theta_h^{m+1}$ in
Experiment~4 and prescribed in Experiments~2--3), seek
$(c_h^{m+1},\mu_h^{m+1})\in Q_h\times Q_h$ such that
\begin{align}
\int_\Omega \frac{c_h^{m+1}-c_h^m}{\Delta t}\,\psi_h\,dx
+\int_\Omega(\mathbf u_h^{m+1}\cdot\nabla c_h^{m+1})\psi_h\,dx
+\frac1{Pe}\int_\Omega\nabla\mu_h^{m+1}\cdot\nabla\psi_h\,dx&=0,
\label{CH_discrete_1_revised}\\
\int_\Omega \mu_h^{m+1}\phi_h\,dx
-Ch^2\int_\Omega\nabla c_h^{m+1}\cdot\nabla\phi_h\,dx
-\int_\Omega\left(8g_h^m-2(1-\beta_h^{m+1})c_h^m\right)\phi_h\,dx&=0
\label{CH_discrete_2_revised}
\end{align}
for all $\psi_h,\phi_h\in Q_h$. The second equation is the weak
finite-element representation of the formal relation
\[
\mu_h^{m+1}=8g_h^m-2(1-\beta_h^{m+1})c_h^m
-Ch^2\Delta c_h^{m+1}.
\]
The bulk derivative is evaluated at time level $m$, so
\eqref{CH_discrete_1_revised}--\eqref{CH_discrete_2_revised} form one linear mixed system. Experiment~1 uses the same explicit nodal treatment of the cubic term, with the phase scaling and interfacial coefficient stated in Section~\ref{sec:exp1-cavity}.}

\subsection{Sequential time-stepping algorithm}

At a time level at which a thermal solve is present, one step
$t^m\mapsto t^{m+1}$ is performed as follows:
\begin{enumerate}
  \item compute $\Theta^{m+1}$ from
  \eqref{Heat_discrete_revised}; in isothermal experiments, use the prescribed
  constant temperature instead;
  \item define the experiment-specific fields $\rho_h^m$ and $\eta_h^m$ from
  $c^m$ and, in Experiment~4, $\Theta^{m+1}$;
  \item compute $(\mathbf u^{m+1},p^{m+1})$ from
  \eqref{Stokes_discrete_momentum}--\eqref{Stokes_discrete_continuity}; in the
  purely diffusive experiment set $\mathbf u^{m+1}=\mathbf0$;
  \item compute $(c^{m+1},\mu^{m+1})$ from
  \eqref{CH_discrete_1_revised}--\eqref{CH_discrete_2_revised}.
\end{enumerate}
\revfive{The splitting is formally first order in time. Every subproblem is linear, because the material coefficients and the bulk derivative are evaluated from already available fields. For $\mathbf u_h=\mathbf0$, choosing $\psi_h\equiv1$ in \eqref{CH_discrete_1_revised} gives conservation of the finite-element mass up to direct-solver roundoff. More generally, when $\varepsilon_p=0$, the discrete continuity equation holds for every $q_h\in Q_h$. Together with $\mathbf u_h\cdot\mathbf n=0$ on $\partial\Omega$, this implies conservation of the discrete mass. No unconditional energy-stability result is claimed for the linearized Cahn--Hilliard update, and no global discrete energy law is claimed for the non-isothermal sequential scheme.
}

\medskip
\noindent

\begin{revisionfive}
{\subsection{Stand-alone verification of the isothermal diffusive
Cahn--Hilliard update}
\label{sec:CH_standalone_verification}

To separate numerical-method verification from the qualitative coupled examples, we performed an independent $P_1$ finite-element assembly of the linearized isothermal Cahn--Hilliard update \eqref{CH_discrete_1_revised}--\eqref{CH_discrete_2_revised} with $\mathbf u=0$.  The test uses
\[
\Omega=(0,1)^2,\quad 100\times100\text{ subdivisions},\quad
Pe=1000,\quad Ch=0.01,\quad \Theta=0.3,\quad
\Delta t=0.01,\quad t_{\rm fin}=2,
\]
and the Bernoulli initialization \eqref{eq:bernoulli_initial_data}. 
The Bernoulli field in this verification was generated with the NumPy PCG64 generator and seed $20260803$. The same integer seed in NumPy and FreeFEM++ does not imply an identical random realization. 
We monitor the normalized mass defect
\[
\delta_M^m=
\frac{\left|\int_\Omega c_h^m\,dx-\int_\Omega c_h^0\,dx\right|}
{\max\left\{1,\left|\int_\Omega c_h^0\,dx\right|\right\}}
\]
and the discrete free energy
\[
\mathcal E_h(c_h^m)=
\int_\Omega\left(
2(c_h^m)^4-(1-\Theta)(c_h^m)^2
+\frac{Ch^2}{2}|\nabla c_h^m|^2
\right)dx.
\]
The maximum mass defect over $0\le t\le2$ is
$6.94\times10^{-16}$.  The energy decreases from $0.380337$ to
$-0.034942$, with no positive time increment in this test.  The relative algebraic residual is defined by
\[
 r^m=\frac{\|A_hx^m-b^m\|_2}{\max\{1,\|b^m\|_2\}},
\]
and its maximum over the run is $4.97\times10^{-18}$. These results document the behaviour of this specific test. {All $L^2$ norms were evaluated with the consistent $P_1$ mass matrix, and the quartic bulk-energy term was integrated with a degree-four triangle quadrature rule.}

\begin{figure}[h!]
\centering
\includegraphics[width=0.63\textwidth]{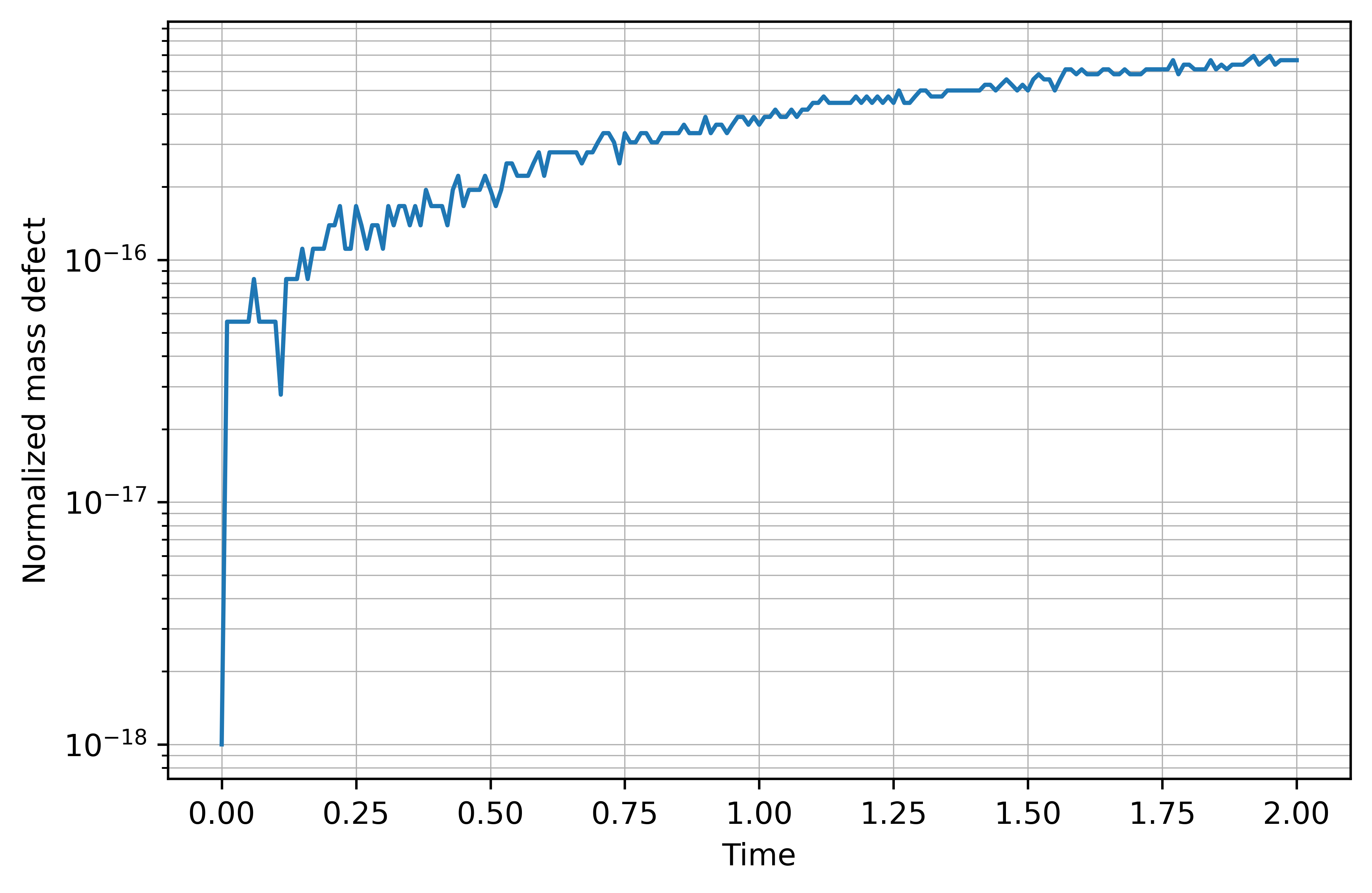}
\caption{Scaled mass defect for the stand-alone isothermal diffusive
Cahn--Hilliard verification.  The defect remains at floating-point roundoff.}
\label{fig:ch_mass_diagnostic}
\end{figure}

\begin{figure}[h!]
\centering
\includegraphics[width=0.63\textwidth]{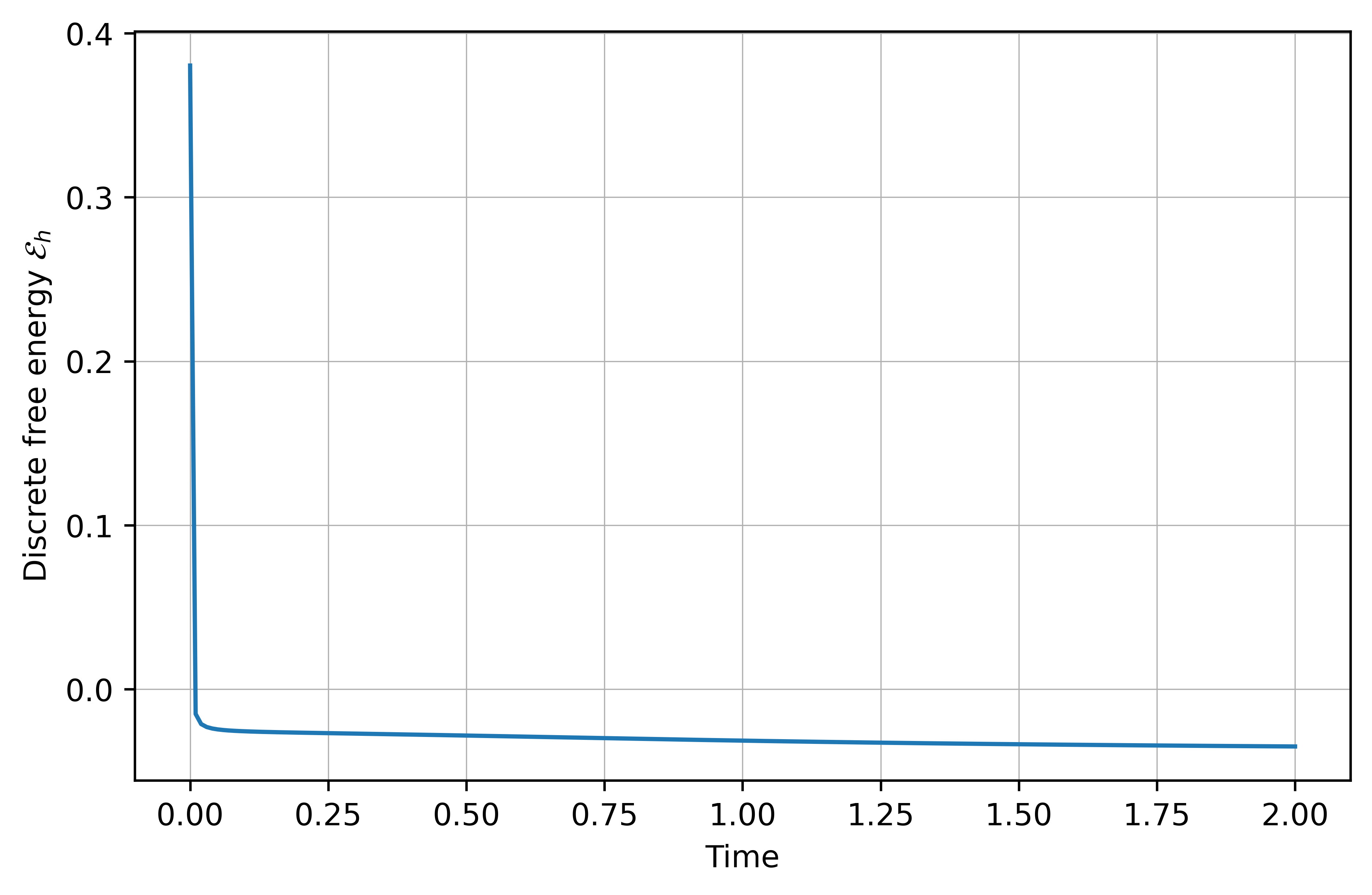}
\caption{Discrete free energy in the same stand-alone verification.  The
energy is monotone decreasing for the tested parameter set.}
\label{fig:ch_energy_diagnostic}
\end{figure}

For time refinement, we fix $t=0.5$ and compare successive time-step
solutions on the $100\times100$ mesh. The last row of Table~\ref{tab:ch_time_refinement} compares
$\Delta t=2.5\times10^{-3}$ with an additional solution computed using
$\Delta t=1.25\times10^{-3}$. The differences and observed rates are consistent with the formal first-order time discretization.

\begin{table}[h!]
\centering
\begin{tabular}{ccc}
\hline
$\Delta t$ &
$\|c_{\Delta t}(0.5)-c_{\Delta t/2}(0.5)\|_{L^2}$ & rate\\
\hline
$2.0\times10^{-2}$ & $6.3875\times10^{-3}$ & --\\
$1.0\times10^{-2}$ & $3.0820\times10^{-3}$ & $1.051$\\
$5.0\times10^{-3}$ & $1.4955\times10^{-3}$ & $1.043$\\
$2.5\times10^{-3}$ & $7.2791\times10^{-4}$ & $1.039$\\
\hline
\end{tabular}
\caption{Time-step refinement for the stand-alone isothermal diffusive
Cahn--Hilliard update.}
\label{tab:ch_time_refinement}
\end{table}

For mesh refinement, we use the smooth initial datum
\[
c_0(x,y)=-0.2+0.05\cos(2\pi x)\cos(2\pi y),
\]
with $Pe=1000$, $Ch=0.04$, $\Theta=0.3$, $\Delta t=5\times10^{-4}$,
and $t_{\rm fin}=0.05$.  The $200\times200$ solution is used as the reference. Each coarser continuous $P_1$ solution is evaluated at the nodes of the nested reference mesh before the consistent-mass $L^2$ difference is computed (Table~\ref{tab:ch_mesh_refinement}). The observed rates are approximately second order for this smooth test. 

\begin{table}[h!]
\centering
\begin{tabular}{ccc}
\hline
subdivisions per direction & $L^2$ error vs. $200\times200$ & rate\\
\hline
$25$  & $3.1688\times10^{-4}$ & --\\
$50$  & $7.6371\times10^{-5}$ & $2.053$\\
$100$ & $1.5940\times10^{-5}$ & $2.260$\\
\hline
\end{tabular}
\caption{Mesh refinement for a smooth stand-alone Cahn--Hilliard test.}
\label{tab:ch_mesh_refinement}
\end{table}

Finally, Table~\ref{tab:ch_limited_sensitivity} gives a limited one-factor check of the two parameters that enter the stand-alone Cahn--Hilliard step. Here
\[
\sigma_c(t)=\left(|\Omega|^{-1}\int_\Omega
|c_h(t)-\overline c_h|^2dx\right)^{1/2}.
\]
At fixed time, lower $Pe$ or lower $Ch$ produces a more developed separated
state in this diagnostic.  The mass defect remains below $9\times10^{-16}$ in
all listed runs.

All sensitivity runs use the $100\times100$ mesh, $\Theta=0.3$,
$\Delta t=0.01$, $t_{\rm fin}=2$, and the same Bernoulli realization as the baseline verification. At fixed time, decreasing $Pe$ increases the mobility $1/Pe$, while decreasing $Ch$ reduces the interfacial penalty. Both changes increase $\sigma_c(2)$ over the tested values. Again, these results are descriptive and are not used to claim mesh-independent sensitivity.

\begin{table}[h!]
\centering
\begin{tabular}{cccc}
\hline
varied parameter & value & fixed companion parameter & $\sigma_c(2)$\\
\hline
$Pe$ & $500$  & $Ch=0.01$ & $0.25687$\\
$Pe$ & $1000$ & $Ch=0.01$ & $0.22804$\\
$Pe$ & $2000$ & $Ch=0.01$ & $0.18037$\\
$Ch$ & $0.0075$ & $Pe=1000$ & $0.25324$\\
$Ch$ & $0.0100$ & $Pe=1000$ & $0.22804$\\
$Ch$ & $0.0150$ & $Pe=1000$ & $0.13567$\\
\hline
\end{tabular}
\caption{Limited one-factor sensitivity of the stand-alone isothermal
Cahn--Hilliard test.  This is not a sensitivity study of the fully coupled
model.}
\label{tab:ch_limited_sensitivity}
\end{table}

A systematic sweep over $Pe_\Theta$, $G$, the density ratio, and the viscosity ratio would require a new fully coupled computational campaign.}
    \end{revisionfive}

We now turn to the numerical experiments.

\subsection{Experiment 1: qualitative lid-driven-cavity benchmark comparison}
\label{sec:exp1-cavity}
This test is a qualitative implementation check rather than validation against the full model of Khatavkar et al. The reference momentum equation contains capillary forcing, whereas the present Stokes solve does not. Consequently, a pointwise or norm-based error would compare solutions of different PDE systems and would not provide a meaningful quantitative validation measure.

For this benchmark only, we use a benchmark-specific phase field $q_h$,
initialized at the levels $\pm1$. It is not identified with the nondimensional
order parameter $c_h$ used in Experiments~2--4. The implemented chemical-
potential update is
\[
\mu_{{\rm cav},h}^{m+1}
=8I_h[(q_h^m)^3]-2(1-\Theta)q_h^m
-\varepsilon_{\rm cav}^2\Delta q_h^{m+1},
\qquad
\varepsilon_{\rm cav}^2=5\times10^{-4}.
\]
The initial condition is
\[
q_h^0(x,y)=\begin{cases}1,&y<0.5,\\-1,&y\ge0.5.\end{cases}
\]
This is equivalent to the conventional top-lid orientation by the reflection $y\mapsto1-y$.
Homogeneous no-flux conditions are used for $q$ and $\mu$.

The parameters are
\[
Pe=10^4,\qquad \Theta=0.3,\qquad \Delta t=0.01,
\qquad t_{\rm fin}=5,
\]
and the interfacial-gradient coefficient in this benchmark is
\[
\varepsilon_{\rm cav}^2
=\tfrac12\,Ch_{\rm code}=5\times10^{-4},
\qquad Ch_{\rm code}=10^{-3}.
\]

\begin{figure}[h!]
  \centering

  \begin{subfigure}[t]{0.48\textwidth}
    \centering
    \includegraphics[width=\textwidth]{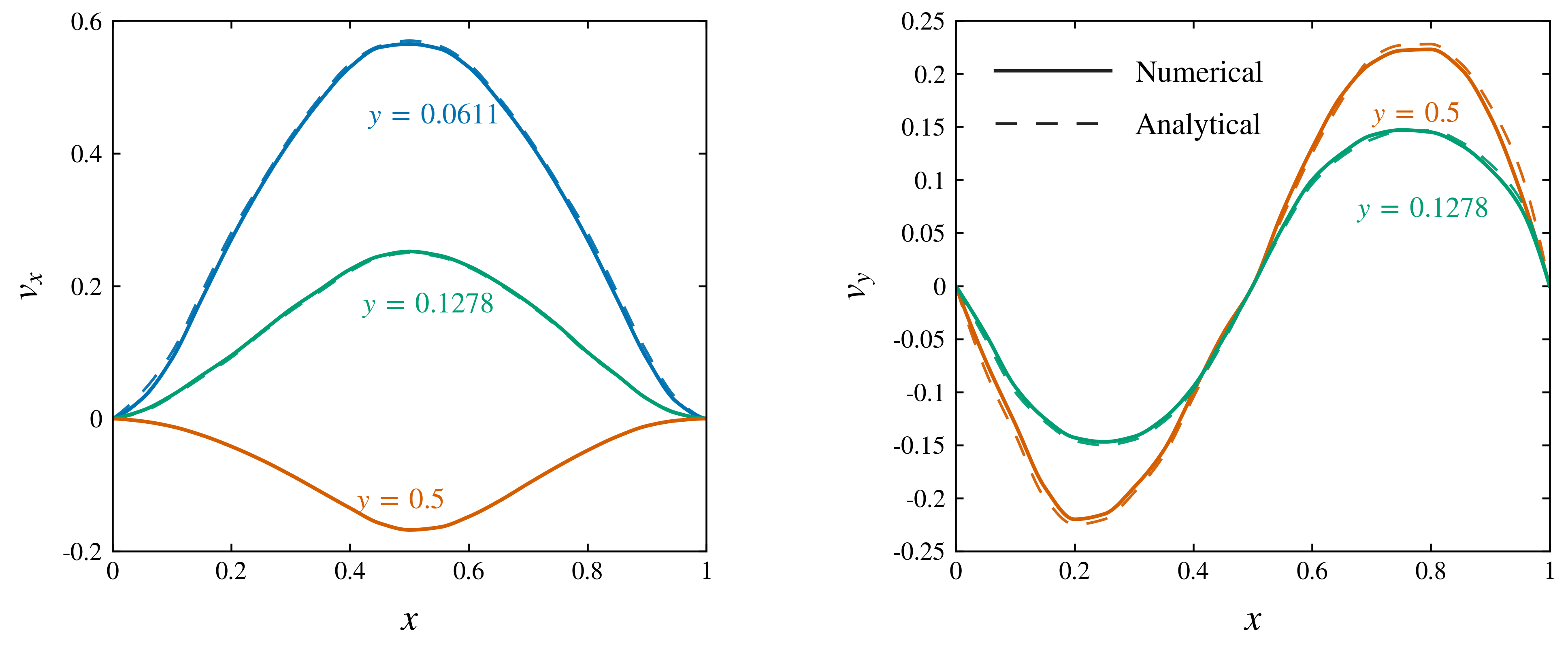}
    \caption{}
    \label{fig:khat_v_a}
  \end{subfigure}
  \hfill
  \begin{subfigure}[t]{0.48\textwidth}
    \centering
    \includegraphics[width=\textwidth]{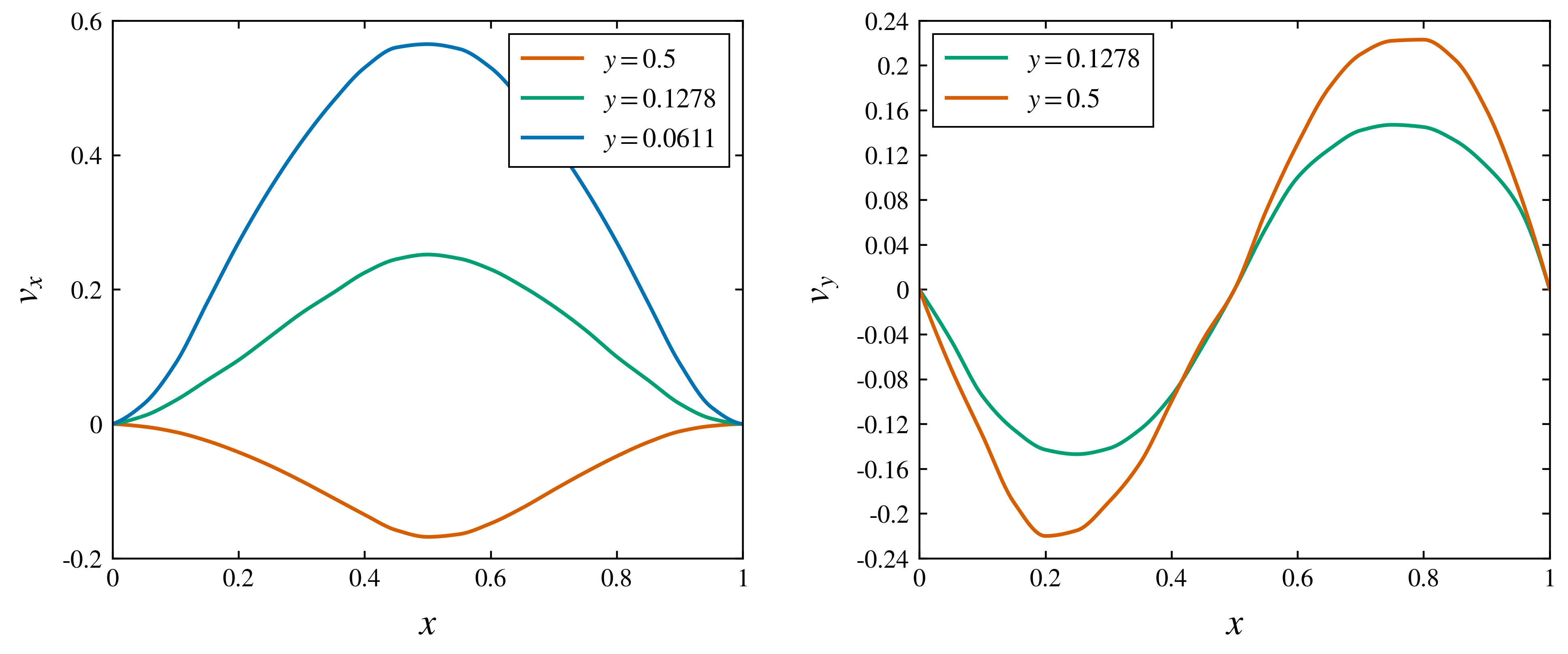}
    \caption{}
    \label{fig:khat_v_b}
  \end{subfigure}

 \caption[Comparison of the present and reference velocity profiles.]
{Comparison of the present velocity profiles and analytical reference curves (a) with the profiles reported by Khatavkar et al.\ (b), \revfive{reference curves digitized and replotted} from~\cite{Khat}.}
  \label{fig:khat_v}
\end{figure}

The benchmark viscosity is
\[
\eta_{\rm cav}(q)=\frac{q+1}{2}+\frac{1-q}{2\lambda},
\]
so that $\eta_{\rm cav}(1)=1$ and
$\eta_{\rm cav}(-1)=1/\lambda$.
\revfive{We use a $72\times72$ mesh and viscosity contrasts $\lambda=1$ and $\lambda=10$.  Gravity is absent, $G=0$. The comparison at
$t=1$ shows the same qualitative cavity circulation and interface-deformation trend as the reference results (but no claim of quantitative agreement or full-model validation is made).}

\begin{table}[h]
\centering
\begin{tabularx}{\textwidth}{@{}l l X@{}}
\hline
\textbf{Parameter} & \textbf{Symbol} & \textbf{Value}\\
\hline
Domain & $\Omega$ & $(0,1)^2$\\
Benchmark phase field & $q=0$ & $q_0=\pm1$\\
Initial interface & $q_0$ & $q_0=1$ for $y<0.5$, $q_0=-1$ for $y\ge0.5$\\
Moving-boundary amplitude & $\gamma$ & $16$\\
Mass P\'eclet number & $Pe$ & $10^4$\\
Prescribed temperature & $\Theta$ & $0.3$\\
FreeFEM interfacial parameter & $Ch_{\rm code}$ & $10^{-3}$\\
Gradient coefficient & $\varepsilon_{\rm cav}^2$ & $5\times10^{-4}$\\
Effective interfacial parameter & $\varepsilon_{\rm cav}$ & $2.236\times10^{-2}$\\
Time step & $\Delta t$ & $0.01$\\
Final time & $t_{\rm fin}$ & $5$\\
Mesh & & $72\times72$ for both viscosity contrasts\\
Viscosity contrast & $\lambda=\eta(q=1)/\eta(q=-1)$ & $1$ and $10$\\
Gravity parameter & $G$ & $0$\\
\hline
\end{tabularx}
\caption{Parameters used in Experiment~1.}
\label{tab:exp1}
\end{table}

\begin{figure}[h!]
  \centering
  \includegraphics[width=0.5\textwidth]{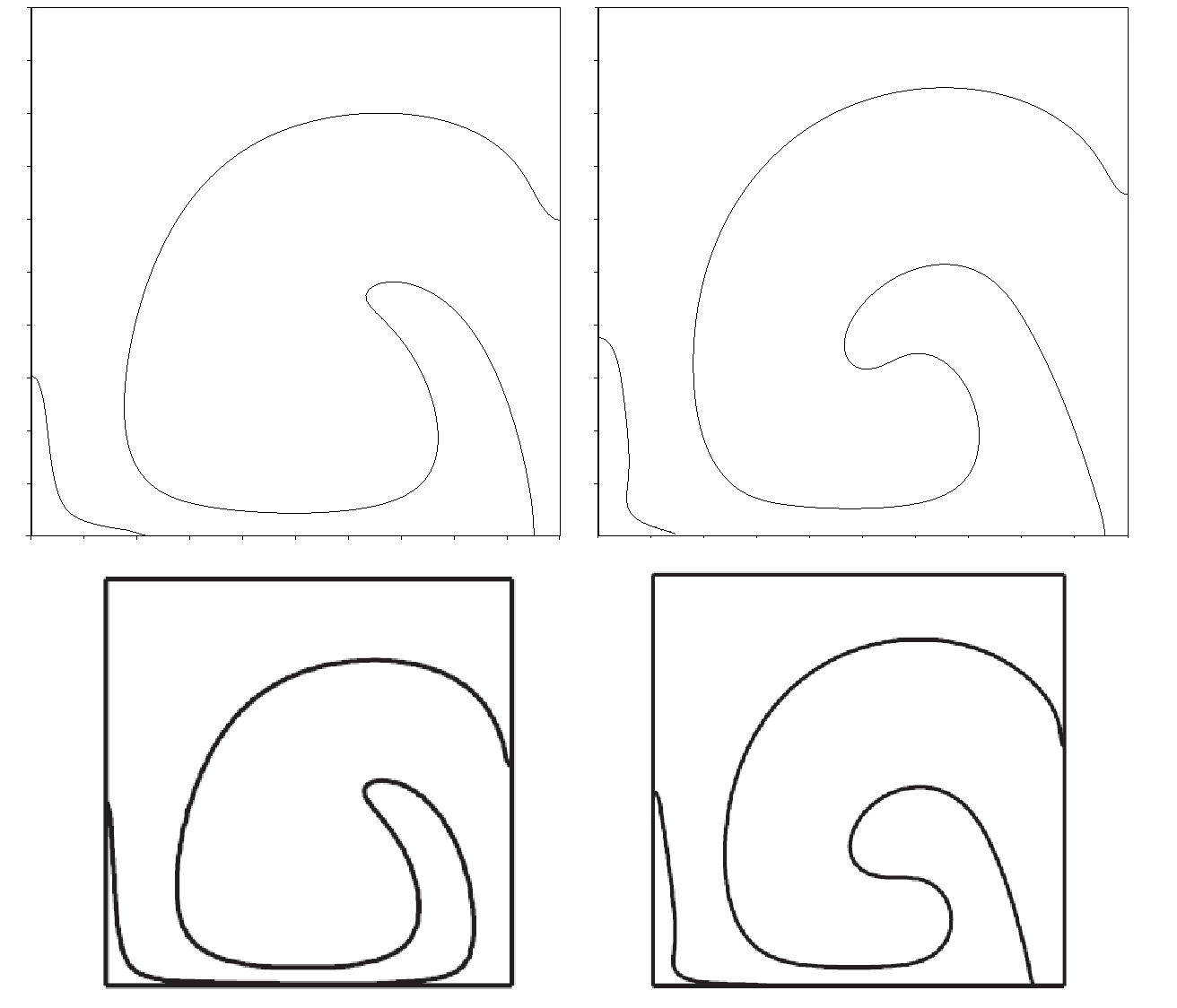}
  \caption[Comparison of interface profiles at $t=1$ for $\lambda=1$ (left) and \(\lambda=10\) (right).]
  {Comparison of the interface profiles $q=0$ at $t=1$ for $\lambda=1$ (left) and $\lambda=10$ (right).  The bottom row reproduces the corresponding results of Khatavkar et al.\ from~\cite{Khat}.}
  \label{together}
\end{figure}

\subsection{Experiment 2: Isothermal phase separation above and below the transition temperature}

In this experiment we suppress advection by setting \(\mathbf u\equiv0\) and do not solve the heat equation.  The evolution is therefore governed solely by the Cahn--Hilliard dynamics with a prescribed constant temperature.  This allows us to isolate the effect of the temperature-dependent bulk potential on phase separation.

The domain is $\Omega=(0,1)^2$, discretized by a structured
$100\times100$ mesh.  We set $\mathbf u=0$, impose homogeneous no-flux
conditions for $c$ and $\mu$, and initialize the nodal phase field by
\eqref{eq:bernoulli_initial_data}. The two prescribed dimensionless temperatures are
\[
\Theta=1.5 \qquad\text{and}\qquad \Theta=0.3,
\]
with
\[
Pe=1000,\qquad Ch=0.01,\qquad \Delta t=0.01,
\qquad t_{\rm fin}=2.
\]
For $\Theta>1$, the quadratic coefficient in the free energy is positive and the bulk potential is single-well, with its unconstrained minimum at $c=0$. Because the no-flux Cahn--Hilliard dynamics conserve the spatial mean, the solution instead relaxes toward a nearly uniform state with the realized average concentration. For $\Theta<1$, the bulk potential is double-well and the solution undergoes phase separation followed by coarsening, as shown in Figures~\ref{exp2_2}--\ref{hsv}.

\begin{table}[h!]
\centering
\begin{tabularx}{\textwidth}{@{}l l X@{}}
\hline
\textbf{Parameter} & \textbf{Symbol} & \textbf{Value}\\
\hline
Domain & $\Omega$ & $(0,1)^2$\\
Mesh & & $100\times100$\\
Boundary conditions & & $\partial_n c=\partial_n\mu=0$ on $\partial\Omega$\\
Initial phase field & $c_h^0$ & $-0.5$ (probability $0.7$), $+0.5$ (probability $0.3$)\\
Random seed & & $20260803$\\
Temperature, case A & $\Theta$ & $1.5$\\
Temperature, case B & $\Theta$ & $0.3$\\
Mass P\'eclet number & $Pe$ & $1000$\\
Cahn number & $Ch$ & $0.01$\\
Time step & $\Delta t$ & $0.01$\\
Final time & $t_{\rm fin}$ & $2$\\
\hline
\end{tabularx}
\caption{Parameters used in Experiment~2.  Temperature is dimensionless, so $\Theta=1$ is the transition value.}
\label{tab:exp2}
\end{table}

\begin{figure}[hbt!]
\begin{subfigure}{.33\textwidth}
  \centering
  \includegraphics[width=.9\linewidth]{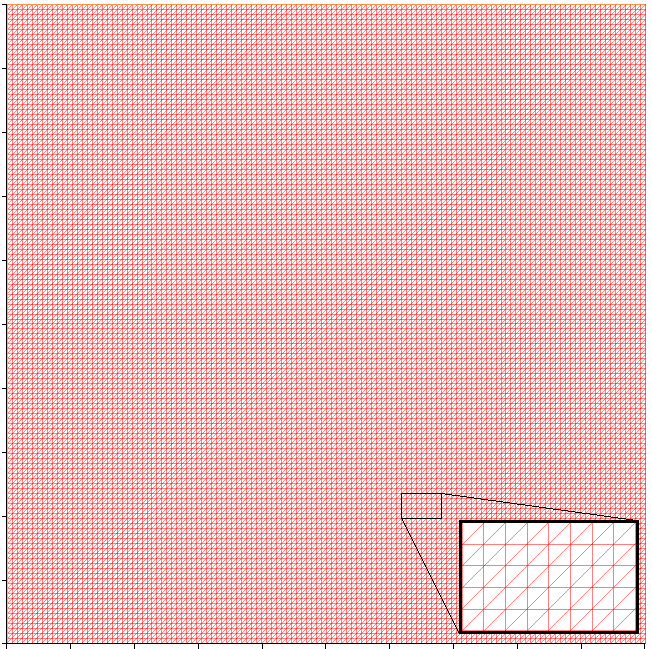}
  \caption{\(100\times100\) mesh}
\end{subfigure}%
\begin{subfigure}{.33\textwidth}
  \centering
  \includegraphics[width=.9\linewidth]{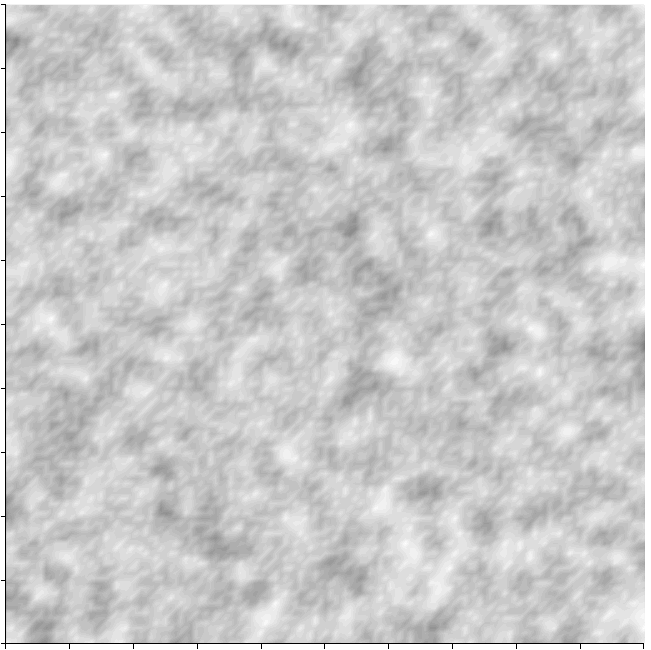}
  \caption{\(t=0\)}
\end{subfigure}
\begin{subfigure}{.33\textwidth}
  \centering
  \includegraphics[width=.9\linewidth]{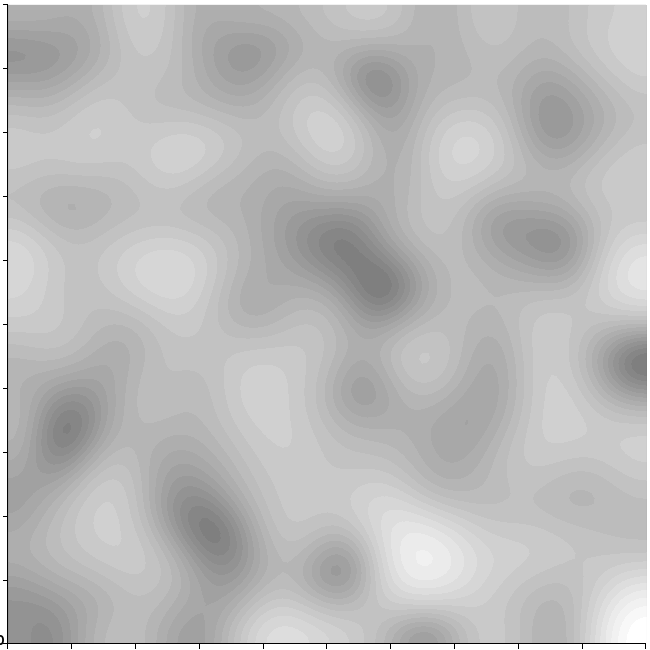}
  \caption{\(t=0.5\)}
\end{subfigure}
\begin{subfigure}{.33\textwidth}
  \centering
  \includegraphics[width=.9\linewidth]{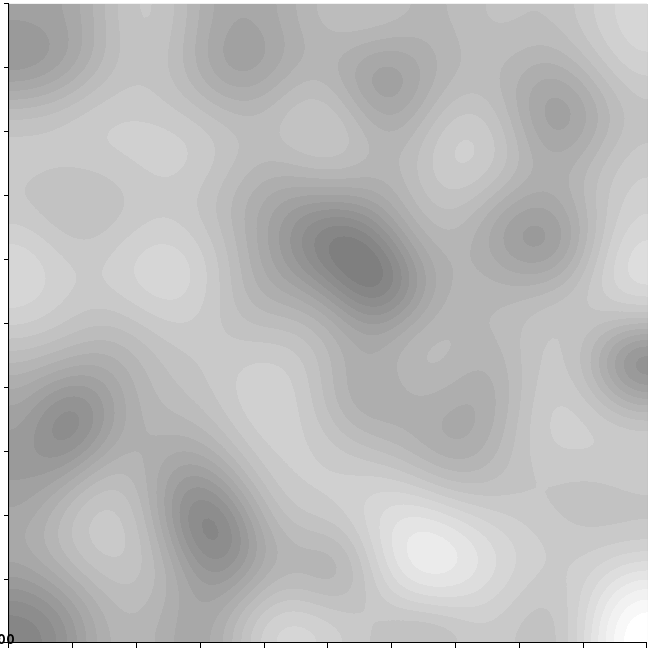}
  \caption{\(t=1\)}
\end{subfigure}%
\begin{subfigure}{.33\textwidth}
  \centering
  \includegraphics[width=.9\linewidth]{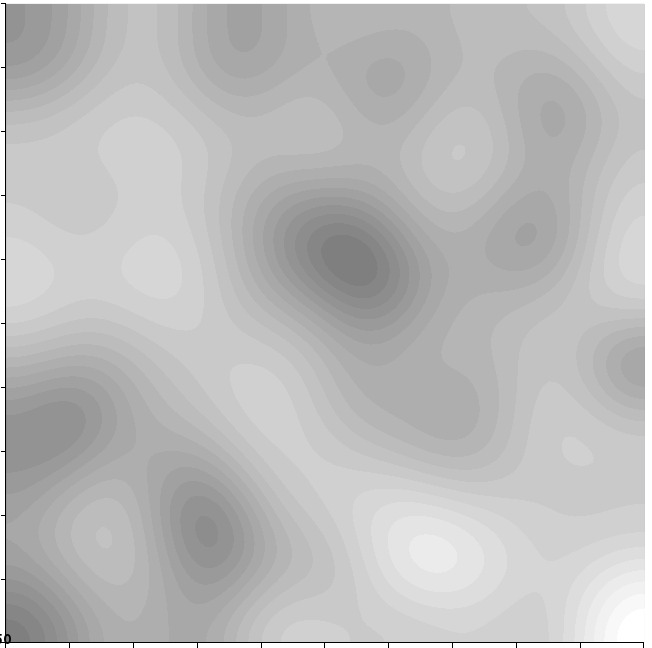}
  \caption{\(t=1.5\)}
\end{subfigure}
\begin{subfigure}{.33\textwidth}
  \centering
  \includegraphics[width=.9\linewidth]{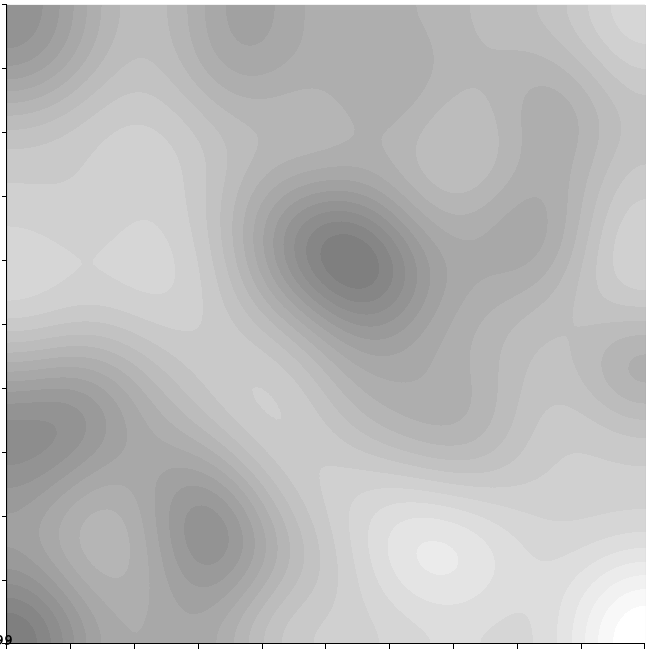}
  \caption{\(t=2\)}
\end{subfigure}
\caption{Snapshots of the phase field for $\Theta=1.5$ (single-well regime). The Bernoulli initial field relaxes rapidly toward a nearly spatially uniform state.}
\label{exp2_1}
\end{figure}

\begin{figure}[h]
\begin{subfigure}{.33\textwidth}
  \centering
  \includegraphics[width=.9\linewidth]{ma_mesh.png}
  \caption{\(100\times100\) mesh}
\end{subfigure}%
\begin{subfigure}{.33\textwidth}
  \centering
  \includegraphics[width=.9\linewidth]{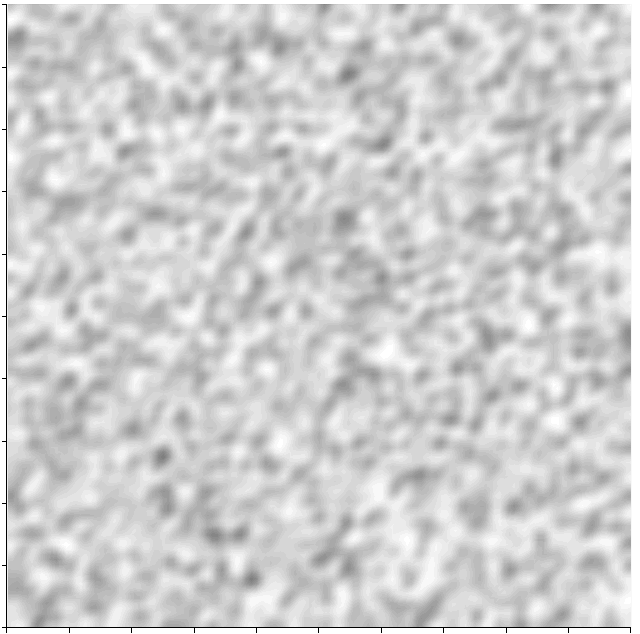}
  \caption{\(t=0\)}
\end{subfigure}
\begin{subfigure}{.33\textwidth}
  \centering
  \includegraphics[width=.9\linewidth]{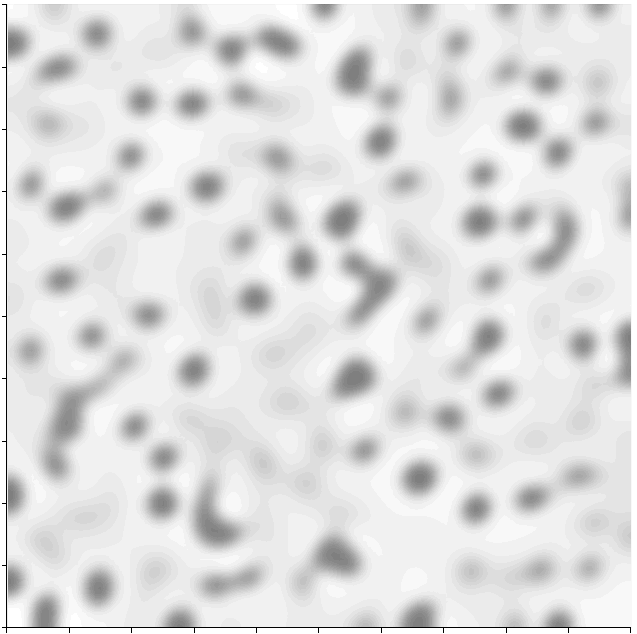}
  \caption{\(t=0.5\)}
\end{subfigure}
\begin{subfigure}{.33\textwidth}
  \centering
  \includegraphics[width=.9\linewidth]{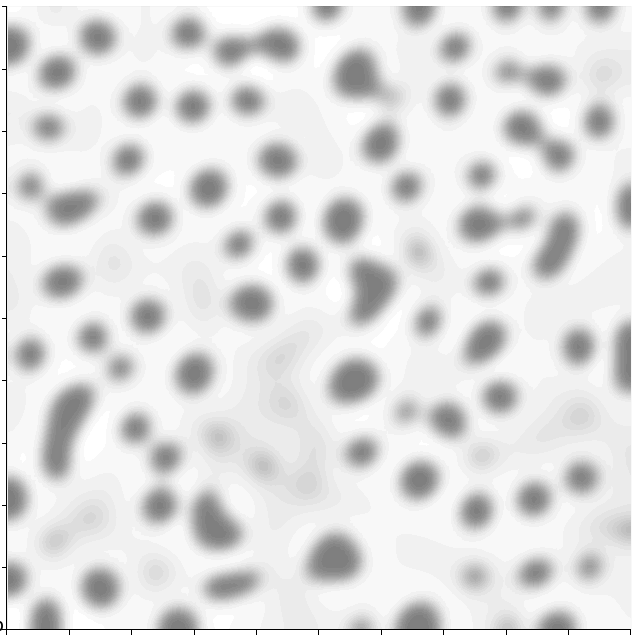}
  \caption{\(t=1\)}
\end{subfigure}%
\begin{subfigure}{.33\textwidth}
  \centering
  \includegraphics[width=.9\linewidth]{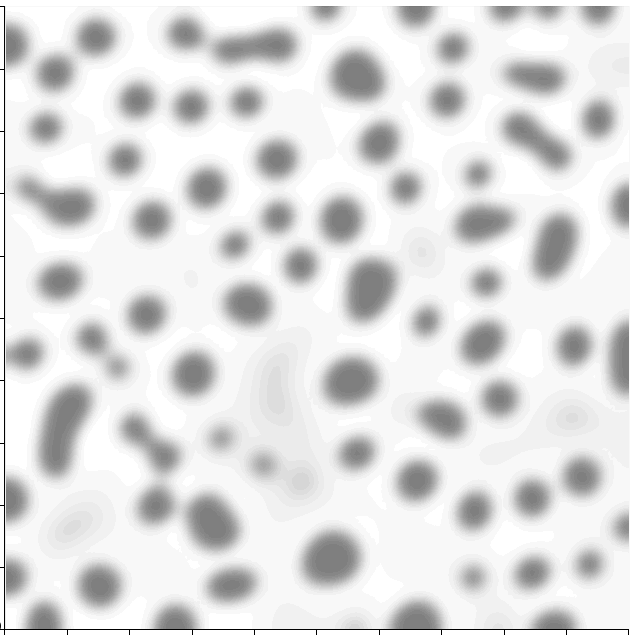}
  \caption{\(t=1.5\)}
\end{subfigure}
\begin{subfigure}{.33\textwidth}
  \centering
  \includegraphics[width=.9\linewidth]{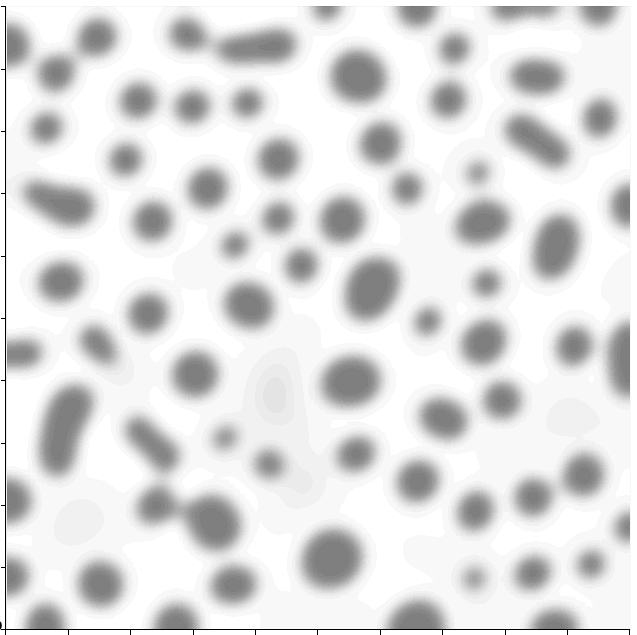}
  \caption{\(t=2\)}
\end{subfigure}
\caption{Snapshots of the phase field for $\Theta=0.3$ (double-well regime). Phase separation sets in rapidly and is followed by coarsening.}
\label{exp2_2}
\end{figure}

\begin{figure}[h]
  \centering
  \includegraphics[width=0.4\textwidth]{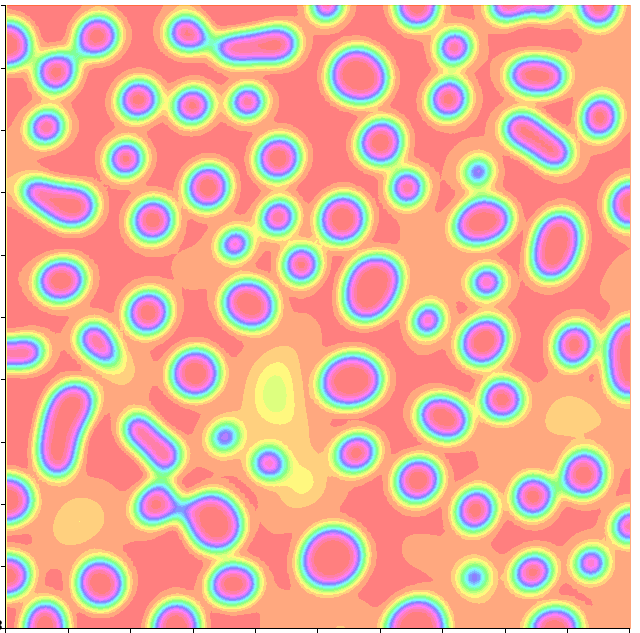}
    \caption{Phase distribution at $t=2$ for $\Theta=0.3$, visualized with an HSV colormap to emphasize the diffuse interfaces.}
  \label{hsv}
\end{figure}

\subsection{Experiment 3: Gravity-driven phase separation with variable density and viscosity}

The third experiment activates the coupling between the Cahn--Hilliard equation
and the quasi-static Stokes system in the presence of gravity.  The goal is to
illustrate the effect of buoyancy when the two phases have different densities
and viscosities.

The domain is $(0,1)^2$ with a $100\times100$ mesh and boundary labels
$\Gamma_1=\{y=0\}$, $\Gamma_2=\{x=1\}$,
$\Gamma_3=\{y=1\}$, and $\Gamma_4=\{x=0\}$.  We impose
$\mathbf u=0$ on $\Gamma_1$ and the natural traction condition on
$\Gamma_2\cup\Gamma_3\cup\Gamma_4$.  The phase field is initialized by
\eqref{eq:bernoulli_initial_data}, the temperature is fixed at $\Theta=0.3$,
and the heat equation is not solved.

\revfive{The gravity term used in the simulations is
\[
G=1,\qquad \widehat{\mathbf g}=(0,-1).
\]}

The coefficient laws used in this experiment are
\[
\rho_{\rm num}(c)=(0.5+c)+\lambda_\rho(0.5-c),\qquad
\eta_{\rm num}(c)=(0.5+c)+\lambda_\eta(0.5-c).
\]
Consequently,
$\rho_{\rm num}(0.5)=\eta_{\rm num}(0.5)=1$,
$\rho_{\rm num}(-0.5)=\lambda_\rho$, and
$\eta_{\rm num}(-0.5)=\lambda_\eta$.  Thus the association between the sign
of $c$ and the material phases is reversed relative to
\eqref{DefRhoc}--\eqref{DefEtac}; all phase-specific interpretations in this
experiment use the numerical labeling above.  With
$\lambda_\rho=0.0009$ and $\lambda_\eta=0.08$, the phase
$c\approx0.5$ is the denser and more viscous phase. 
The resulting figures are interpreted only as a qualitative sedimentation
trend under this particular parameter set.
All parameters are summarized in Table~\ref{tab:exp3}.

Snapshots of the phase field are shown in Figure~\ref{exp3}. Starting from the heterogeneous Bernoulli field, regions associated with the denser numerical phase migrate downward under the combined action of diffuse-interface relaxation and gravity and accumulate near the bottom boundary. The result is consistent with the expected qualitative sedimentation trend for this parameter set and these mixed no-slip/traction-free boundary conditions.

\begin{table}[h!]
\centering
\begin{tabularx}{\textwidth}{@{}l l X@{}}
\hline
\textbf{Parameter} & \textbf{Symbol} & \textbf{Value}\\
\hline
Domain & $\Omega$ & $(0,1)^2$\\
Mesh & & $100\times100$\\
Phase boundary conditions & & $\partial_n c=\partial_n\mu=0$ on $\partial\Omega$\\
Velocity boundary conditions & & $\mathbf u=0$ on $\Gamma_1$; traction-free on $\Gamma_2\cup\Gamma_3\cup\Gamma_4$\\
Initial phase field & $c_h^0$ & $-0.5$ (probability $0.7$), $+0.5$ (probability $0.3$)\\
Random seed & & $20260803$\\
Prescribed temperature & $\Theta$ & $0.3$\\
Mass P\'eclet number & $Pe$ & $1000$\\
Cahn number & $Ch$ & $0.01$\\
Gravity & $G\widehat{\mathbf g}$ & $(0,-1)$, i.e. $G=1$\\
Density contrast coefficient & $\lambda_\rho$ & $0.0009$\\
Viscosity contrast coefficient & $\lambda_\eta$ & $0.08$\\
Time step & $\Delta t$ & $0.01$\\
Final time & $t_{\rm fin}$ & $2$\\
\hline
\end{tabularx}
\caption{Parameters used in Experiment~3.}
\label{tab:exp3}
\end{table}

\begin{figure}[h]
\begin{subfigure}{.33\textwidth}
  \centering
  \includegraphics[width=.9\linewidth]{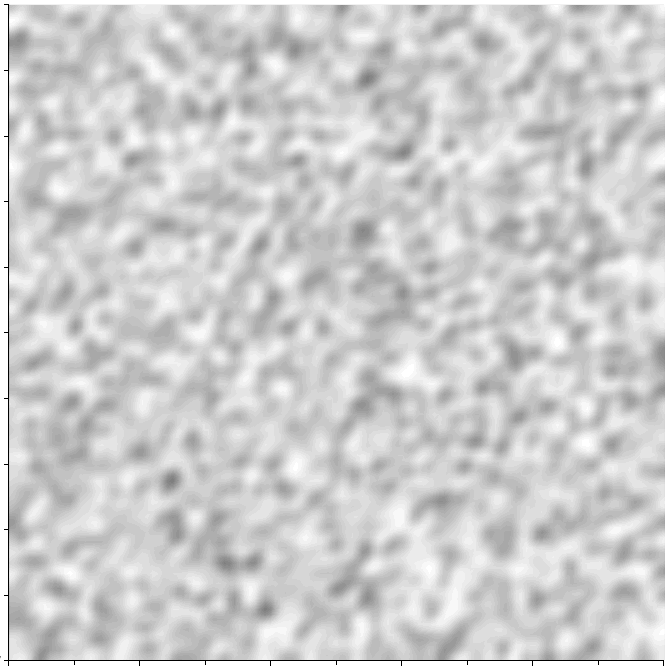}
  \caption{\(t=0\)}
\end{subfigure}%
\begin{subfigure}{.33\textwidth}
  \centering
  \includegraphics[width=.9\linewidth]{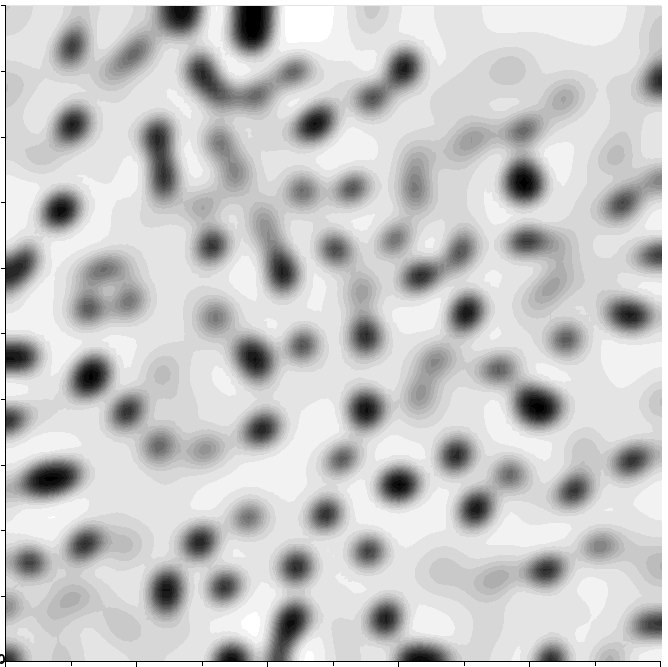}
  \caption{\(t=0.5\)}
\end{subfigure}
\begin{subfigure}{.33\textwidth}
  \centering
  \includegraphics[width=.9\linewidth]{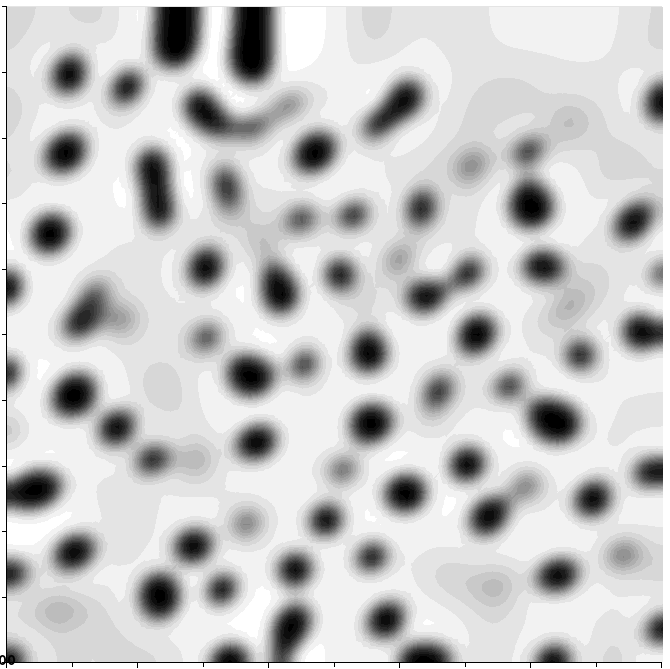}
  \caption{\(t=1\)}
\end{subfigure}
\begin{subfigure}{.33\textwidth}
  \centering
  \includegraphics[width=.9\linewidth]{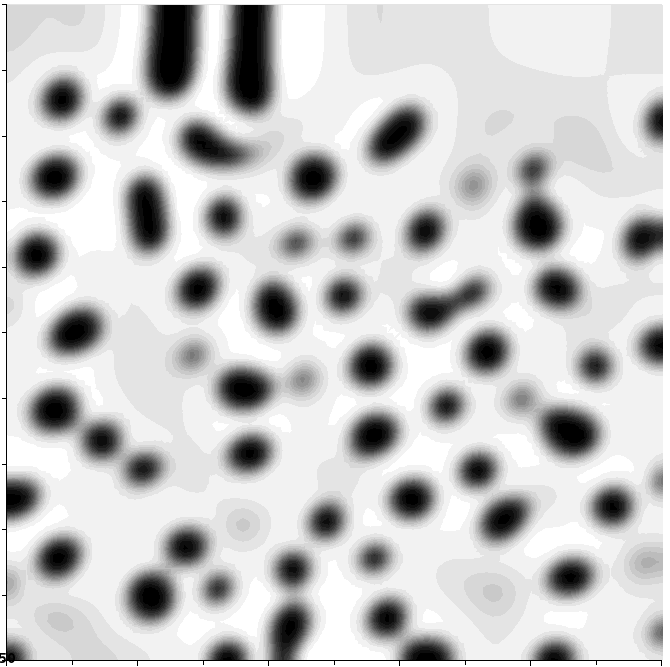}
  \caption{\(t=1.5\)}
\end{subfigure}%
\begin{subfigure}{.33\textwidth}
  \centering
  \includegraphics[width=.9\linewidth]{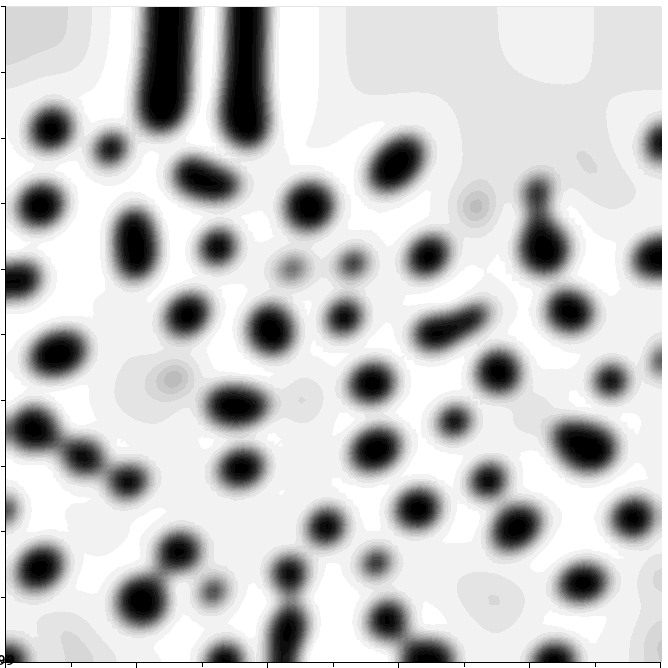}
  \caption{\(t=2\)}
\end{subfigure}
\caption{Phase-field snapshots for Experiment~3.  Dense domains migrate toward the lower boundary under the action of gravity and merge into a heavier layer.}
\label{exp3}
\end{figure}

\subsection{Experiment 4: Temperature-driven phase-field response near a heated or cooled boundary}

In the final experiment we study the response of the coupled system to an
imposed thermal gradient and a dynamically evolving temperature field.  The domain is $\Omega=(0,2)\times(0,1)$, discretized by a
$128\times64$ mesh.  The boundary labels are
$\Gamma_1=\{y=0\}$, $\Gamma_2=\{x=2\}$,
$\Gamma_3=\{y=1\}$, and $\Gamma_4=\{x=0\}$.  We impose
$\mathbf u=0$ on $\Gamma_1$ and the natural traction condition on
$\Gamma_2\cup\Gamma_3\cup\Gamma_4$.  For temperature,
$\Gamma_\Theta=\Gamma_4$ and
$\Gamma_{\Theta_N}=\Gamma_1\cup\Gamma_2\cup\Gamma_3$.

The parameters are
\[
Pe=1000,\qquad Ch=0.01,\qquad Pe_\Theta=2.5,
\qquad G=1,\qquad \widehat{\mathbf g}=(0,-1),
\qquad \Delta t=0.01,\qquad t_{\rm fin}=2.
\]
The temperature update is the pure-diffusion step
\eqref{Heat_discrete_revised}.  The initial phase field is given by
\eqref{eq:bernoulli_initial_data}.

We use the following exploratory,
piecewise-temperature-dependent contrast functions:
\[
\lambda_\rho(\Theta)=
\begin{cases}
0.0009,&\Theta<1,\\
1.6,&\Theta\ge1,
\end{cases}
\qquad
\lambda_\eta(\Theta)=
\begin{cases}
0.08,&\Theta<1,\\
1.67,&\Theta\ge1.
\end{cases}
\]
The numerical coefficient laws are
\[
\rho_{\rm num}(c,\Theta)
=(0.5+c)+\lambda_\rho(\Theta)(0.5-c),
\qquad
\eta_{\rm num}(c,\Theta)
=(0.5+c)+\lambda_\eta(\Theta)(0.5-c).
\]
This piecewise constitutive dependence is not part of the analytical model in
Sections~\ref{sec:model}--\ref{sec:analysis}.

In the heated-boundary case, $\Theta_0=0.3$ and
$\Theta|_{\Gamma_4}=1.5$.  In the cooled-boundary case,
$\Theta_0=1.5$ and $\Theta|_{\Gamma_4}=0.3$.  The remaining thermal boundaries
are insulated. \revfive{The figures illustrate the qualitative response of the phase field to these imposed thermal profiles under the above 
constitutive choices.}

The temperature affects this experiment in two ways: it changes the Landau coefficient $\beta(\Theta)$ and it changes the density and viscosity contrasts through the piecewise functions above. Because both contrast functions jump from values below one to values above one
at $\Theta=1$, the sign of $c$ associated with the denser and more viscous phase reverses at the temperature threshold. Gravity remains active with $G=1$.  Consequently, this experiment does not isolate the effect of $\beta(\Theta)$ alone.  It is an exploratory constitutive extension of the model analyzed in Sections~\ref{sec:model}--\ref{sec:analysis}, and the figures are interpreted as a combined thermal, material-contrast, and buoyancy response. The thermal equation is diffusion-only, so the temperature drives the phase and flow fields but is not advected by the computed velocity. Although Dirichlet temperature data are imposed on $\Gamma_4$, this boundary
is mechanically traction-free and should therefore be described as a
thermally controlled boundary rather than a no-slip wall.

We consider two thermal boundary configurations.

\paragraph{Case 1 (heated boundary)}
The interior temperature degrees of freedom are initialized at
$\Theta_0=0.3$, while the Dirichlet value $\Theta_D=1.5$ is imposed on
$\Gamma_4$ from the first heat-diffusion update. The remaining thermal
boundaries are insulated.

The resulting phase-field evolution is shown in Figure~\ref{exp4_1}, and the temperature distribution at $t=1$ in Figure~\ref{exp4_temp1}.  As heat diffuses from $\Gamma_4$, the local Landau potential changes from a double-well to a single-well form when $\Theta$ crosses $1$.  At the same time, the density and viscosity contrast functions switch to their high-temperature values. The observed near-boundary homogenization is consistent with this combined constitutive response and is not attributed exclusively to $\beta(\Theta)$.

\begin{figure}[hbt!]
\begin{subfigure}{.5\textwidth}
  \centering
  \includegraphics[width=.9\linewidth]{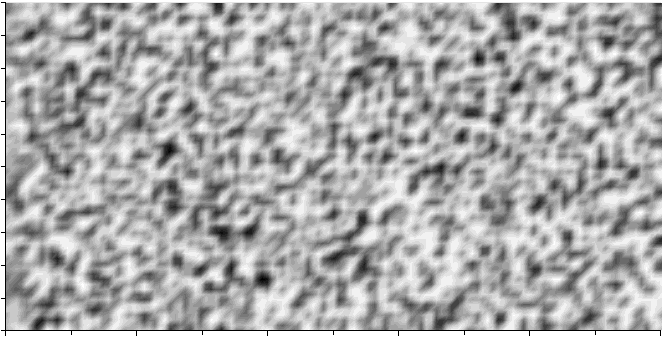}
  \caption{\(t=0\)}
\end{subfigure}
\begin{subfigure}{.5\textwidth}
  \centering
  \includegraphics[width=.9\linewidth]{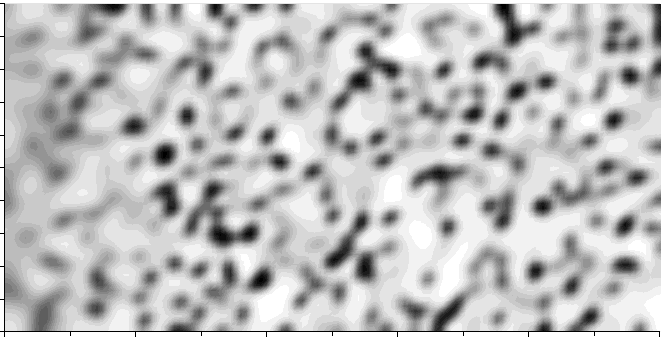}
  \caption{\(t=0.2\)}
\end{subfigure}
\begin{subfigure}{.5\textwidth}
  \centering
  \includegraphics[width=.9\linewidth]{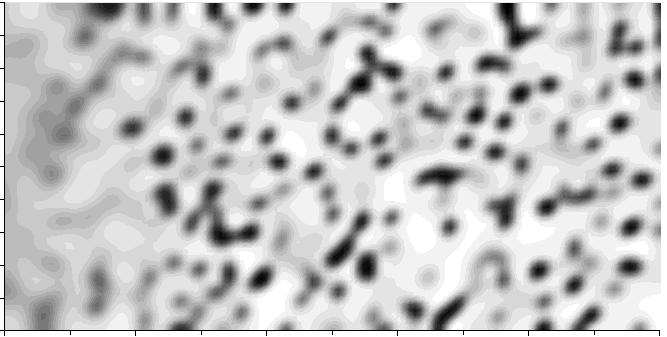}
  \caption{\(t=0.4\)}
\end{subfigure}
\begin{subfigure}{.5\textwidth}
  \centering
  \includegraphics[width=.9\linewidth]{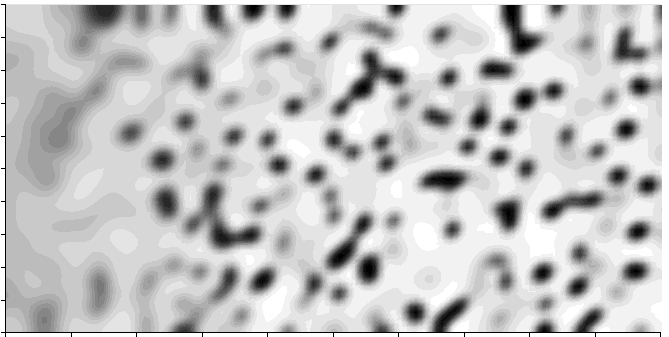}
  \caption{\(t=0.6\)}
\end{subfigure}
\begin{subfigure}{.5\textwidth}
  \centering
  \includegraphics[width=.9\linewidth]{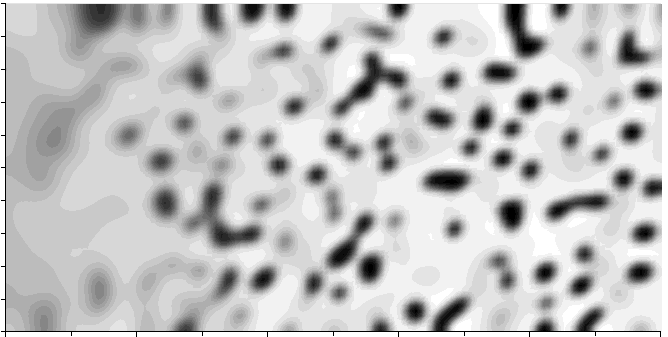}
  \caption{\(t=0.8\)}
\end{subfigure}
\begin{subfigure}{.5\textwidth}
  \centering
  \includegraphics[width=.9\linewidth]{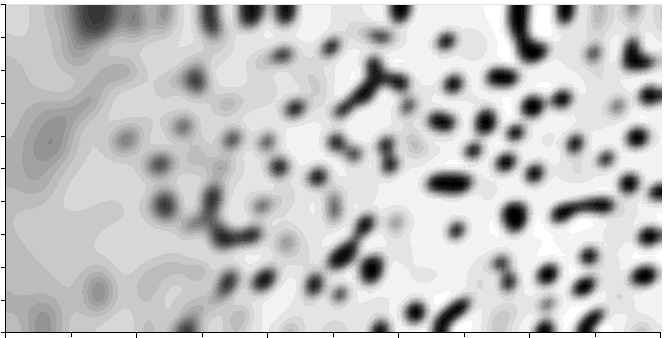}
  \caption{\(t=1\)}
\end{subfigure}
\caption{Phase-field snapshots for the heated-boundary configuration:
$\Theta_0=0.3$ in the interior and
$\Theta|_{\Gamma_4}=1.5\,\Theta_S$.  Heating near the left boundary suppresses
phase separation locally.}
\label{exp4_1}
\end{figure}

\begin{figure}[h!]
  \centering
  \includegraphics[width=0.6\textwidth]{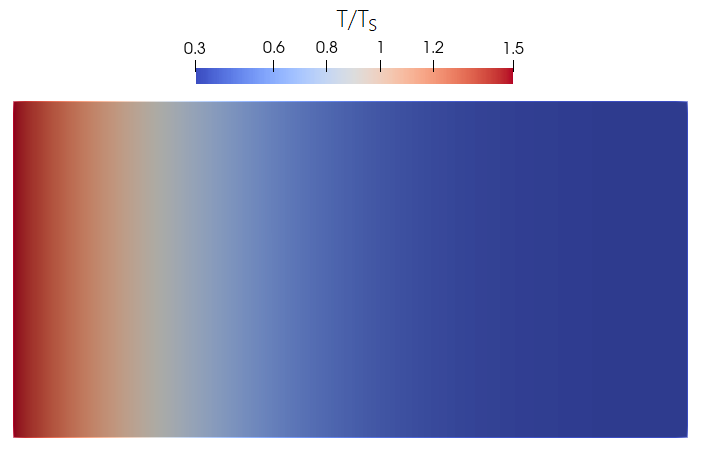}
  \caption{Temperature distribution at \(t=1\) for the heated-boundary configuration.}
  \label{exp4_temp1}
\end{figure}

\paragraph{Case 2 (cooled boundary)}
The interior temperature degrees of freedom are initialized at
$\Theta_0=1.5$, while the Dirichlet value $\Theta_D=0.3$ is imposed on
$\Gamma_4$ from the first heat-diffusion update. The remaining thermal
boundaries are insulated.
All other parameters are unchanged. The evolution of the phase field is shown in Figure~\ref{exp4_2}, and the temperature distribution at $t=2$ in Figure~\ref{exp4_temp2}.  Cooling lowers $\beta(\Theta)$ near $\Gamma_4$ and restores a local double-well bulk potential, while the material contrast functions simultaneously switch to their low-temperature values. Phase-separated domains appear near the cooled
boundary. This is interpreted as the combined response of the potential and the temperature-dependent material coefficients.

\begin{table}[h]
\centering
\begin{tabularx}{\textwidth}{@{}l l X@{}}
\hline
\textbf{Parameter} & \textbf{Symbol} & \textbf{Value}\\
\hline
Domain & $\Omega$ & $(0,2)\times(0,1)$\\
Mesh & & $128\times64$\\
Phase boundary conditions & & $\partial_n c=\partial_n\mu=0$ on $\partial\Omega$\\
Velocity boundary conditions & & $\mathbf u=0$ on $\Gamma_1$; traction-free on $\Gamma_2\cup\Gamma_3\cup\Gamma_4$\\
Thermal boundary conditions & & Dirichlet on $\Gamma_4$; insulated on $\Gamma_1\cup\Gamma_2\cup\Gamma_3$\\
Initial phase field & $c_h^0$ & $-0.5$ (probability $0.7$), $+0.5$ (probability $0.3$)\\
Random seed & & $20260803$\\
Initial/boundary temperature, heated case & $(\Theta_0,\Theta_D)$ & $(0.3,1.5)$\\
Initial/boundary temperature, cooled case & $(\Theta_0,\Theta_D)$ & $(1.5,0.3)$\\
Mass P\'eclet number & $Pe$ & $1000$\\
Thermal P\'eclet number & $Pe_\Theta$ & $2.5$\\
Cahn number & $Ch$ & $0.01$\\
Gravity & $G\widehat{\mathbf g}$ & $(0,-1)$, i.e. $G=1$\\
Density contrast & $\lambda_\rho(\Theta)$ & $0.0009$ for $\Theta<1$; $1.6$ for $\Theta\ge1$\\
Viscosity contrast & $\lambda_\eta(\Theta)$ & $0.08$ for $\Theta<1$; $1.67$ for $\Theta\ge1$\\
Time step & $\Delta t$ & $0.01$\\
Final time & $t_{\rm fin}$ & $2$\\
\hline
\end{tabularx}
\caption{Parameters used in Experiment~4.  Temperature is dimensionless and
the transition value is $\Theta=1$.}
\label{tab:exp4}
\end{table}

\begin{figure}[hbt!]
\begin{subfigure}{.5\textwidth}
  \centering
  \includegraphics[width=.9\linewidth]{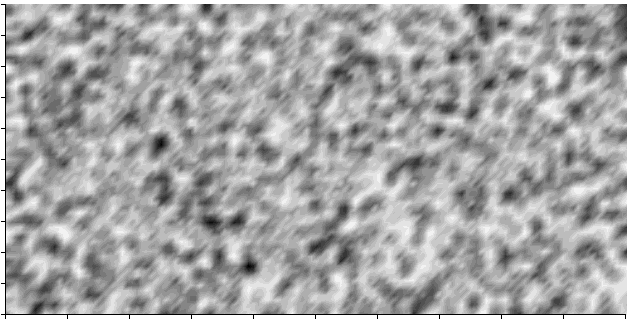}
  \caption{\(t=0\)}
\end{subfigure}
\begin{subfigure}{.5\textwidth}
  \centering
  \includegraphics[width=.9\linewidth]{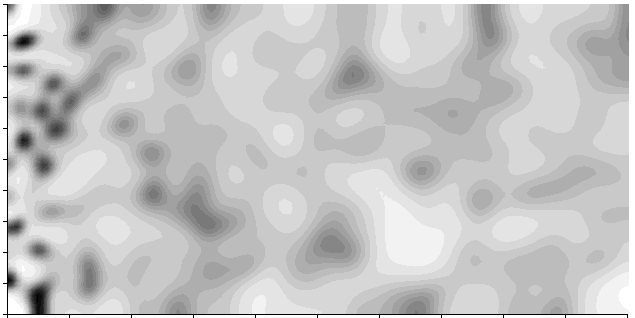}
  \caption{\(t=0.5\)}
\end{subfigure}
\begin{subfigure}{.5\textwidth}
  \centering
  \includegraphics[width=.9\linewidth]{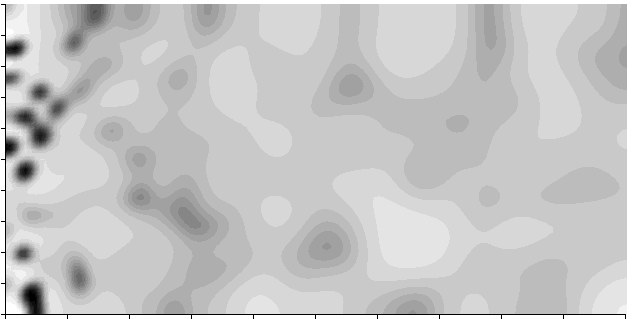}
  \caption{\(t=1\)}
\end{subfigure}
\begin{subfigure}{.5\textwidth}
  \centering
  \includegraphics[width=.9\linewidth]{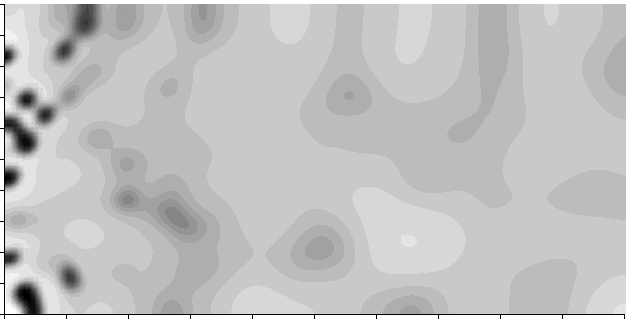}
  \caption{\(t=1.5\)}
\end{subfigure}
\begin{subfigure}{.5\textwidth}
  \centering
  \includegraphics[width=.9\linewidth]{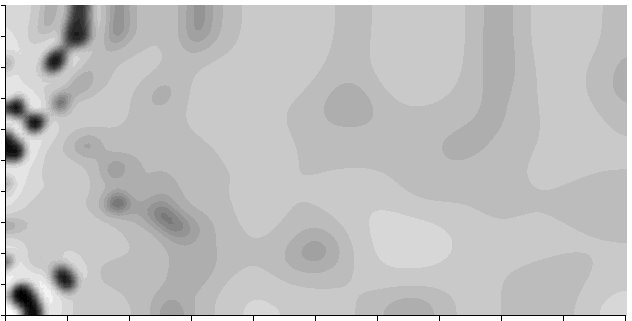}
  \caption{\(t=2\)}
\end{subfigure}
\caption{Phase-field snapshots for the cooled-boundary configuration:
$\Theta_0=1.5\,\Theta_S$ in the interior and
$\Theta|_{\Gamma_4}=0.3$.  Cooling near the left boundary triggers local phase separation and subsequent growth of phase-separated domains.}
\label{exp4_2}
\end{figure}

\begin{figure}[!h]
  \centering
  \includegraphics[width=0.6\textwidth]{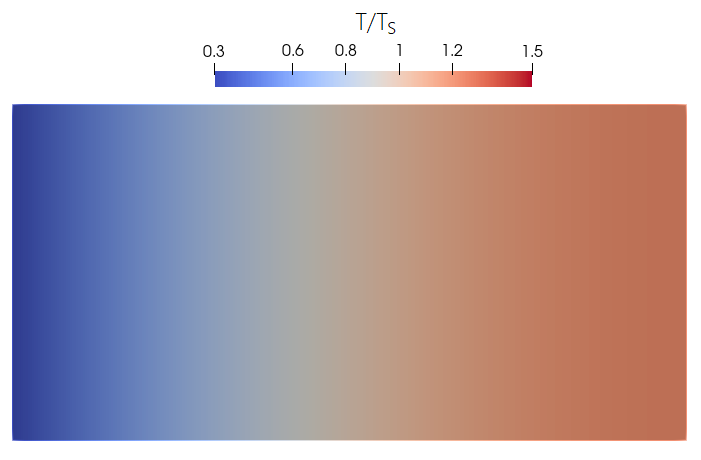}
  \caption{Temperature distribution at \(t=2\) for the cooled-boundary configuration.}
  \label{exp4_temp2}
\end{figure}

\section{Conclusion}

In this work, we introduced a temperature-coupled Cahn--Hilliard--Stokes--Heat model for thermally driven phase separation motivated by condensation phenomena in battery-inspired geometries. The model combines a temperature-dependent Landau free energy with variable-density Stokes flow and heat transport, providing a diffuse-interface framework for studying the interaction between phase separation, fluid motion, and thermal effects.

On the analytical side, we established the local-in-time existence of weak solutions for a regularized auxiliary formulation of the coupled system. The proof is based on the analysis of the heat, Cahn--Hilliard, and Stokes subproblems, the construction of the corresponding solution operators, and their coupling through a Schauder fixed-point argument.

To complement the analytical results, we presented parameter-specific
finite-element illustrations for the quasi-static numerical model and a
separate quantitative verification of the stand-alone isothermal diffusive Cahn--Hilliard update. The latter confirms machine-precision mass conservation and the expected refinement behaviour and shows monotone energy decay for the tested parameter set. \revfive{The coupled examples are qualitative and parameter-specific, no global
discrete energy law or comprehensive multi-parameter robustness claim is made for the non-isothermal numerical scheme.}

Several questions remain open. A particularly important direction for future work is to establish a well-posedness theory for the quasi-static model considered in the numerical simulations. Other natural extensions include the study of global-in-time solutions, long-time dynamics, and more general constitutive laws and coupling mechanisms motivated by condensation phenomena.

\section*{Declaration of competing interest}
The authors declare that they have no known competing financial interests or personal relationships that could have appeared to influence the work reported in this paper.

\section*{Acknowledgement}
This work was partially supported by the Croatian Science Foundation under project IP-2022-10-7261 (ADESO) (A.N.) and IP-2022-10-2962 (B.M. and M.D.). B.M. was supported by the European Union – NextGenerationEU through the National Recovery and
Resilience Plan 2021-2026. Institutional grant of University of Zagreb Faculty of Science IK IA 1.1.3. Impact4Math. 
A.N. permanent address is the University of Zagreb.
The authors thank the anonymous reviewer for the careful reading of the manuscript and for the constructive comments and suggestions, which helped to improve the paper.

\section*{Declaration of generative AI and AI-assisted technologies in the manuscript preparation process}
During the preparation of this work, the authors used ChatGPT (OpenAI) to assist with language editing and to improve the clarity and readability of the manuscript. After using this tool, the authors reviewed and edited the content as needed and take full responsibility for the content of the published article.

\pagebreak

\end{document}